\documentclass[11pt]{article}

\usepackage{graphicx, color}

\usepackage{amssymb}

\usepackage{graphicx}
\usepackage{pgf,tikz}
\usetikzlibrary{arrows}

\usepackage{amsthm}

\usepackage{epstopdf}

\usepackage{amsmath}
\usepackage{amsfonts}

\def\phi{{\varphi}}

\newcommand{\CO}[2]{ \left\langle #1 , #2 \right\rangle}

\DeclareSymbolFont{AMSb}{U}{msb}{m}{n}
\DeclareMathSymbol{\N}{\mathbin}{AMSb}{"4E}
\DeclareMathSymbol{\Z}{\mathbin}{AMSb}{"5A}
\DeclareMathSymbol{\R}{\mathbin}{AMSb}{"52}
\DeclareMathSymbol{\Q}{\mathbin}{AMSb}{"51}
\DeclareMathSymbol{\I}{\mathbin}{AMSb}{"49}
\DeclareMathSymbol{\C}{\mathbin}{AMSb}{"43}

\def\be{\begin{equation}}
\def\ee{\end{equation}}
\def\ber{\begin{eqnarray}}
\def\eer{\end{eqnarray}}

\def\beq{\begin{equation}}
\def\eeq{\end{equation}}

\def\Z{{\mathbb{Z}}}

\def\IR{{\mathbb{R}}}

\newcommand{\E}[0]{ \varepsilon}

\newcommand{ \pOm}{\partial \Omega}

\begin{document}

\addtolength{\textheight}{0 cm} \addtolength{\hoffset}{0 cm}
\addtolength{\textwidth}{0 cm} \addtolength{\voffset}{0 cm}

\newenvironment{acknowledgement}{\noindent\textbf{Acknowledgement.}\em}{}

\setcounter{secnumdepth}{5}

 \newtheorem{proposition}{Proposition}[section]
 
\newtheorem{theorem}{Theorem}[section]

\newtheorem{lemma}[theorem]{Lemma}
\newtheorem{coro}[theorem]{Corollary}
\newtheorem{remark}[theorem]{Remark}
\newtheorem{extt}[theorem]{Example}
\newtheorem{claim}[theorem]{Claim}
\newtheorem{conj}[theorem]{Conjecture}
\newtheorem{definition}[theorem]{Definition}
\newtheorem{application}{Application}
\newtheorem{exam}{Example}[section]
\newtheorem{thm}{Theorem}[section]
\newtheorem{prop}{Proposition}[section]

\newtheorem*{assumption}{Assumptions on $a(x)$}

\newtheorem*{thm*}{Theorem A}

\newtheorem{corollary}[theorem]{Corollary}

\title{On supercritical elliptic problems: existence,  multiplicity of positive  and   symmetry breaking solutions\footnote{Both authors  are pleased to acknowledge the support of the  National Sciences and Engineering Research Council of Canada.}}
\author{Craig Cowan\footnote{University of Manitoba, Winnipeg Manitoba, Canada, craig.cowan@umanitoba.ca} \quad  Abbas Moameni\footnote{School of Mathematics and Statistics,
Carleton University,
Ottawa, Ontario, Canada,
momeni@math.carleton.ca} }

\date{}

\maketitle

\vspace{3mm}

\begin{abstract}
The main thrust of our current work is to exploit very specific characteristics of a given problem
in order to acquire improved compactness for supercritical problems and to prove existence of
new types of solutions. To this end, we shall develop a variational machinery in order to  construct a new type of  classical solutions for a large class of supercritical elliptic partial differential equations.\\
The issue of symmetry and symmetry breaking is challenging and  fundamental in mathematics and physics. Symmetry breaking is the source of many interesting phenomena namely phase transitions, instabilities, segregation, etc. 
  As a consequence of our results we shall establish  the existence of several symmetry breaking solutions when the underlying problem is fully symmetric.   Our methodology is variational, and  we are not seeking  non symmetric solutions which bifurcate from the symmetric one.  Instead,  we construct many new  positive solutions by   developing  a minimax principle for general semilinear elliptic problems restricted to a  given convex subset instead of the whole space. 
As a byproduct of our investigation,   several new Sobolev embeddings are established for  functions   having a mild  monotonicity on symmetric monotonic domains.   

\end{abstract}


\tableofcontents

\section{Introduction}

In this work we develop a variational machinery to examine a large class of  significant   supercritical  elliptic partial differential equations that arise naturally in various physical models: solitary waves in nonlinear Schr\"odinger  equations;   gravitational potential of a Newtonian self gravitating, spherically symmetric, polytropic fluid; and a model for a cluster of stars.   Our method is variational but as opposed to working on the natural energy space, which typically limits problems to subcritical and critical, we work on closed convex sets (not necessarily a linear subspace) which increases the available compactness. Working on symmetric functions can sometimes increase compactness,  which together with   the  principal of symmetric criticality provides an efficient tool to deal seemingly nocompact settings (see for instance \cite{Bart} and \cite{lion}).   Our method further increases compactness as we are restricting our problems on  an appropriate subsets which goes well beyond the symmetry   induced function spaces under certain compact groups.  The main thrust of our current work is to exploit  very specific characteristics of a given problem in order to acquire improved compactness for supercritical problems and to prove existence of   new types of  solutions. Our approach is broad enough to cover many elliptic partial differential equations, and in general,   one can employ   a combination of  symmetry, monotonicity, smallness in certain norms, convexity, and etc to name a few.\\

 Broadly speaking we are interested in obtaining positive classical solutions of equations of the form 
\begin{equation} \label{eq}
-\Delta u + V(x)  u= a(x) u^{p-1}  \mbox{ in } \Omega, 
\end{equation}
where $\Omega \subset \IR^N$ is either the full space or $\Omega$ is a bounded subset and in which case we add the boundary condition $ u=0$ on $ \pOm$.  
Our main interest will be in obtaining solutions in the case of $p>2$ and supercritical.    Generally $a$ will be a sufficiently smooth function which satisfies some symmetry and monotonicity assumptions and we point out any added compactness is not coming from $a$; which is a different phenomena from the H\'enon equation. The domains we will examine will be domains of double and triple revolution with some added monotonicity properties.   Additionally when the problems has extra symmetry we will obtain solutions which do not inherit the extra symmetry of the problem.  On radial domains we will obtain nonradial solutions which are not 
 foliated Schwarz symmetric. As a consequence of our approach, many new  multiplicity results are also obtained. \\
 

Since we address  existence and multiplicity issues for  numerous
 supercritical problems we  list the equations here for the   convenience  of the   readers.
Even though each of these problems poses  their own difficulty,  our variational machinery is able to give a unified approach.  
  \begin{itemize}  \item     In  Section \ref{section_double_annular} we examine the following problem 
     \begin{equation} 
\left\{\begin{array}{ll}
-\Delta u = a(x) u^{p-1} &  \mbox{ in } \Omega, \\
u= 0 &   \mbox{ on } \partial \Omega.
\end {array}\right.
\end{equation}   Here we consider annular domains which are radial and nonradial.   On the radial domains we obtain new type of positive  nonradial solutions for which do not  have the  foliated Schwarz symmetry.  In all cases we obtain results for a supercritical range of $p$.   The main result is  Theorem \ref{theorem-nonlinear-annular}.

\item In Section  \ref{henon_ball} we examine 
 \begin{equation} 
\left\{\begin{array}{ll}
-\Delta u = |x|^\alpha u^{p-1} &  \mbox{ in } B_1, \\
u= 0 &   \mbox{ on } \partial B_1,
\end {array}\right.
\end{equation} where $B_1$ is the unit ball in $\R^N.$ In  Theorem \ref{non-radial-henon}, we obtain several types of positive new  nonradial solutions on a range of supercritical $p$.

    \item In Section \ref{henon_full_section} we examine 
 \begin{equation} \label{henon_full_t}
 -\Delta u + u = |x|^\alpha u^{p-1} \qquad \mbox{ in } \IR^N= \IR^{n} \times \IR^n, 
 \end{equation} and we show there is a positive classical solution for 
 \[ \frac{2N+2 \alpha-4}{N-2}<p< \frac{2N+2\alpha}{N-2},\] and for large $\alpha$ we obtain a nonradial solution.  
  Theorem \ref{henon_full_thm} is devoted to this problem.

    \item In Section  \ref{singular_sect} we examine 
    \begin{equation} \label{eq_sing_pot_t}
\left\{\begin{array}{ll}
-\Delta u + \frac{u}{|x|^\alpha}  = u^{p-1} &  \mbox{ in } B_1, \\
u= 0 &   \mbox{ on } \partial B_1,
\end {array}\right.
\end{equation} where $ \alpha>2$ (note this  is in some sense supercritical).  For $ 2<p< \frac{2N+2\alpha-4}{N-2}$ we obtain a positive classical solution of (\ref{eq_sing_pot_t}) and for large $\alpha$ we obtain a nonradial solution.   Additionally the solution decays to zero at the origin quicker than any polynomial.  See Theorem  
 \ref{sing_theorem} for details.   Note here the zero order potential is playing a key role and we believe this is new phenomena.

\item In Section \ref{non-rad-sect}  we give an  approach to show  ground states of various problems on radial domains are nonradial.  Indeed, as stated  in Theorem \ref{nonradial},  the best constant in the well known hardy inequality corresponding to the underlying domain  plays a major role to address this challenging  affair.

\item  In Section \ref{dim+1}  we examine 
 \begin{equation} \label{eq_triple_t}
\left\{\begin{array}{ll}
-\Delta u  = a(x) u^{p-1} &  \mbox{ in } \Omega, \\
u= 0 &   \mbox{ on } \pOm,
\end {array}\right.
\end{equation}
 where $\Omega$ is a bounded domain in $\IR^N$ which is also a domain of triple revolution.  Under various assumptions we prove the existence and also some multiplicity results (for a range of supercritical $p$).  We have listed our contributions in Theorems  \ref{triple-thm},    \ref{annulus_triple_dep}, \ref{phi-dep} and 
  Corollary \ref{mult_ann}.      
 
  \end{itemize} 
  
  A crucial step in proving the above existence  results will be in obtaining improved Sobolev imbeddings for various classes of symmetric and monotonic functions.   The increases in compactness comes from two distinct properties of the closed convex sets we choose to work on,  namely the symmetry and also the monotonicity.    One should  note that these improved imbeddings also play a crucial role in the proof of the regularity of the solution. One added benefit of our approach is we can use energy levels directly to prove various results.

\subsection{Outline of the paper}

  We now give a brief outline of the paper.  In Section  \ref{variational_abst}  we develop our abstract variational machinery.   In Section \ref{double_dom}  we introduce domains of $m$ revolution and in particular we discuss domains of double revolution.   Then in Section \ref{section_double_annular} we consider elliptic problems on domains of double revolution which are also annular type domains.    The H\'enon equation on the unit ball is considered in Section \ref{henon_ball}.    In Section \ref{henon_full_section} we consider a H\'enon like equation, but with a zero order term, on the full space.    In Section \ref{singular_sect} we consider a singular potential problem.  Section \ref{non-rad-sect} is where we develop the needed machinery to obtain solutions on symmetric domains without the naturally expected symmetry.   Finally in Section \ref{dim+1}
 we consider domains of triple revolution.

 \subsection{Background}

Here we give some background on the the problem and for this we take $ a(x)=1$ and $ V(x)=0$ and hence we consider 
\begin{equation} \label{eq_a=1}
\left\{\begin{array}{ll}
-\Delta u =  u^{p-1} &  \mbox{ in } \Omega, \\
u>0    &   \mbox{ in } \Omega, \\
u= 0 &   \mbox{ on } \pOm.
\end {array}\right.
\end{equation}
We assume $ \Omega$ a bounded smooth domain in $ \IR^N$.   For $N \ge 3$ the critical  exponent   $2^*:=\frac{2N}{N-2}$ plays a crucial role and for $ 2<p<2^*$ a variational approach shows the existence of a smooth positive solution of (\ref{eq_a=1}). For $ p \ge 2^*$ there is no positive classical solution via the Pohozaev identity on star shaped domains, see \cite{poho}.   For general domains in the critical/supercritical case, $ p \ge 2^*$, the existence versus nonexistence of positive solutions of (\ref{eq_a=1}) presents a great degree of difficulties; see \cite{Bahri-Coron, Coron,del_p,man_1,man_2,man_3,wei,  A2, Passaseo, Passaseo_2,shaaf, schmitt}.  Many of these results are very technical and some require  perturbation arguments.  \\

The possibility of utilizing  the most of features  that a given problem can offer  to gain  improved compactness for supercritical problems and to prove existence of   new types of  solutions is what motivated us for this work. As mentioned earlier,  these features could be a combination of symmetry, monotonicity, convexity
and etc. For instance, let us consider the Neumann boundary problem 
 \begin{equation} \label{eq_Nball}
\left\{\begin{array}{rr}
-\Delta u + u = a(r) u^{p-1} &  \mbox{ in } B_1, \\
\partial_\nu u= 0 &   \mbox{ on } \partial B_1,
\end {array}\right.
\end{equation}
 where $B_1$ is the unit ball centered at the origin in $\IR^N$.  The interest here is in obtaining nontrivial solutions for values of $p>\frac{2N}{N-2}$. 
 In  \cite{first_rad_neum} they considered the variant of (\ref{eq_Nball}) given by $ -\Delta u + u = |x|^\alpha u^{p-1}$ in $B_1$ with $ \frac{ \partial u}{ \partial \nu}=0$ on $ \partial B_1$ (for Dirichlet versions of the H\'enon equation see, for instance, \cite{Ni,glad,cowan}).
 They proved the existence of a positive radial solutions of this equation with arbitrary growth using a shooting argument.   The solution turns out to be an increasing function.
They also perform numerical computations to see the existence of positive oscillating solutions.  In  \cite{serra_tilli} they considered  (\ref{eq_Nball})   along with the classical energy associated with the equation given by
  \[E(u):=\int_{B_1} \frac{ | \nabla u|^2 +u^2}{2}\, dx  -\int_{B_1} a(|x|)F(u)\,  dx,\] where $F'(u)=f(u)$ (they considered a more general nonlinearity).    Their goal was  to find critical points of $E$ over $H^1_{rad}(B_1):=\{ u \in H^1(B_1):  u \mbox{ is radial} \}$.  Of course since $f$ is supercritical the standard approach of finding critical points will present difficulties and hence their idea was to find critical points of $E$ over the cone $ \{ u \in H_{rad}^1(B_1): 0 \le u, \mbox{ $u$ increasing} \}$.   Doing this is somewhat standard but now the issue is the critical points don't necessarily correspond to critical points over $H_{rad}^1(B_1)$ and hence one can't conclude the critical points solve the equation; for instance the critical point could lie on the boundary of the convex cone and then one cannot perturb in all directions.   The majority of their work was to show that in fact the critical points of $E$ on the cone are really critical points over the full space.   We remark  that this work generated a lot of interest in this equation and many authors investigated these idea's of using monotonicity to overcome a lack of compactness.  For further results  regarding these Neumann problems on radial domains (some using these monotonicity ideas and some using other new methods) see  \cite{Mo2, first_grossi, Weth, grossi_new, den_serra, add4, ACL,  add5}. 
  
In   \cite{cowan-abbas}, by making use of duality theory in convex analysis,   we examined the super critical Neumann problem given by 
\begin{equation}   \label{eq_trans}
\left\{\begin{array}{ll}
-\Delta u + u= a(x) f(u), &  \mbox{ in } \Omega, \\
u>0,    &   \mbox{ in } \Omega, \\
\frac{\partial u}{\partial \nu}= 0, &   \mbox{ on } \pOm,
\end {array}\right.
\end{equation}
for multiradial domains which are  a  natural extension of radial domains.  The idea of using convexity to deal with partial differential equations has a very long history starting from \cite{EKT, Toland} and also the recent papers \cite{Mo, Mo3}.   For Neumann problems on general domains see \cite{Mo2, 100, 101, 102,104,103,105,50,106}.  \\ 

We now return to the Dirichlet problems.   There have been many supercritical works that deal with domains that have certain symmetry, for instance, see \cite{clapp1,clapp2,clapp3, clapp4,clapp5,clapp6, Mo4}.  

In   the case of  the annulur domains the authors in  \cite{Ann100, Ann101,Ann102} examined subcritical or slightly supercritical problems on expanding annuli and obtained nonradial solutions.  In \cite{Ann2}  they obtain nonradial solutions to supercritical problems on expanding annulur domains.  In \cite{Ann1} they consider nonradial expanding annulur domains and they obtain the existence of positive solutions.   In \cite{wei, clapp5}  they consider domains with a small hole and obtain positive solutions. 
We shall also refer the interested reader to the recent works 
\cite{gelf_a, Weth_annulus, orgin_an}  where the idea of monotonicity together with variational and non-varitioanal  methods were employed to deal with equation (\ref{eq_trans}) in annular type domains.

\section{A variational approach towards  supercritical problems} \label{variational_abst}

In this section we assume that $\Omega$ is a domain in $\R^N$ which is not necessarily bounded. We also assume that $a $ is a non-negative measurable function that is not identically zero.
For $p>1,$ we define 
\[L^p_a(\Omega)=\left \{u:\,\,  \int_\Omega a(x)|u|^p\,dx <\infty\right \},\]
equipped with the norm 
\[\|u\|_{L^p_a(\Omega)}=\left ( \int_\Omega a(x)|u|^p\,dx  \right)^{\frac{1}{p}}.\]
We have the following general variational principle  for possibly super critical elliptic problems.


\begin{thm}\label{var-pri} ($K$ ground state solution) Let $\Omega$ be a  domain in $\R^N$,    $p>2,$  and $a$ be a non-negative function that is not identically zero. Let $\lambda $ be a non-negative number which is strictly positive if $\Omega$ is unbounded.  Consider the problem 
\begin{eqnarray}
\left\{
  \begin{array}{ll}
    -\Delta  u+\lambda u=a(x)|u|^{p-2}u, & x \in \Omega, \\
   u \in H_0^1(\Omega),
  \end{array}
\right.
\end{eqnarray}
and its formal  Euler-Lagrange functional 
\[I(u)=\frac{1}{2}\int_\Omega (|\nabla u|^2+\lambda u^2) \, dx-\frac{1}{p}\int_\Omega a(x)|u|^p\, dx.\]
Let    $K$ be a  convex and closed subset of $ H_0^1(\Omega)$. Suppose the following two assertions hold:
\begin{enumerate}
\item[$(i)$] $K$ is compactly  embedded in $L_a^p(\Omega),$ i.e., 
every bounded sequence in $K$ has a converging subsequence in $L_a^p(\Omega).$

\item[ $(ii)$] (Pointwise invariance property) For each $\bar u \in K$ there exists $\bar{v}\in K$   such that 

\[-\Delta \bar{v}+\lambda \bar v= a(x)|\bar{u}|^{p-2}\bar{u},\] 

in the weak sense, i.e.,
\[\int_\Omega \nabla \bar v \cdot \nabla\eta \, dx  +\lambda \int_\Omega\bar v\eta \, dx  =\int_\Omega a(x)|\bar{u}|^{p-2}\bar{u}\eta\, dx, \qquad \forall \eta \in H_0^1(\Omega)\cap L_a^p(\Omega).\]
\end{enumerate} 
 Then there exist $c>0$ and  $\tilde {u} \in K$ such that $I(\tilde u)=c$  and  $\tilde u$ is  a weak solution of the equation 
\begin{equation}  \label{eqthm}
\left\{\begin{array}{ll}
-\Delta u +\lambda u=  a(x)|u|^{p-2}u, &   x\in \Omega \\
u= 0, &    x\in  \partial \Omega.
\end {array}\right.
\end{equation}  We call $\tilde u$ a \textbf{$K$-ground state solution} of (\ref{eqthm}).
A characterization for the critical value $c$ is given in the proof.
\end{thm}

We shall need some preliminaries before proving this theorem. 
Consider the Banach space  $V= H_0^1(\Omega)\cap L_a^p(\Omega)$  equipped  with the following norm
\[\|u\|_V= \|u\|_{H_0^1(\Omega)}+\|u\|_{L_a^{p}(\Omega)},\]
and 
note that the duality pairing between $V$ and its dual $V^*$ is defined by
\[\langle u, u^*\rangle=\int_\Omega u(x) u^*(x)\, dx,\qquad \forall u\in V, \, \forall u^* \in V^*.\]
We define $\Psi: V \to \R$ and $\Phi: V \to \R$ by 
\[\Psi(u)=\frac{1}{2}\int_\Omega (|\nabla u|^2+\lambda u^2) \, dx,\]
and 
\[\Phi(u)=  \frac{1}{p}\int_{\Omega}  a(x)|u|^p dx. \]

We remark that even though $\Phi$ is not even well-defined on  $H_0^1(\Omega)$ for large $p$, but it is continuously differentiable on the space $V= H_0^1(\Omega)\cap L_a^p(\Omega)$.
Finally, let us introduce the  functional $E_K(u): V\rightarrow (-\infty, +\infty]$ defined by
\begin{equation}\label{E}
E_K(u):= \Psi_K(u)- \Phi(u)
\end{equation}
where
\begin{eqnarray}
\Psi_K(u)=\left\{
  \begin{array}{ll}
      \Psi(u), & u \in K, \\
    +\infty, & u \not \in K.
  \end{array}
\right.
\end{eqnarray}
Note that $E_K$  is indeed the Euler-Lagrange functional corresponding to (\ref{eqthm})  restricted to $K$.
We shall now recall some notations and results for the minimax principles for lower semi-continuous functions.
\begin{definition}\label{cp}
Let $V$ be a real Banach space,  $\Phi\in C^1(V,\mathbb{R})$ and $\Psi: V\rightarrow (-\infty, +\infty]$ be proper (i.e. $Dom(\Psi)\neq \emptyset$), convex and lower semi-continuous. 
A point $u\in  V$ is said to be a critical point of
\begin{equation} \label{form}I:=  \Psi-\Phi \end{equation} if $u\in Dom(\Psi)$  and if it satisfies
the inequality
\begin{equation}
 \CO{D \Phi(u)}{ u-v}+ \Psi(v)- \Psi(u) \geq 0, \qquad \forall v\in V,
\end{equation}
where $ \CO{ \cdot }{ \cdot } $ is the duality pairing between $V$ and its dual $V^*.$
\end{definition}

\begin{definition}\label{psc}
We say that $I$ satisfies the Palais-Smale compactness  condition (PS)   if
every sequence $\{u_n\}$ such that
\begin{itemize}\label{2}
\item  $I[u_n]\rightarrow c\in \mathbb{R},$
\item  $\CO{D \Phi(u_n)}{ u_n-v}+ \Psi(v)- \Psi(u_n) \geq -\E_n\|v- u_n\|, \qquad \forall v\in V,$
\end{itemize}
where $\E_n \rightarrow 0$, then $\{u_n\}$ possesses a convergent subsequence.
\end{definition}

The following non-smooth mountain pass  theorem is  due to  A. Szulkin  \cite{szulkin}. 
\begin{thm}   \label{MPT}   Suppose that
$I : V \rightarrow (-\infty, +\infty ]$ is of the form (\ref{form}) and satisfies the Palais-Smale   condition and  the Moutaint Pass Geometry (MPG):
\begin{enumerate}
\item $I(0)= 0$.
\item  There exists $e\in V$ such that $I(e)\leq 0$.
\item There exists some $\rho$ such that $0<\rho<\|e\|$ and for every $u\in V$ with $\|u\|= \rho$ one has $I(u)>0$.
\end{enumerate}
Then $I$ has a critical value $c>0$ which is  characterized by
$$c= \inf_{\gamma\in \Gamma}  \sup_{t\in [0,1]} I[\gamma(t)],$$
where   $\Gamma= \{\gamma\in C([0,1],V): \gamma(0)=0,\gamma(1)= e\}.$
\end{thm}

\noindent 
\textbf{Proof of Theorem \ref{var-pri}} Note first   that $K$ is a weakly closed convex subset  in $H_0^1(\Omega)$ where we equip 
$H_0^1(\Omega)$ by the following norm:
\[\|u\|^2_{H_0^1(\Omega)}= \int_{\Omega} (|\nabla u|^2+\lambda |u|^2) dx.\]
It follows from condition $(i)$ in the theorem that $K$ is compactly embedded in $L^p_a$.  Thus,    there exists a constant $C$ such that 
\begin{equation}\label{equiv}\|u\|_{H_0^1(\Omega)}\leq \|u\|_V \leq C \|u\|_{H_0^1(\Omega)}, \qquad \forall u \in K. \end{equation}

Both the mountain pass geometry and  (PS) compactness condition for the function $E_K= \Psi_K- \Phi$ given in (\ref{E}) follow from the standard arguments together with inequality (\ref{equiv}). Here,  for the conveience of the reader,  we sketch the proof for the (PS) compactness condition and the mountain pass geometry.
Suppose that $\{u_n\}$ is a sequence in $K$ such that $E_K(u_n)\rightarrow c\in \mathbb{R}$, $\E_n \rightarrow 0$ and
\begin{equation}\label{od}
 \Psi_K(v)- \Psi_K(u_n)+ \langle D\Phi(u_n), u_n-v \rangle\geq -\E_n \|v- u_n\|_V, \qquad \forall v\in V.
\end{equation}
We must show that $\{u_n\}$ has a convergent subsequence in $V$.  Firstly,  we prove that $\{u_n\}$ is bounded in $V$. 
Note that 
   since $E_K(u_n)\rightarrow c$, then  for large values of $n$ we have 
\begin{equation}\label{od0}
 \frac{1}{2}\| u_n\|^2_{H_0^1(\Omega)}- \frac{1}{p}\int_{\Omega}  a(x)|u|^p dx\leq c+1.
\end{equation}
Note that 
\[\langle D\varphi(u_n),  u_n\rangle = \int_{\Omega} a(x) |u_n(x)|^{p} dx.\]
Thus, by setting  $v= ru_n$ in (\ref{od}) with $r=1+1/p$ we get 
\begin{equation}\label{od1}
\frac{(1-r^2)}{2} \|u_n\|^2_{H_0^1(\Omega)}+ (r-1)\int_{\Omega}a(x)|u_n|^p dx\leq \E_n (r-1) \| u_n\|_V.
\end{equation}
Adding up  (\ref{od1}) and  (\ref{od0}) yields that 
\begin{equation*}
 \|u_n\|^2_{H_0^1(\Omega)}\leq C_0(1+  \| u_n\|_{V}),
\end{equation*}
for some constant $C_0>0.$
Therefore, by considering (\ref{equiv}),   $\{u_n\}$ is bounded in $H_0^1(\Omega).$  Using standard results in Sobolev spaces, after passing to a subsequence if necessary, there exists $\bar u \in H_0^1(\Omega)$ such that  $u_n\rightharpoonup \bar{u}$ weakly in $H_0^1(\Omega)$ and  $u_n\rightarrow  \bar{u}$ a.e..  Also according to condition $(i)$  in the theorem,  from boundedness of $\{u_n\}\subset K$  in $H_0^1(\Omega)$, one can deduce that  the strong  convergence of $u_n$ to $\bar u$ in  $L_a^p.$ 
Now in (\ref{od}) set $v= \bar{u}$ to get 
\begin{equation}\label{oder1}
\frac{1}{2} ( \|\bar{u}\|^2_{H_0^1(\Omega)} - \|u_n\|^2_{H_0^1(\Omega)}) +\int_{\Omega}a(x)|u_n|^{p-1} ( u_n- \bar{u}) dx
\geq -\E_n \|u_n- \bar{u}\|_V.
 \end{equation}
Therefore, it follows from  (\ref{oder1})  that 
\[
\frac{1}{2} (      \limsup_{n\rightarrow \infty} \|u_n\|^2_{H_0^1(\Omega)}  -     \|\bar{u}\|^2_{H^1}  ) \leq 0.
\]
The latter  yields that 
$$u_n\rightarrow \bar{u}\quad \text{strongly in }\quad V$$
as desired. 
We now   verify the mountain pass geometry of the functional $E_K.$
It is clear that $E_K(0)=0$. Take $e \in K$. It follows that 
\begin{equation*}
 E_K(te)=
\frac{t^2}{2} \int_{\Omega} (|\nabla e|^2+\lambda |e|^2) dx-  \frac{t^p}{p}\int_{\Omega} a(|x|)|e|^{p} dx\\
 \end{equation*}
Now, since $p> 2$,  for $t$ sufficiently large $E_K(te)$ is negative.
 Take 
 $u\in K$ with $\|u\|_{V}= \rho>0$. We have
\begin{equation*}\label{MPG6}
E_K(u)= \frac{1}{2} \|u\|^2_{H^1}- \frac{1}{p}\int_{\Omega} a(|x|) |u|^p dx.
\end{equation*}
Note that by (\ref{equiv}), there exist positive constant $C$  such that  for every $u\in K$ one has
\begin{equation}\label{norm}
\|u\|_{H^1}\leq\|u\|_{V}\leq C \|u\|_{H^1}.
\end{equation}
We also have that 
$$\int_{\Omega} a(|x|) |u|^p dx\leq C_0\|u\|^p_{V}. $$
Therefore
\begin{align}
 E_K(u)& \geq  \frac{1}{2} \|u\|^2_{H^1}- \frac{1}{p}\int_{\Omega} a(|x|) |u|^p dx\geq  \frac{1}{2} \|u\|^2_{H^1}-  \frac{C_0}{p} \|u\|^p_{V}\notag\\
 & \geq \frac{1}{2C^2} \|u\|^2_{V}-  \frac{C_0}{p} \|u\|^p_{V}=    \frac{1}{2C^2} \rho^2-  \frac{1}{p}\rho^p >0, \notag
 \end{align}
 provided  $\rho>0$ is  small enough, since $p>2$.  If $u\notin K$, then clearly $E_K(u)>0$.  Thus, (MPG) holds for the functional $E_K$.  It now follows from Theorem (\ref{MPT}) that $E_K$ has a critical point $\bar u \in K,$  with 
 $E_K(\bar u)= c>0$ where the critical value $c$ is characterized by
\begin{equation} \label{critical value}c= \inf_{\gamma\in \Gamma}\max_{t\in [0,1] }E_K[\gamma(t)],\end{equation}
where \[\Gamma= \{\gamma\in C([0,1], V) : \gamma(0)=0\neq \gamma(1), E_K(\gamma(1))\leq 0\}.\]
Since  $E_K(\bar u)>0$,  we have that $\bar u$ is non-zero.  Since $\bar{u}$ is a critical point of $E_K$,   it follows from Definition \ref{cp}   that 
\begin{equation}\label{p1}
 \CO{D \Phi(\bar{u})}{ \bar{u}-v}+ \Psi_K(v)- \Psi_K(\bar{u}) \geq 0, \qquad \forall v\in V.
\end{equation}
On the other hand, by (ii), there exists $\bar v \in K$ satisfying 
\begin{equation}\label{eqpp}
\left\{\begin{array}{ll}
-\Delta \bar  v+\lambda \bar v =  a(x)|\bar{u}|^{p-2}\bar{u}, &   x\in \Omega \\
\bar v= 0, &    x\in  \partial \Omega,
\end {array}\right.
\end{equation}
in the weak sense. 
By setting $v= \bar{v}$ in (\ref{p1}) we obtain that 
\begin{eqnarray*}
\frac{1}{2}\int_{\Omega} |\nabla \bar{v}|^2 +\lambda \bar v^2dx  -\frac{1}{2}\int_{\Omega} |\nabla \bar{u}|^2+\lambda \bar u^2 dx\notag  &\geq & \int_{\Omega}a(x)|\bar{u}|^{p-2}\bar{u}(\bar{v}- \bar{u})\, dx
 \\&=& \int_{\Omega}\nabla\bar{v} \cdot \nabla(\bar{v}- \bar{u})+\lambda \bar v(\bar v-\bar u)\, dx
\end{eqnarray*}
where the last equality  follows from (\ref{eqpp}).
Therefore, 

\begin{equation}\label{p2}
\frac{1}{2}\int_{\Omega} |\nabla\bar{v}- \nabla \bar{u}|^2\, dx+\frac{\lambda }{2}\int_{\Omega} |\bar{v}- \bar{u}|^2\, dx\leq 0,
\end{equation}
which implies that $\bar{u}= \bar{v}$. Taking into account that $\bar{u}= \bar{v}$ in (\ref{eqpp}) we have that $\bar u$ is a weak solution of 
(\ref{eqthm}):
\begin{equation}\label{pscb}
\left\{\begin{array}{ll}
-\Delta u+\lambda u=  a(x)|u|^{p-2}u, &   x\in \Omega, \\
u= 0, &    x\in  \partial \Omega.
\end {array}\right.
\end{equation} \\

\hfill $\square$



\section{Domains of double revolution} \label{double_dom}  In this section we gather some information about the domains of double and higher revolution. We also state and prove useful embedding theorems for these type of domains.\\ We start by domains of double revolution. Consider writing $ \IR^N=\IR^{m} \times \IR^{n} $ where $ m,n \ge 1$ and $m+n=N$.
We define the variables $s$ and $t$ by
\[ s:= \left\{ x_1^2 + \cdot \cdot \cdot  + x_{m}^2\right\}^\frac{1}{2}, \qquad t:=\left\{x_{m+1}^2 + \cdot \cdot \cdot + x_N^2 \right\}^\frac{1}{2}.  \] We say that $\Omega \subset \IR^N$ is a \emph{domain of double revolution} if it is invariant under rotations of the first $m$ variables and also under rotations of the last $n$ variables.  Equivalently, $\Omega$ is of the form $ \Omega=\{ x \in \IR^N: (s,t) \in U \}$ where $U$ is a domain in $ \IR^2$ symmetric with respect to the two coordinate axes.  In fact,
\[ U= \big \{ (s,t) \in \IR^2:  x=(x_1=s, x_2=0, ... , x_m=0,  x_{m+1}=t, ... , x_N =0 ) \in \Omega \big \},\] is the intersection of $\Omega$ with the $(x_1,x_{m+1})$ plane.  Note that $U$ is smooth if and only if $\Omega$ is smooth. 
We denote $\widehat{\Omega}$ to  be the intersection of $U$ with the first quadrant of $\IR^2$, that is,
\begin{equation}\label{omegahat}\widehat{\Omega}=\big\{(s,t) \in U: \,\, s> 0, \,\,  t>0 \big\}.
\end{equation}

Using polar coordinates we can write $ s= r \cos(\theta),$ $ t = r \sin(\theta)$ where $ r=|x|= |(s,t)|$ and $ \theta$ is the usual polar angle in the $ (s,t)$ plane. \\ 

 All domains will be bounded domains in $ \IR^N$ with smooth boundary unless otherwise stated.   To describe the domains in terms of the above polar coordinates we will write \begin{equation}\label{omegahat+} \widetilde{\Omega}:=\big \{ (\theta,r): (s,t) \in \widehat{\Omega} \big \}.  
\end{equation}
Define  \[ H^1_{0,G}:=\left\{ u \in H^1_0(\Omega): gu =u \quad \forall g \in G \right\},\] where $G:=O(m) \times O(n)$ where $ O(k)$ is the orthogonal group in $ \IR^k$ and $ gu(x):=u( g^{-1} x)$.

In \cite{orgin_an} we have considered annular domains and annular domains with monotonicity via the following definition:
\begin{definition} \label{def_an_pi2} We refer to a domain of double revolution in $\R^{N}$ with $N=m+n$ an annular domain if its  associated domain $\widehat \Omega$ in the $(s,t)$ plane in $\R^2$ is of  the form  
\begin{equation} 
\widetilde{\Omega}=\left\{ (\theta,r): g_1(\theta)< r < g_2(\theta), \theta \in \left(0, \frac{\pi}{2}\right) \right\} 
\end{equation} 
 in polar coordinates.  Here  $ g_i>0$ is smooth on $ [0, \frac{\pi}{2}]$ with $ g_i'(0)=g_i'( \frac{\pi}{2})=0$ and $ g_2(\theta)> g_1(\theta)$ on $ [0, \frac{\pi}{2}]$.   We  call  $\Omega$ an  \emph{annular domain with monotonicity} if  $ g_1$ is increasing and $ g_2$ is decreasing on $ (0, \frac{\pi}{2})$. 
 \end{definition}  
 To distinguish these domains from the new ones we will refer to these as $ \frac{\pi}{2}$ annular domains with and without monotonicity.     We proved the following imbeddings: 
 \begin{thm*} \cite{orgin_an} Let $\Omega$ denote  a $\frac{\pi}{2}$ annular domain in $ \IR^N$.

\begin{enumerate}  \item (Imbedding without monotonicity) Suppose $\Omega$ has no monotonicity and  \[ 1 \le p < \min \left\{ \frac{2(n+1)}{n-1},  \frac{2(m+1)}{m-1} \right\}.\] Then 
$ H^1_{0,G}(\Omega) \subset \subset  L^p(\Omega)$ with the obvious interpretation in the case of $m=n=1$.

\item (Imbedding with monotonicity) Suppose $ \Omega$ has monotonicity,  $ n \le m$ and   
\[ 1 \le p<  \frac{2(n+1)}{n-1} =\max  \left\{ \frac{2(n+1)}{n-1},  \frac{2(m+1)}{m-1} \right\}.\]
Then $K_{-,\frac{\pi}{2}}  \subset \subset L^p(\Omega)$  with the obvious interpretation if $n=1$ where 
\[ K_{-,\frac{\pi}{2}}=\left\{ 0 \le u \in H^1_{0,G}(\Omega): u_\theta \le 0 \mbox{ a.e. in } \widetilde{\Omega} \right\}.\]

\end{enumerate} 
\end{thm*}

\begin{remark} \label{remark_imbed_mon} \begin{enumerate} \item  The above imbedding makes sense with a bit of heuristics.    Consider an annular domain in $ \IR^N$ with $N=m+n$ and we suppose $ n \le m$. Suppose  we are given a sequence of  functions $ 0  \le u_k \in H^1_{0,G}(\Omega)$.  If the functions concentrate near $ t=0$  then the problem looks like a problem in dimension $ n+1$ (ie. the $t$ variable has dimension $n$ and the $s$ variable has dimension $1$ since we are away from $ s=0$) and hence the critical Sobolev exponent $ \frac{2(n+1)}{(n+1)-2}$ should play a role.   The functions can also concentrate near $ s=0$ and then the relevant exponent is $ \frac{2(m+1)}{(m+1)-2}$.  The functions can also concentrate in other regions but they are of lower dimension and hence doesn't limit the imbedding.   This suggests part 1 of Theorem A. 

\item To see part 2 of Theorem A we note that we now have monotonicity in $ \theta$ and hence the functions only have the option to concentrate on $ \theta=0$ or on the $ s$ axis and hence this gives the improved result. 
\end{enumerate}
\end{remark}

Before going into more details we give some more background on domains of double revolution. 

  Assume $\Omega$ is a domain of double revolution and $ v$ is a function defined on $ \Omega$ that just depends on $ (s,t)$,  then one has 
\[ \int_\Omega v(x) dx = c(m,n) \int_{\widehat{\Omega}} v(s,t) s^{m-1} t^{n-1} ds dt,\]
where $c(m,n)$ is a positive constant depending on $n$ and $m.$ Note that strictly speaking we are abusing notation here by using the same name; and we will continuously do this in this article. Given a function $v$ defined on $ \Omega$ we will write $ v=v(s,t)$ to indicate that the function has this symmetry.   

To solve equations on domains of double  revolution one needs to relate the equation to a new one on $ \widehat{\Omega}$ defined in (\ref{omegahat}). Suppose $\Omega$  is a  domain of double revolution and $ f$ has is function defined on $\Omega$ with the same symmetry 
(ie. $ gf(x)= f(g^{-1} x)$ all $ g \in G$). 
Suppose that $u(x)$ solves 
\begin{equation} \label{eq_double_lin}
\left\{\begin{array}{ll}
-\Delta u(x) = f(x) &  \mbox{ in } \Omega, \\
u= 0 &   \mbox{ on } \pOm.
\end {array}\right.
\end{equation} Then $ u=u(s,t)$ and $u$ solves 
\begin{equation} \label{eq_double_lin_st}
-u_{ss}-u_{tt}- \frac{(m-1) u_s}{s}-\frac{(n-1) u_t}{t} = f(s,t)   \mbox{ in } \widehat{\Omega}, 
\end{equation} with $u=0$ on $ (s,t) \in \partial \widehat{\Omega} \backslash ( \{s=0\} \cup \{t=0\} )$.  If $u$ is sufficiently smooth then $ u_s =0 $ on $\partial \widehat{\Omega} \cap \{s=0\}$ and $u_t=0$ on $ \partial \widehat{\Omega} \cap \{t=0\}$ after considering the symmetry properties of $u$.  \\

One can easily refine  the notion of the domain of double revolution  to domains of $m$ revolution.\\

\noindent 
\textbf{Domains of $m$ revolution.}   Consider writing $ \IR^N = \IR^{n_1} \times \IR^{n_2} \times \cdot \cdot \cdot \times \IR^{n_m}$ where $ n_1 + \cdot \cdot \cdot + n_m=N$ and $n_1,...,n_m\geq 1.$
We say that $\Omega \subset \IR^N$ is a \emph{domain of $m$ revolution} if it is invariant under rotations of the first $n_1$ variables, the next $n_2$ variables, ..., and finally in the last $n_m$ variables.  We define the variables $ t_i$ via
\[ t_1^2:={ x_1^2+ \cdot \cdot \cdot  + x_{n_1}^2}, \quad t_2^2:= { x_{n_1+1}^2 + \cdot \cdot \cdot  + x_{n_1+n_2}^2},\] and similar for $ t_i$ for $ 3 \le i <m$. Finally we define
\[ t_m^2:= { \sum_{k=n_1+ n_2 + \cdot \cdot \cdot  + n_{m-1} +1}^N x_k^2}.\]
We now define
\begin{eqnarray*} U= \Big \{ t   \in \IR^m;  x=(x_1,...,x_N) \in \Omega, \text{ where }x_1=t_1,\,  x_{n_1+n_2+ \cdot \cdot \cdot  +n_{k-1} +1}= t_k \text{ for  } 2\le k\leq m,\text{ and }\\ x_i=0 \text{ for } i\not=1, n_1+1, n_1+n_2+1,..., n_1+n_2+...+n_{m-1}+1 \Big \}.\end{eqnarray*}
 We define $ \widehat{\Omega} \subset \IR^m$ to be the intersection of $U$ with the first sector of $ \IR^m$. We now define the appropriate measure
\[ d \mu_m(t) = d \mu_m^{(n_1, ... ,n_m)}(t_1, ... , t_m) = \prod_{k=1}^m t_k^{n_k-1} d t_k.\]  Given any function $v$ defined in $\Omega$, that depends only on the radial variables $t_1,t_2,..,t_m$ one has
\[ \int_\Omega v(x) dx = c(n_1,...,n_m) \int_{\widehat{\Omega}}v(t)  d \mu_m(t),\]
where $c(n_1,...,n_m)$ just depends on $n_1,...,n_m.$
Given that $\Omega \subset \R^N$  is a domain of $m$ revolution with $\sum_{i=1}^m n_i=N$, let
\[ G:= O(n_1) \times O(n_2) \times ... \times O(n_m),\]  where $O(n_i)$ is  the orthogonal group in $ \IR^{n_i}$ and
 consider
 \[ H^1_{0,G}:=\left\{ u \in H^1_0(\Omega): gu =u \quad \forall g \in G \right\},\]  where $gu(x):=u(g^{-1} x)$.  If $u \in H^1_{0,G} $ then
$u$ has symmetry compatibility  with $ \Omega$, ie. $u(x)$ depends on just $ t_1,...,t_m$ and we write this as  $u(x)=u(t_1,...,t_m)$ where $(t_1,...,t_m)\in \widehat{\Omega}.$
 We have the following embedding result for the domains of $m$
 revolution.
 
  \begin{thm} \label{embed-mrev}
  Let $\Omega$ denote  a bounded    domain of $m$ revolution in $ \IR^N$ with $N=n_1+...+n_m$ and $n_i\geq 1$ such that $0\not \in \bar \Omega.$
  Assume that  \[ 1 \le p < \min \left\{ \frac{2(N-n_i+1)}{N-n_i-1}; \,\, i=1,...,m\right \}.\] Then 
$ H^1_{0,G}(\Omega) \subset \subset  L^p(\Omega)$ with the obvious interpretation in the case of $N-n_i=1$. 
\end{thm}

\begin{proof} 
Assume that $x=(y_1,...,y_m) \in \R^N=\Pi_{i=1}^m\R^{n_i}$.  Let $R_1$ and $R_2$ be such that $0<R_1<|x|<R_2$
for all $x \in \Omega.$
Choose  $\delta$ small enough such that $\sqrt{m}\delta<R_1.$   It then follows that for each $x=(y_1,...,y_m)\in \Omega $ we have that $|y_i|\geq \delta$ for at least one $i\in \{1,...,m\}.$
Therefore, 
\begin{eqnarray*}
\int_\Omega |u|^p\, dx\leq \Sigma_{i=1}^m \int_{\Omega,\, |y_i|\geq \delta} |u|^p\, dx
&\leq &\Sigma_{i=1}^m c_i \int_{\Omega,\, |r_i|\geq \delta}r_i^{n_i-1} |u(y_1,..., y_{i-1}, r_i, y_{i+1},..,y_m)|^p\, dy_1\,...dr_i...dy_m \\
&\leq &\Sigma_{i=1}^m c_iR_2^{n_i-1} \int_{\Omega,\, |r_i|\geq \delta} |u(y_1,..., y_{i-1}, r_i, y_{i+1},..,y_m)|^p\, dy_1\,...dr_i...dy_m
\end{eqnarray*}
for appropriate constants $c_i$.
Morovere, for 
\[1<p< \frac{2(N-n_i+1)}{N-n_i-1},\]
we have that 
\[\int_{\Omega,\, |r_i|\geq \delta} |u(y_1,..., y_{i-1}, r_i, y_{i+1},..,y_m)|^p\, dy_1\,...dr_i...dy_m \]
is being controlled by the $H_1(\Omega,\, |r_i|\geq \delta)$.
On the other hand 
\begin{eqnarray*}
\int_{\Omega,\, |r_i|\geq \delta}\left( |\nabla u(y_1,..., y_{i-1}, n_i, y_{i+1},..,y_m)|^2+|u(y_1,..., y_{i-1}, r_i, y_{i+1},..,y_m)|^2\right)dy_1\,...dr_i...dy_m &\leq & \\
 \delta^{-n_i+1}\int_{\Omega,\, |r_i|\geq \delta} r_i^{n_i-1} \left(|\nabla u(y_1,..., y_{i-1}, r_i, y_{i+1},..,y_m)|^2+|u(y_1,..., y_{i-1}, r_i, y_{i+1},..,y_m)|^2\right)dy_1\,...dr_i...dy_m &\leq &\\
 C_i\delta^{-n_i+1}\int_{\Omega}  \left(|\nabla u(y_1,..., y_{i-1}, y_i, y_{i+1},..,y_m)|^2+|u(y_1,..., y_{i-1}, y_i, y_{i+1},..,y_m)|^2\right)dy_1\,...dy_i...dy_m &=&\\
C_i\delta^{-n_i+1}\|u\|^2_{H_1(\Omega)},
\end{eqnarray*}
for appropriate constants $C_i.$  This completes the proof. 
\end{proof}

\section{Supercritical elliptic problems on domains of double revolution} \label{section_double_annular} 
 In this section we examine the  equation  
 
 \begin{equation} \label{eq_fir}
\left\{\begin{array}{ll}
-\Delta u = a(x) u^{p-1} &  \mbox{ in } \Omega, \\
u= 0 &   \mbox{ on } \partial \Omega,
\end {array}\right.
\end{equation}
 where  $ \Omega$ is a domain of double revolution in $\R^N=\R^n\times\R^n$.   Note when $m=n$, Theorem A does not show any improvements in compactness when using monotonicity.   In this case the equation has a certain invariance across $ \theta=\frac{\pi}{4}$ and this suggests one examine domains with a certain invariance also.  This brings us to a first type of new domains.  
 
 \begin{definition} \label{pi4} We will call a domain of double revolution in $ \IR^N$ a \emph{$\frac{\pi}{4}$-annular domain with monotonicity} provided the domain is an annular domain via Definition  \ref{def_an_pi2} (ie. $ g_i>0$ is smooth on $ [0, \frac{\pi}{2}]$ with $ g_i'(0)=g_i'( \frac{\pi}{2})=0$ and $ g_2(\theta)> g_1(\theta)$ on $ [0, \frac{\pi}{2}]$) and $g_1$ is increasing and $g_2$ is decreasing on $(0,\frac{\pi}{4})$ and both $g_1,g_2$ are even across $ \theta=\frac{\pi}{4}$.   For these new domains we define a suitable subset of $ \widetilde{\Omega}$ given by 
 \begin{equation}
 \label{tilde_half}
\widetilde{\Omega}_0=\left\{(\theta,r):  g_1(\theta)<r<g_2(\theta), 0<\theta < \frac{\pi}{4} \right\}.
\end{equation}

 \end{definition}

   We now are in a position to define the class of functions we work on in this setting. 
\begin{enumerate}
    \item ($K_-$) In the case of $\Omega$ a $\frac{\pi}{4}$-annular domain with monotonicity (see Definition \ref{pi4}) we define $K_-$ to be  the set of nonnegative functions $ u \in H_{0,G}^1(\Omega)$ with $ u_\theta \le 0$ in $ \widetilde{\Omega}_0$ and which are even across $ \theta=\frac{\pi}{4}$.

    \item ($K_+$) In the case of $\Omega$ an annulus we define $K_+$  to be  the set of nonnegative functions $ u \in H_{0,G}^1(\Omega)$ with $ u_\theta \ge 0$ in $ \widetilde{\Omega}_0$ and which are even across $ \theta=\frac{\pi}{4}$.  
    \end{enumerate} 
     Note $K_-$ is defined for an annulus and a more general annular domain with the added assumptions where as   we only define $K_+$ for an annulus.  Our approach utilizing $K_+$ will fail on a more general annular domain.      The imbeddings we prove regarding $K_-$ are essentially the same as Theorem A.  For $K_+$ one expects to get more.    Before we state our main theorem for this section we need to define a quantity that will be relevant to showing the ground states on radial domains are nonradial and this quantity will be relevant for the equaitons that follow in later sections also.  
   Indeed, we define  
    \begin{equation} 
\beta_0(\Omega):=\inf_{u\in H_0^1(\Omega)}
\frac{\int_\Omega |\nabla u|^2 \, dx}{\int_\Omega \frac{u^2}{|x|^2} \, dx}. \end{equation}  Note this quantity  is just the best constant in the classical Hardy inequality.  So if $ 0 \in \Omega$ or $ \Omega$ is an exterior domain then $ \beta_0(\Omega) = \frac{(N-2)^2}{4}$.

    \begin{thm} \label{theorem-nonlinear-annular} Let  $\Omega$ be a bounded domain in $\R^N$ with $N=2n.$   
    
    \begin{enumerate} \item Suppose $\Omega$ is a bounded $\frac{\pi}{4}$-annular domain with monotonicity, $ a=a(s,t) $ is positive and sufficiently smooth and $ a_\theta \le 0$ in $ \widetilde{\Omega}_0$. Then for all 
   \[ 2 < p < \frac{2N+4}{N-2},\] there is a positive classical $K_-$ ground state solution $u $ of (\ref{eq_fir}).   Note this case includes the case of $\Omega$ an annulus. 
   
   \item Suppose $\Omega$ is an annulus with $ a=a(s,t)$ positive and sufficiently smooth and $ a_\theta \ge 0 $ in $ \widetilde{\Omega}_0$. 
   \begin{itemize}
   \item[2-a] Then for all $2<p<\infty$ there is a positive   classical $K_+$ ground state solution $ u $ of (\ref{eq_fir}).  
   
   \item[2-b]  Moreover, if $a$ is a radial function then  for 
   \[ \frac{4(N+2)}{\beta_0(\Omega)} < p < \infty,\]
   the ground state solution $ u $ in $\it{2}$-$\it{a}$ in nonradial.   
   \end{itemize}
    \end{enumerate} 
    
\end{thm}

 We shall make use of Theorem \ref{var-pri} to prove the above result.  In that regard,  we shall need to verify two conditions in Theorem \ref{var-pri},  namely,  the compact embedding and the point wise invariance property.

 \begin{prop} \label{imbed_pi4_ann} ($ \frac{\pi}{4}$- annular domain imbeddings) Suppose  $n=m=\frac{N}{2}$.
 
 \begin{enumerate}  
 \item ($K_-$ imbedding) Suppose $\Omega$ is $\frac{\pi}{4}$-annular domain with monotonicity and 
   \[ 1 \le p <  \frac{4(N+1)}{N-2}.\] 
  Then $K_- \subset \subset L^p(\Omega)$.  

  \item ($K_+$ imbedding) Suppose $\Omega$ is an annulus in $\IR^N$ and  $ 1 \le p< \infty$.   Then $K_+ \subset \subset L^p(\Omega)$. 
\end{enumerate} 
\end{prop}

\begin{proof} 
Part 1:    The proof used in the proof of Theorem A carries over to this case. \\
Part 2: If we take $u \in K_+$ note that the function is largest at $ \theta=\frac{\pi}{4}$.  So note the problems appears to be a genuine two dimensional problem near $ \theta=\frac{\pi}{4}$ and hence we expect to have imbeddings for all $ p$, see Remark \ref{remark_imbed_mon} for related comments.   For concreteness we work on the annulus centered at the origin with inner radius $1$ and outer radius $ 2$. 

Then note for $ 0 \le  u \in H^1_{0,G}(\Omega)$ (which are also even about $ \theta=\frac{\pi}{4}$ but may not have any monotonicity) we have 
\[ \int_\Omega u(x)^p dx= \int_1^2 \int_0^\frac{\pi}{4} u(r,\theta)^p r^{2n-1} \cos^{n-1}(\theta) \sin^{n-1}(\theta) d \theta dr,\] and 
\[ \int_\Omega | \nabla u(x)|^2 dx = \int_1^2 \int_{0}^\frac{\pi}{4}  \left\{ u_r^2 + \frac{u_\theta^2}{r^2} \right\} r^{2n-1} \cos^{n-1}(\theta) \sin^{n-1}(\theta) d \theta d r.\]  For any $ 1 \le p<\infty$ there is some $C_p>0$ (independent of $u$ as above) such that 

\begin{equation} \label{first_est_1}
\left\{\int_1^2 \int_\frac{\pi}{8}^\frac{\pi}{4} u(r,\theta)^p r^{2n-1} \cos^{n-1}(\theta) \sin^{n-1}(\theta) d \theta dr \right\}^\frac{2}{p},
\end{equation} is bounded above by 
\[C_p\int_1^2 \int_{\frac{\pi}{8}}^\frac{\pi}{4}  \left\{ u_r^2 + \frac{u_\theta^2}{r^2} \right\} r^{2n-1} \cos^{n-1}(\theta) \sin^{n-1}(\theta) d \theta d r.\]  The two important points are that the integrals are over $ 1 <r<2$ and $ \frac{\pi}{8}<\theta< \frac{\pi}{4}$.    Note on this range of $ \theta$ and $ r$  the measure $ d \mu(r,\theta)=  r^{2n-1} \cos^{n-1}(\theta) \sin^{n-1}(\theta) d \theta d r$ is essentially two dimensional, ie. comparable to $ d \theta dr$.  This allows one to use the two dimensional Sobolev imbedding.  To see this more rigously one can consider working on $(r,\theta) \in (1,2) \times (\frac{\pi}{8}, \frac{\pi}{4})$ and hence we can consider the Sobolev imbeddings in the product space.  Let $ u \in K_+$ and then note that 
\begin{eqnarray*}
\int_1^2 \int_0^\frac{\pi}{8} u(r,\theta)^p r^{2n-1} \cos^{n-1}(\theta) \sin^{n-1}(\theta) d \theta d r & \le & \int_1^2 \int_0^\frac{\pi}{8} u(r,\theta+ \frac{\pi}{8})^p r^{2n-1} \cos^{n-1}(\theta) \sin^{n-1}(\theta) d \theta d r \\ 
& \le &  \int_1^2 \int_\frac{\pi}{8}^\frac{\pi}{4} u(r,\hat{\theta})^p r^{2n-1}  d \hat{\theta} d r
\end{eqnarray*} where in the first line we used the monotonicity of $u$.  Note this final quantity is bounded above by the $\frac{p}{2}$ power of (\ref{first_est_1}). We can now combine the results which completes the proof of part 2.   
\end{proof}


     The following theorem develops pointwise invariance property (see Theorem \ref{var-pri} part (ii)) which is related to  the linear problem
     \begin{equation} \label{linear_new_cut-off}
\left\{\begin{array}{ll}
-\Delta v = a(x) u^{p-1} &  \mbox{ in } \Omega, \\
v= 0 &   \mbox{ on } \pOm.
\end {array}\right.
\end{equation}   
    
 \begin{prop} \label{pointwise_inv} (Pointwise invariance property; case $m=n=\frac{N}{2}$). Suppose $ a$ is nonnegative with $ a=a(s,t)$, $ a_\theta = sa_t-t a_s$ is bounded and $a$ is even across $ \theta= \frac{\pi}{4}$.   
 \begin{enumerate}
 
     \item Suppose $\Omega$ is $\frac{\pi}{4}$-annular domain with monotonicity and $ a_\theta \le 0$ in $ \widetilde{\Omega}_0$.    If  $u \in K_-$ and $v$ satisfies (\ref{linear_new_cut-off}) then $ v \in K_-$. 
     
     \item Suppose $\Omega$ is an annulus and $ a_\theta \ge 0 $ in $ \widetilde{\Omega}_0$.  If $ u \in K_+$ and $v$ satisfies (\ref{linear_new_cut-off}) then $ v \in K_+$. 
 \end{enumerate}
 \end{prop}
 
 \begin{proof}  Much of the proof won't depend on which case we are in.  Additionally we have $m=n$ but for the time being we won't indicate this since many of these computations will be useful in later cases where they are not equal. Let $ u \in K_\pm$ and for $k$ large consider $u_k(x)=\min \{ u(x),k\}$ and note  that $u_k \in K_\pm$.    Let $ v^k$ denote a solution of 
\begin{equation} \label{linear_new_cut-off_proof}
\left\{\begin{array}{ll}
-\Delta v = a(x) u_k^{p-1} &  \mbox{ in } \Omega, \\
v= 0 &   \mbox{ on } \pOm.
\end {array}\right.
\end{equation} 

By elliptic regularity we have $ v^k \in H^1_{0,G}(\Omega) \cap C^{1,\alpha}(\overline{\Omega})$ for any $ 0<\alpha<1$.  In terms of $(s,t)$ we see that $v^k$ satisfies 
\begin{equation} \label{s_t_equa}
-v_{ss}^k - v_{tt}^k - \frac{(m-1) v_s^k}{s}- \frac{(n-1) v_t^k}{t} = a u_k^{p-1} \quad \mbox{ in } \widehat{\Omega},
\end{equation}
with $v^k=0$ on $ (s,t) \in \partial \widehat{\Omega} \backslash ( \{s=0\} \cup \{t=0\} )$.  Since $ v^k$ is sufficiently smooth then 
 $ v^k_s =0 $ on $\partial \widehat{\Omega} \cap \{s=0\}$ and $v^k_t=0$ on $ \partial \widehat{\Omega} \cap \{t=0\}$ after considering the symmetry properties of $v^k$ (see \cite{orgin_an} for details).   We now want to show that $v^k$ has the added symmetry across the line $ t=s$.  Here there are a few ways to argue.  We can directly use the $(s,t)$ coordinates or we can switch to polar coordinates, we will use the second approach.  \\

 A computation shows that 
\[ \frac{(m-1) v_s^k}{s}+ \frac{(n-1) v_t^k}{t} = \frac{(N-2) v_r^k}{r} + \frac{ v_\theta^k}{r^2} \left\{ \frac{n-1}{\tan(\theta) }- (m-1) \tan(\theta) \right\}\] if we write the equation in terms of polar coordinates (recall we have $ s=r\cos(\theta), t = r \sin(\theta)$).   \\ 
Writing out (\ref{s_t_equa}) in polar coordinates gives 
\begin{equation} \label{polar_k}
    -v_{rr}^k -\frac{(N-1) v_r^k}{r}- \frac{v_{\theta \theta}^k}{r^2} + \frac{v_\theta^k}{r^2} h(\theta) = a u_k^{p-1}, \; \mbox{ in } \widetilde{\Omega},
    \end{equation}
    with $ v^k=0$ on $ \partial \widetilde{\Omega} \backslash \left( \Gamma_L \cup \Gamma_R \right) $ where $\Gamma_L$ (respectively $\Gamma_R$) corresponds to the portion of $\partial \widetilde{\Omega}_0$ given by $ \{\theta=0\}$ (respectively $ \{ \theta=\frac{\pi}{2}\}$) and where $v^k_\theta=0$ on $ \Gamma_L \cup \Gamma_R$ and where
    \begin{equation}  \label{defn_h}
    h(\theta)= (m-1) \tan(\theta) - \frac{(n-1)}{\tan(\theta)}.
    \end{equation} 
  We now show that $v^k$ is even across $ \theta=\frac{\pi}{4}$; so we set $\widehat{v}(r,\theta)= v^k(r, \frac{\pi}{2}- \theta)$ and we want to show that $ \widehat{v}=v^k$ in $ \widetilde{\Omega}$.    Because of the smoothness of $v^k$ we have  $ \partial_\theta v^k=0$ at $ \theta=0,\frac{\pi}{2}$ and hence we have the same for $ \widehat{v}$.  Also note that since $m=n$ we have $h$ is odd across $ \theta= \frac{\pi}{4}$, ie. 
\[ h(\theta)= -h( \frac{\pi}{2}-\theta) \] for $ 0<\theta<\frac{\pi}{2}$.  Note the right hand side of (\ref{polar_k}) is even across $ \theta=\frac{\pi}{4}$.   From this we see $\widehat{v}$ satisfies (\ref{polar_k}) with the same boundary conditions and hence by uniqueness of solution we have $ \widehat{v}=v^k$ in $ \widetilde{\Omega}$.   Now since $v^k$ is even across $ \theta=\frac{\pi}{4}$ and $v^k$ is sufficiently smooth we have $v^k_\theta=0$ on $ \theta=\frac{\pi}{4}$.  \\

\noindent
\textbf{Monotonicity.}  Let $ w=v_\theta^k$ and then note that if we take a derivative in $\theta$ of the equation for $ v^k$ we arrive at 
\begin{equation} \label{eq_w}
-w_{rr} -\frac{(N-1) w_r}{r}- \frac{w_{\theta \theta}}{r^2} + \frac{w_\theta}{r^2} h(\theta) + \frac{w}{r^2} h'(\theta)= \partial_\theta \left\{a u_k^{p-1} \right\}, \mbox{ in } \widetilde{\Omega},
\end{equation} and in particular the equation is satisfied in $ \widetilde{\Omega}_0$ with $ w=0$ on the portion of $ \partial \widetilde{\Omega}_0$ corresponding to $ \theta=0, \frac{\pi}{4}$.   A computation shows that if write the left hand side of  (\ref{eq_w}) in terms of $x$ we arrive at 
\[ -\Delta w(x) + \frac{(n-1) |x|^2 w(x)}{ (x_1^2 + \cdot \cdot \cdot + x_m^2)( x_{m+1}^2 + \cdot \cdot \cdot+x_N^2)},\] which, at least formally, satisfies a maximum principle.  \\ 

We now separate the cases of $u \in K_-$ and $u \in K_+$.  Suppose $ u \in K_+$ and $ \Omega$ an annulus.  Then $ w=0$ on the curved portions of $ \widetilde{\Omega}_0$ since $ v^k=0$ on these portions of the boundary.    Also note that the right hand side of (\ref{eq_w}) is nonnegative and assuming we can apply the maximum principle we arrive at $ w \ge 0$ in $ \widetilde{\Omega}_0$.  \\ 

We now suppose $u \in K_-$.  Then we have the right hand side of (\ref{eq_w}) is nonpositive.  Since $ v^k \ge 0$ in $ \widetilde{\Omega}_0$ and noting the monotnicity of $g_1$ and $g_2$ we see that $ w=v^k_\theta \le 0$ on the curved portions of $ \partial \widetilde{\Omega}_0$ and again if we can apply the maximum principle we arrive at $ w \le 0$ in $ \widetilde{\Omega}_0$.  \\




To make these maximum principle arguments used above rigorous we use the idea of \cite{Weth_annulus}  (see also \cite{orgin_an}).  Consider the case of $ u \in K_-$.   Let $ \E>0$ be small and consider $ \psi:=(w-\E)_+$.  By the smoothness properties of $v^k$ and noting the boundary values of $v^k$ we have $ w=0$ near $ \theta=0$ and $ \theta=\frac{\pi}{4}$.   Using $ \psi$ as a test function on a suitable weak notion of a solution of (\ref{eq_w}) one will arrive at $ \psi=0$ and sine $ \E>0$ is arbitrary we have $ w \le 0$.\\

\noindent
\textbf{Sending $k \rightarrow \infty$.}  We now get bounds on $v^k$ which allow us to pass to the limit in $k$.  We assume that $ u \in K_-$ and $v^k$ as above.  Then testing the weak formulation for $v^k$ on $v^k$ gives 
\begin{eqnarray*}
\int_\Omega | \nabla v^k |^2 dx &=& \int_\Omega a u_k^{p-1} v^k dx \\
& \le & C  \| u_k^{p-1} \|_{L^{p'}} \| v^k \|_{L^p}  \\
 & \le & C_0 \|u_k \|_{L^{p'(p-1)}}^{p-1} \| \nabla v^k \|_{L^2} \\ 
 & =& C_0 \| u_k\|_{L^p}^{p-1}  \| \nabla v^k \|_{L^2}
 \end{eqnarray*} where the second last inequality follows by part 1 of by Proposition \ref{imbed_pi4_ann} after noting the restriction on $p$ and the final equality follows since $p'(p-1)=p$.  Using the imbedding again we arrive at 
 \[ \| \nabla v^k \|_{L^2} \le C_0 \| u_k \|_{L^p}^{p-1} \le C_1 \| \nabla u_k \|_{L^2}^{p-1},\] since $u_k \in K_-$, and now note  this quantity on the right is bounded independently of $k$ and hence $v^k$ is bounded in $H_{0,G}^1(\Omega)$ and after passing to a subsequence we can assume that there is some $ v \in H^1_{0,G}(\Omega)$ such that $ v^k \rightharpoonup v$ in $H^1_{0,G}(\Omega)$ and its clear that $v$ is an $H^1_{0,G}(\Omega)$ solution of (\ref{linear_new_cut-off}).  Also note that $ u_k^{p-1} \rightarrow u^{p-1} $ in $L^{p'}(\Omega)$ and hence by passing to another subsequence we have $ v^k \rightarrow v$ in $ W^{2,p'}(\Omega)$ and hence we can assume $ \nabla v^k \rightarrow \nabla v$ in $L^{p'}(\Omega)$ and a.e. in $ \Omega$.     We now suppose that $ 0 \le \psi \in C_c^\infty( \widetilde{\Omega}_0)$ and note that we have 
 \[ 0 \ge \int_{\widetilde{\Omega}_0} v_\theta^k \psi dr d \theta = - \int_{\widetilde{\Omega}_0} v^k \psi_\theta dr d \theta,\]   and noting that $ v^k \rightarrow v$ in $L^2_{loc}( \widetilde{\Omega}_0, dr d \theta)$ (recall we are away from the origin in this problem and the measures only have issues on $ \theta=0, \frac{\pi}{2})$) and hence we can pass to the limit here to see that 
 $0 \ge \int_{\widetilde{\Omega}_0} v \psi_\theta dr d \theta$  but this is sufficient to see that $ v_\theta \ge 0$ a.e. in $ \widetilde{\Omega}_0$. 

 The case of $ u \in K_+$ has a similar proof and we skip the details.  
\end{proof}

   \noindent
   \textbf{Proof of Theorem \ref{theorem-nonlinear-annular}.}    We are going to use Theorem  \ref{var-pri} for the proof. 
   Note that conditions $(i)$ and $(ii)$ in Theorem \ref{var-pri} follows from Propositions \ref{imbed_pi4_ann} and \ref{pointwise_inv} respectively. This proves the existence of a weak solution of (\ref{eq_fir}) for both cases $(\it 1)$ and ($\it 2$-$a$).   It also follows from Theorem \ref{nonradial}
   that for 
   \[ \frac{4(N+2)}{\beta_0(\Omega)} < p < \infty,\]
   the ground state solution $ u $ in $\it{2}$-$\it{a}$ is non-radial. \\

    \noindent
{\it Regularity of the solution.}  We will prove the case of part 1, the case of part 2 is easier since one doesn't need an iteration.   Let $ q:=\frac{2(n+1)}{n-1}$ and take $t_0=1$ and 
\[ t_{k+1}:= \frac{q t_k}{2} - \frac{p-2}{2},\] where $ 1<p<q$.  Then by examining the cobweb we see that $t_k \rightarrow \infty$.  \\ 

We now prove the following inductive step.  If $k \ge 0$ and $ u^{t_k} \in K_-$ then $ u^{t_{k+1}} \in K_-$.   Assuming this is true for a moment then note we see that since $ u^{t_0}=u \in K_-$ we can iterate to see $u \in L^T(\Omega)$ for all $T<\infty$ and hence we see that $u$ is $C^{1,\delta}( \overline{\Omega})$ and then we can proceed with the Schauder regularity theory and the exact smoothness of $u$ will depend on the smoothness of $a$.  Assuming $a$ at least H\"older continuous we have $u$ is a classical solution.  \\ 

We now prove the iteration step.  Suppose $u^{t_k} \in K_-$ for some $k \ge 0$ and for $i$ a large integer   define 
\begin{equation} 
\phi(x)=\left\{\begin{array}{ll}
u(x)^{2t_k-1}  &  \mbox{ if } u(x)<i, \\
i^{2t_k-1}     &   \mbox{ if }u(x) \ge i. \\
\end {array}\right.
\end{equation} Note that $ \phi \in K_-$.  We can test (\ref{eq_fir}) on $ \phi$ to arrive at (here $\Omega_i:=\{x \in \Omega: u(x)<i \}$) 
\begin{eqnarray*}
\frac{(2t_{k+1}-1)}{t_{k+1}^2} \int_{\Omega_i} | \nabla u^{t_{k+1}}|^2 dx &=& \int_\Omega a u^{p-1} \phi dx \\
&=& \int_{\Omega_i} a u^{p+2 t_{k+1}-2} dx + \E_{k,i} \\
&=& \int_{\Omega_i} a \left( u^{t_k} \right)^\frac{p+2t_{k+1}-2}{t_k} dx + \E_{k,i} \\
&=& \int_{\Omega_i} a \left( u^{t_k} \right)^qdx + \E_{k,i} \\
& \le & \int_{\Omega} a \left( u^{t_k} \right)^q dx + \E_{k,i},
\end{eqnarray*} where 
\[ \E_{k,i}:= i^{2t_{k+1}-1} \int_{\Omega \backslash \Omega_i} a u^{p-1}dx = i^{qt_k -(p-1)}\int_{\Omega \backslash \Omega_i} a u^{p-1}dx.\]  First note since $ u^{t_k} \in K_-$ then we see the $ u^{t_k} \subset L^q(\Omega)$ by the imbedding and hence the integral on the right is finite. Set $C_k=\int_\Omega u^{t_k q} dx$ and note we have 
\[ i^{t_k q} | \Omega \backslash \Omega_i| \le C_k,\] for all large $i$.   Put $ \delta_{k,i}:=\int_{\Omega \backslash \Omega_i} u^{t_k q} dx$ and note $ \delta_{k,i} \rightarrow 0$ as $ i \rightarrow \infty$.  Let $(p-1) \tau = t_k q$ and then note 
\begin{eqnarray*}
\frac{\E_{k,i}}{i^{q t_k - (p-1)}} &= &\int_{\Omega \backslash \Omega_i} a u^{p-1} dx \\ 
& \le & C_a \left( \int_{\Omega \backslash \Omega_i} u^{t_k q} dx \right)^\frac{1}{\tau} | \Omega \backslash \Omega_i|^\frac{1}{\tau'} \\
\end{eqnarray*} so we have 
\begin{eqnarray*}
\frac{\E_{k,i}^{\tau'}}{i^{(q t_k - (p-1))\tau'}} & \le  & C_a^{\tau'} \delta_{k,i}^\frac{\tau'}{\tau} \frac{C_k}{i^{t_k q}}
\end{eqnarray*} which gives us 
\[ \E_{k,i}^{\tau'} \le C_a^{\tau'} \delta_{k,i}^\frac{\tau'}{\tau} C_k \rightarrow 0,\] as $ i \rightarrow \infty$. From this we see that 
\[ \frac{(2 t_{k+1}-1)}{t_{k+1}^2} \int_\Omega | \nabla u^{t_{k+1}}|^2 dx \le \int_\Omega a (u^{t_k})^q dx < \infty, \] and hence we see that $u^{t_k+1} \in H^1_{0,G}(\Omega)$ and its clear the monotonicity and symmetry is sufficient that $ u^{t_{k+1}} \in K_-$.  \hfill $\Box$

 \section{H\'enon equation on $B_1$ in even dimensions} \label{henon_ball}
 
 
 In this section we examine the H\'enon equation given by 
 
 \begin{equation} \label{eq_hen}
\left\{\begin{array}{ll}
-\Delta u = |x|^\alpha u^{p-1} &  \mbox{ in } B_1, \\
u= 0 &   \mbox{ on } \partial B_1,
\end {array}\right.
\end{equation}
 where $B_1$ is the unit ball in $ \IR^N$ centered at the origin and $N \ge 3$ and $ \alpha>0$.  Our interest is in obtaining positive classical nonradial solutions in the supercritical case  
 \[ \frac{2N}{N-2}<p< \frac{2N+2 \alpha}{N-2},\] via our variational approach. In the radial case the weight improves compactness of the Sobolev imbedding to $H^1_{0,rad}(B_1) \subset \subset L^p(B_1, |x|^\alpha dx)$ to $1 \le p<\frac{2N+2\alpha}{N-2}$, see \cite{Ni} and this allows one to obtain a positive radial solution for this range of $p$.   The first work to obtain a nonradial solution was in \cite{h1} in the subcritical case.  This was later extended to other values of $p$ in \cite{h2,h3,h4,h5,h6}.  Many of these works used bifurcation approaches to show the existence of nonradial solutions. \\

 We now define $K_+$ in essentially the same way we did on the annulus; 
 \[ K_+=\left\{ 0 \le u \in H^1_{0,G}(B_1): u \mbox{ is even in $\theta$ across $ \theta= \frac{\pi}{4}$ with $ u_\theta \ge 0$ for $ 0<r<1$ and $ 0<\theta<\frac{\pi}{4}$ } \right\},\] and we define $ \widetilde{\Omega}$ and $ \widetilde{\Omega}_0$ in the obvious way after considering the definitions in Section \ref{double_dom}.   We will not consider working on $K_-$ here even though it would give a different type of solution as compared to $K_+$,  but one would need to further restrict the upper bound on $p$ and so we chose not to include this. 

Here is our main theorem in this section. 
 \begin{thm} \label{non-radial-henon} ($K_+$ solutions for H\'enon equation).  Let  $B_1$ is the unit ball in $ \IR^N$ centered at the origin, $N \ge 4$ is even and $ \alpha>0$. The following assertions hold:
 
 \begin{enumerate}   \item Suppose $ 2<p< \frac{2N+2\alpha}{N-2}$.  Then there is a positive classical $K_+$ ground state solution $u $ of (\ref{eq_hen}).    
 
 \item Suppose \[\frac{16(N+2)}{(N-2)^2}+2 < p< \frac{2N+2\alpha}{N-2}. \]  Then the positive classical $K_+$ ground state solution $u$ of (\ref{eq_hen}) is nonradial.

 \end{enumerate} 
 \end{thm}

 \begin{remark} Note all these results can immediately give results regarding fast decay solutions of related problems on exterior domains after applying a Kelvin transform. 
 \end{remark} 
 
 We shall need some preliminaries before proving this theorem. 
\begin{prop} \label{henon_imbed} (Imbedding iteration) Let $ \Omega$ denote a bounded domain of double revolution in $ \IR^N= \IR^{m+n}$ (here $m$ and $n$ need not be equal). For all integers $k \ge 0$ there is some $C_k \ge 0$ such that for all $ 0 \le \phi \in H^1_{0,G}(\Omega)$ with $ \| \nabla \phi \|_{L^2(\Omega)}=1$ we have 
\begin{equation} \label{k_thin}
\int_{\widehat{\Omega}} ( \phi(s,t))^{2^*+2k} s^{(k+1)(m-1)} t^{(k+1)(n-1)} ds dt \le C_k,
\end{equation} where $2^*=2_N^*=\frac{2N}{N-2}$. 
 \end{prop}

 \begin{proof} Take $ 0 \le u \in H^1_{0,G}(\Omega)$ (and say Lipschitz)  and take $ \beta_i>0$.   By extending $ u$ to the full first quadrant by extending it to be zero outside of $ \widehat{\Omega}$ we have 
 \[ u(s,t) \le \int_s^\infty | \nabla_{s,t} u( \tau_1, t)| d \tau_1,\]  
  \[ u(s,t) \le \int_t^\infty | \nabla_{s,t} u( s, \tau_2)| d \tau_2,\] and hence we have 
 \[ u(s,t)^2\le \int_s^\infty | \nabla_{s,t} u( \tau_1, t)| d \tau_1 \int_t^\infty | \nabla_{s,t} u( s, \tau_2)| d \tau_2,\] 
   and we now multiply by sides by $ s^{2 \beta_1} t^{2 \beta_2}$ where $ \beta_i>0$ and integrate over $ \widehat{\Omega}$ we arrive at 
  \begin{equation} \label{step_gen}
  \int_{\widehat{\Omega}} u(s,t)^2 s^{2\beta_1} t^{2\beta_2} ds dt \le \left(  \int_{\widehat{\Omega}} |\nabla_{s,t} u(s,t)| s^{\beta_1} t^{\beta_2} ds dt \right)^2.
  \end{equation}  We now suppose $ 0 \le \phi \in H^1_{0,G}(\Omega)$ is smooth and with the gradient assumption as in the hypothesis and we put $u=\phi^\gamma$ into (\ref{step_gen}) where $ \gamma \ge 1$.   Then we arrive at 
  \begin{eqnarray*}
  \int_{\widehat{\Omega}} \phi^{2\gamma} s^{2 \beta_1} t^{2 \beta_2} ds dt &\le &\gamma^2 \left( \int_{\widehat{\Omega}} \left\{| \nabla_{s,t} \phi| s^\frac{m-1}{2} t^\frac{n-1}{2} \right\} \left\{\phi^{\gamma-1} s^{\beta_1 - \frac{m-1}{2}} t^{\beta_2- \frac{n-1}{2}} \right\} ds dt \right)^2 \\
  &\le& \gamma^2 \| \nabla \phi \|_{L^2(\Omega)}^2 \int_{\widehat{\Omega}} \phi^{2(\gamma-1)} s^{2 \beta_1 - (m-1)} t^{2 \beta_2 - (n-1)} ds dt,
  \end{eqnarray*} where we performed the Cauchy–Schwarz inequality
 and recall $ \| \nabla \phi \|_{L^2}=1$.  We will now use this inequality to perform an iteration in $ \gamma$ and $ \beta_i$.  For $k \ge 0$ define 
  \[ \gamma_k = \frac{2^*}{2}+k, \quad \beta_1^k = \frac{(k+1) (m-1)}{2}, \quad \beta^k_2= \frac{(k+1)(n-1)}{2}.\]   Now suppose for $k \ge 1$ we have 
  \[ \int_{\widehat{\Omega}} \phi^{2 \gamma_{k-1}} s^{2 \beta_1^{k-1}} t^{2\beta_2^{k-1}} ds dt =C_k,\]  then by putting $ \gamma=\gamma_k$ and $\beta_1= \beta_1^k, \beta_2= \beta_2^k$  into (\ref{step_gen}) we arrive at 
  \begin{eqnarray*}
  \int_{\widehat{\Omega}} \phi^{2 \gamma_k} s^{2 \beta_1^k} t^{2 \beta_2^k} ds dt & \le & \gamma_k^2 \int_{\widehat{\Omega}} \phi^{2(\gamma_k-1)} s^{2 \beta_1^k-(m-1)} t^{2 \beta_2^k-(n-1)} ds dt \\
  &=& \gamma_k^2 \int_{\widehat{\Omega}} \phi^{2 \gamma_{k-1}} s^{2 \beta_1^{k-1}} t^{2 \beta_2^{k-1}} ds dt \\
  &=& \gamma_k^2 C_k,
  \end{eqnarray*} after noting 
  \[ 2(\gamma_k-1)=2\gamma_{k-1}, \quad 2 \beta_1^{k-1} = 2\beta_1^k-(m-1), \quad 2 \beta_2^{k-1}=2 \beta_2^k -(n-1).\]   Also note we can start the iteration since the first term is given by 
  \[ \int_{\widehat{\Omega}} \phi^{2^*} s^{m-1} t^{n-1} ds dt,\] which is controlled by $ \| \nabla \phi \|_{L^2(\Omega)}$ by the classical critical Sobolev imbedding theorem. 
  \end{proof}

 \begin{coro} \label{imbed_+} Let $ m=n=\frac{N}{2}$ and suppose we have $ 1 <p  \le \frac{2N+2\alpha}{N-2}$ where $ \alpha>0$.
  Then we have $K_+ \subset L^p(B_1, |x|^\alpha dx)$ (ie. a continuous imbedding).  
  \end{coro}
    Note the imbedding if optimal after considering the radial imbedding. \\
  
\noindent 
\textbf{Proof of Corollary \ref{imbed_+}.}  We first prove the result for the case of $p=  2^* + 2k$ for some positive integer $k$ and we suppose $ \alpha$ satisfies the hypothesis.  Let $ \phi \in K_+$ and we suppose $ \| \nabla \phi \|_{L^2(\Omega)}=1$.   By the symmetry of the function it is sufficient we bound the desired integral on $ \{(s,t) \in \widehat{\Omega}: s >t \}$  which in polar coordinates corresponds to $ \{(\theta,r): 0<\theta< \frac{\pi}{4}, 0<r<1 \}$.  Since  $ \phi \in K_+$ we have 
  \[ \phi(r,\theta) \le \frac{16}{\pi} \int_\frac{3 \pi}{16}^\frac{\pi}{4} \phi(r, \hat{\theta}) d \hat{\theta} \quad \mbox{ for } 0<\theta<\frac{\pi}{8},\] and by Jensen's inequality we have (for $p=2^*+2k$) 
     \[ \phi(r,\theta)^p \le \frac{16}{\pi} \int_\frac{3 \pi}{16}^\frac{\pi}{4} \phi(r, \hat{\theta})^p d \hat{\theta}.\]  Then note if we write out the $L^p(B_1, |x|^\alpha dx)$ norm of $ \phi$ over the region corresponding to $ 0<\theta<\frac{\pi}{8}$ we arrive at (note the extra power of $r$ is from $ds dt= r dr d \theta$)
     \[ \int_0^1 \int_0^\frac{\pi}{8} \phi(r,\theta)^p r^\alpha r^{2(n-1)} r \cos^{n-1}(\theta) \sin^{n-1}(\theta)  d \theta dr \]but this is bounded above by 
     \[ \frac{16}{\pi} \int_0^1 \int_\frac{3 \pi}{16}^\frac{\pi}{4} \phi(r, \hat{\theta})^p d \hat{\theta} r^{\alpha +2(n-1) +1} dr  \int_0^\frac{\pi}{8} \cos^{n-1}(\theta) \sin^{n-1}(\theta) d \theta,\] and hence we just need to control  
     \begin{equation} \label{need_con}
     \int_0^1 \int_\frac{3 \pi}{16}^\frac{\pi}{4} \phi(r, \theta)^p r^{\alpha+2(n-1)+1} d \theta dr.
     \end{equation} 
     
     We now show for all $ 0<\theta_0< \frac{\pi}{4}$ we can control 
     \begin{equation} \label{want}
     \int_0^1 \int_{\theta_0}^\frac{\pi}{4} \phi(r, \theta)^p r^{\alpha+2(n-1)+1} d \theta dr.
     \end{equation}

     We now write out (\ref{k_thin}) in terms of polar coordinates and noting the sine and cosine terms don't play a role now we see (\ref{k_thin}) gives the existence of some $D_k(\theta_0)>0$ such that 
     \begin{equation} \label{have_10}
     \int_0^1 \int_{\theta_0}^\frac{\pi}{4} \phi(r, \theta)^p r^{2(k+1)(n-1)+1}  d \theta dr \le D_k.
     \end{equation} 
     Note the assumption on $ \alpha$ is exactly $ 2k(n-1) \le \alpha$ and this gives us that 
     \[\alpha+2(n-1)+1 \ge 2(k+1)(n-1)+1,\] and hence we get the desired result for the case of $p=2^*+2k$. \\
     

     We now prove the result for general $p$.   Let $ p$ and $ \alpha $ satisfy the hypothesis and we assume $p>2^*$.   First note that this assumption on $p$ implies $ p(n-1)-2n>0$.  Since $ p \le \frac{2N+2 \alpha}{N-2}$ we have $ \alpha \ge p(n-1)-2n$.  Define $ \alpha_p=p(n-1)-2n$ and hence  $ p=\frac{2N+2 \alpha_p}{N-2}$ so $ p-2^*= \frac{\alpha_p}{n-1}$.   Pick $k$ large integer such that  $ 2^* +2k >p$ and $ 2k(n-1)>\alpha_p$.   Then set 
      \[ \beta_k= \frac{(2^*+2k) \alpha_p}{2k(n-1)} \] and note $ \beta_k<p$ for large $k$.
       Set 
    $t_k=\frac{2^*+2k}{\beta_k}>1$ for large $k$.  Then note we have $ t_k'(p-\beta_k)=2^*$.   Hence we have 
    \begin{eqnarray*}
    \int_\Omega  \phi(x)^p |x|^{ \alpha_p} dx &=& \int_\Omega \phi^{\beta_k} |x|^{\alpha_p} \phi^{p-\beta_k} dx \\ 
    & \le & \left( \int_\Omega \phi^{\beta_k t_k} |x|^{\alpha_p t_k} dx \right)^\frac{1}{t_k} \left( \int_\Omega \phi^{ t_k' (p-\beta_k)} dx \right)^\frac{1}{t_k'} \\
    &=& \left( \int_\Omega \phi^{2^*+2k} |x|^{2k(n-1)} dx \right)^\frac{1}{t_k} \left( \int_\Omega \phi^{2^*} dx \right)^\frac{1}{t_k'}
    \end{eqnarray*} and this gives us the desired bound at least in the case of $ \alpha_p$.  Noting that $ \alpha \ge \alpha_p$ gives the desired result. 
\hfill $ \Box$  

\begin{prop} \label{henon_prop_point}(Pointwise invariance for the H\'enon equation) Suppose $N$ is even with $ 2m=2n=N$,  $u \in K_+$ and $v$ solves  \begin{equation} \label{eq_hen_pointwise}
\left\{\begin{array}{ll}
-\Delta v = |x|^\alpha u^{p-1} &  \mbox{ in } B_1, \\
v= 0 &   \mbox{ on } \partial B_1.
\end {array}\right.
\end{equation} Then $ v \in K_+$. 
\end{prop} 

\begin{proof} The proof that proved the analagous result on an annulus works in this case also (the main difference is one needs to take some care near the origin now).  In this proof we will write $ \widehat{\Omega}, \widetilde{\Omega}, \widetilde{\Omega}_0$ even though its understood that $\Omega=B_1$. Let $ u \in K_+$ and we perform the cut off as always $ u_k(x) = \min \{ u(x),k \}$ and we let $v^k$ denote a solution of  
\begin{equation} \label{eq_hen_pointwise_cut}
\left\{\begin{array}{ll}
-\Delta v^k = |x|^\alpha u_k^{p-1} &  \mbox{ in } B_1, \\
v^k= 0 &   \mbox{ on } \partial B_1.
\end {array}\right.
\end{equation}

Writing this in term of polar coordinates gives 
\begin{equation} \label{polar_k_hen}
    -v_{rr}^k -\frac{(N-1) v_r^k}{r}- \frac{v_{\theta \theta}^k}{r^2} + \frac{v_\theta^k}{r^2} h(\theta) = r^\alpha u_k^{p-1}=G(r,\theta), \; \mbox{ in } \widetilde{\Omega},
    \end{equation} where $h$ is defined as in (\ref{defn_h}) with $m=n$ and note that $G$ is even across $ \theta=\frac{\pi}{4}$ after noting the conditions on $u$.  Using the symmetry of $v^k$ one sees, as in the case of the annulus, that $v^k_\theta=0$ on $ \theta=0,\frac{\pi}{2}$ provided one stays away from the origin.   As in the case of the annulus we consider $ \widehat{v}(r,\theta)= v^k(r, \frac{\pi}{2}-\theta)$ and as before $ \widehat{v}$ also satisfies (\ref{polar_k_hen}) with the same boundary conditions as $v^k$.    Set $ \widehat{v}(x)$ to be $ \widehat{v}(r,\theta)$ written in terms of $ x$ and we set $W(x)=v^k(x) - \widehat{v}(x)$.  Then note $ \Delta W(x)=0$ in $B_1 \backslash \{0\}$ with $ W=0$ on $ \partial B_1$ and since we are assuming the dimension $N \ge 3$ we can use the regularity of $W$ to see that $W=0$ and hence we have $ v^k$ is even across $ \theta= \frac{\pi}{4}$ and hence we have $ v^k_\theta(r, \frac{\pi}{4})=0$ for $ 0<r \le 1$.   \\

\noindent
\textbf{Monotonicity.}  Let $ w=v_\theta^k$ and then note that if we take a derivative in $\theta$ of the equation for $ v^k$ we arrive at 
\begin{equation} \label{eq_w_hen}
-w_{rr} -\frac{(N-1) w_r}{r}- \frac{w_{\theta \theta}}{r^2} + \frac{w_\theta}{r^2} h(\theta) + \frac{w}{r^2} h'(\theta)= \partial_\theta \left\{r^\alpha u_k^{p-1} \right\}, \mbox{ in } \widetilde{\Omega},
\end{equation} and in particular the equation is satisfied in $ \widetilde{\Omega}_0$ with $ w=0$ on the portion of $ \partial \widetilde{\Omega}_0$ corresponding to $ \theta=0, \frac{\pi}{4}$ and $ w=0$ on $ r=1$.   As before a computation shows that if write the left hand side of  (\ref{eq_w_hen}) in terms of $x$ we arrive at 
\[ -\Delta w(x) + \frac{(n-1) |x|^2 w(x)}{ (x_1^2 + \cdot \cdot \cdot + x_m^2)( x_{m+1}^2 + \cdot \cdot \cdot+x_N^2)},\] which, at least formally, satisfies a maximum principle.   We can now proceed as in the annulus case to show that $ w  \ge 0$ in $ \widetilde{\Omega}_0$;  the only real difference is the added singularity at the origin.   Note that $ w$ is H\"older continuous and there is some $C>0$ such that $ |w_\theta| \le C r$.  This bound allows us to proceed as before using the method of \cite{Weth_annulus}   to see that $ w \ge 0$ in $ \widetilde{\Omega}_0$. \\




\noindent 
\textbf{Sending $k \rightarrow \infty$.}  We can utilize the same arguments from the case of the annular domain  in passing to the limit in $k$ in  Theorem \ref{pointwise_inv}.
\end{proof}

 \noindent
 \textbf{Proof of Theorem \ref{non-radial-henon}.}   
 Here again, we are going to use Theorem  \ref{var-pri} for the proof. 
   Note that conditions $(i)$ and $(ii)$ in Theorem \ref{var-pri} follows from Corollary  \ref{imbed_+} and Proposition \ref{henon_prop_point} respectively. This proves the existence of a weak solution $u$ of (\ref{eq_hen}).    It also follows from Theorem \ref{nonradial}
   that for 
   \[ \frac{4(N+2)}{\beta_0(B_1)} < p-2 ,\]
   the ground state solution $ u $ obtained above is non-radial. Here $\beta_0(B_1)$  is the best constant for the Hardy inequality on $B_1$,  and in fact $\beta_0(B_1)=(N-2)^2/4.$
   Thus,  our solution $u$ is non-radial provided 
   \[\frac{16(N+2)}{(N-2)^2}+2 < p< \frac{2N+2\alpha}{N-2}. \]

\noindent
{\it Regularity of the solution.}   Set $ q:=\frac{2N+2\alpha}{N-2}$ and consider $ t_0=1$ and 
\[ t_{k+1}:= \frac{q t_k}{2} - \frac{p-2}{2},\] for $ k \ge 0$.  Since $ 1<p<q$ we have, as before,  $ t_k \rightarrow \infty$.   Let $ u \in K_+$ denote the ground state and note then we have 
\[ \int_{B_1} |x|^\alpha u^{p+2(1)-2} dx = \int_{B_1} | \nabla u|^2 dx < \infty.\]  We now prove the following iteration: 
\[ \mbox{ if } \int_{B_1} |x|^\alpha u^{p+2 t_k -2} dx=C_k<\infty  \quad \mbox{ then } \int_{B_1} | x|^\alpha u^{p+2 t_{k+1}-2} dx =D_k<\infty.\] Fix $k \ge 0$ and suppose $ C_k$ is finite and then we consider 
\begin{equation} 
\phi(x)=\left\{\begin{array}{ll}
u(x)^{2t_k-1}  &  \mbox{ if } u(x)<i, \\
i^{2t_k-1}     &   \mbox{ if }u(x) \ge i, \\
\end {array}\right.
\end{equation} for positive integers $i$.   This is a suitable test function to test the equation for $u$ on and we then arrive at 
\[ (2t_k-1) \int_{\Omega_i} u^{2t_k-2} | \nabla u|^2 dx = \int_{\Omega_i} |x|^\alpha u^{p+2 t_k-2} dx + \E_{k,i},\] where $\Omega_i:=\{x \in B_1: u(x)<i\}$ and 
\[ \E_{k,i} = \int_{\Omega_i^c} |x|^\alpha u^{p-1} i^{2 t_k-1} dx,\] where $\Omega_i^c$ is the compliment of $\Omega_i$ in $B_1$. We will later show that $\E_{k,i} \rightarrow \infty$ as $ i \rightarrow \infty$ and hence lets accept this for now.  Sending $ i \rightarrow \infty$ in the above equality we arrive at 
\[ \frac{(2t_k-1)}{t_k^2} \int_{B_1} | \nabla u^{t_k}|^2 dx = \int_{B_1} | x|^\alpha u^{p+2t_k-2} dx.\]  From this we see that $u^{t_k} \in H^1_{0,G}(B_1)$ and hence we see that $u^{t_k} \in K_+$.  We can now use the continuous imbedding to see there is some $C=C_q$ such that 
\[ \frac{(2t_k-1) C_q}{t_k^2} \left( \int_{B_1} |x|^\alpha u^{t_k q} dx \right)^\frac{2}{q} \le \int_{B_1} | x|^\alpha u^{p+2t_k-2} dx=C_k,\] but note that $ q t_k = 2t_{k+1}+p-2$ and hence we have $D_k<\infty$, which proves the inductive step.   Since we have the result for $ t_0$ we can start the iteration and hence we have $ C_k$ is finite for all $k$.  Since $ \alpha>0$ we see that after a finite number of steps that $ |x|^\alpha u^{p-1} \in L^T(B_1)$ for some $T>\frac{N}{2}$ and hence we have the solution is H\"older continuous.  We can now use Schauder regularity theory to show the solution is a classical solution.   

We now prove the claim that $\E_{k,i} \rightarrow 0$ as $ i \rightarrow \infty$. First note that since 
\[ \int_{B_1} |x|^\alpha u^{p+2t_k-2} dx =C_k < \infty,\] we have 
\begin{equation} \label{i_00}
\int_{\Omega_i^c} |x|^\alpha dx \le \frac{C_k}{i^{p+2t_k-2}}.
\end{equation}
Let $ 1<\tau<\infty$ be such that $ (p-1) \tau= p+2 t_k-2$ and then note we have 
\begin{eqnarray*}
\frac{\E_{k,i}}{i^{2t_k-1}} &=& \int_{\Omega_i^c} |x|^{p-1} dx \\ 
& \le & \left( \int_{\Omega_i^c} |x|^\alpha u^{\tau(p-1)} dx \right)^\frac{1}{\tau} \left( \int_{\Omega_i^c} |x|^\alpha dx \right)^\frac{1}{\tau'} \\
\end{eqnarray*} and put $ \delta_i:= \int_{\Omega_i^c} |x|^\alpha u^q dx $ and note $ \delta_i \rightarrow 0$.  So we can now use this and (\ref{i_00}) to see that 
\[\frac{\E_{k,i}^{\tau'}}{i^{\tau'(2t_k-1)}} \le \delta_i^\frac{\tau'}{\tau} \frac{C_k}{i^{p+2t_k-2}},\] and note the exponents on $i$ are equal and hence we see that $ \E_{k,i} \rightarrow 0$ as $ i \rightarrow \infty$. 
 \hfill $\Box$

 \section{H\'enon  equation with a zero order term on $\IR^N$}  \label{henon_full_section}
 
 In this section we examine solutions of 
 \begin{equation} \label{henon_full}
 -\Delta u + u = |x|^\alpha u^{p-1} \qquad \mbox{ in } \IR^N= \IR^{n} \times \IR^n.
 \end{equation}  A particular interest will be in obtaining positive classical nonradial solutions. Before stating our main result we recall the definition of the best constant in Hardy inequality for  $\R^N$,  that is,
 \begin{equation} \label{hardyy00}
\beta_1=\inf_{u\in H_0^1(\R^N)}
\frac{\int_{\R^N} |\nabla u|^2 \, dx+ \int_{\R^N} u^2 \, dx}{\int_{\R^N} \frac{u^2}{|x|^2} \, dx}. \end{equation}
 
 Here is our main result in this section. 
 \begin{thm} \label{henon_full_thm}  Let  $N \ge 3$ be an even number and  $ \alpha>0.$  The following assertions hold:
 
 \begin{enumerate}   \item Suppose 
 \[ \frac{2N+2 \alpha-4}{N-2}<p< \frac{2N+2\alpha}{N-2}.\]
 Then there is a positive classical $K_+$ ground state solution $u $ of (\ref{henon_full}) (see below for a definition of $K_+$).     
 
 \item Suppose 
 \[\max \left\{\frac{4(N+2)}{\beta_1}+2, \frac{2N+2 \alpha-4}{N-2}\right\}  < p< \frac{2N+2\alpha}{N-2}. \]  Then the positive classical $K_+$ ground state solution $u$ of (\ref{henon_full}) is nonradial.  

 \end{enumerate} 

  \end{thm}

 Consider the full space $ \IR^N= \IR^m \times \IR^n$ (here $m$ and $n$ need not be equal but later we will set them equal)  
 \[ H^1_G(\IR^N):=\left\{ u \in H^1(\IR^N): gu =u \quad \forall g \in G\right\},\] where $gu(x)$ and $ G:=O(m) \times O(n)$ are as defined before.    We now take $ \widehat{\Omega}$ as before  and hence in this case we have $ \widehat{\Omega}$ is the first quadrant in the $(s,t)$ plane.  We define 
 \[ \widetilde{\Omega}:=\left\{(\theta,r): 0<r<\infty, 0<\theta<\frac{\pi}{2} \right\},\] and we take 
 \[ \widetilde{\Omega}_0:=\left\{(\theta,r): 0<r<\infty, 0<\theta<\frac{\pi}{4} \right\}.\]  We set $ K_+$ where the definition has the added modifications to $\IR^N$ that one would expect; so the functions are even across $ \theta =\frac{\pi}{4}$ and increasing in $\theta$ on $(0,\frac{\pi}{4})$.
 
 \begin{prop} \label{full_st} (Imbedding) For all integers $ k \ge 0$ there is some $C_k$ such that for all $ 0 \le \phi \in H^1_G(\IR^N)$ with $ \| \phi \|_{H^1} \le 1$ one has 
 \begin{equation} \label{first_star}
 \int_{\widehat{\Omega}} (\phi(s,t))^{2^*+2k} s^{(k+1) (m-1)} t^{(k+1)(n-1)} ds dt \le C_k,
 \end{equation}
  \begin{equation} \label{second_star}
 \int_{\widehat{\Omega}} (\phi(s,t))^{2(k+1)} s^{(k+1) (m-1)} t^{(k+1)(n-1)} ds dt \le C_k.
 \end{equation}
 \end{prop}

 \begin{proof} Both results  will start with the same basic proof and they will follow by almost the same computation as  the proof of  Proposition  \ref{henon_imbed}.   By a density argument we can assume $ \phi \ge 0$ is smooth and zero for large enough $ r= (s^2+t^2)^\frac{1}{2}$ and we write  $ d \mu(s,t)= s^{m-1} t^{n-1} d s dt$.  Suppose 
\[ \int_{\widehat{\Omega}} \left( | \nabla_{s,t} \phi |^2 + \phi^2 \right) d \mu(s,t)   \le 1.\]  Let $ 0 \le u$ denote a function defined in $(s,t)$ and zero for large $(s,t)$   (we will take $u$ to be a power of $ \phi$). 
 As before we have 
 \[ u(s,t)^2\le \int_s^\infty | \nabla_{s,t} u( \tau_1, t)| d \tau_1 \int_t^\infty | \nabla_{s,t} u( s, \tau_2)| d \tau_2,\] 
   and we now multiply by sides by $ s^{2 \beta_1} t^{2 \beta_2}$ where $ \beta_i>0$ and integrate over $ \widehat{\Omega}$ we arrive at 
  \begin{equation} 
  \int_{\widehat{\Omega}} u(s,t)^2 s^{2\beta_1} t^{2\beta_2} ds dt \le \left(  \int_{\widehat{\Omega}} |\nabla_{s,t} u(s,t)| s^{\beta_1} t^{\beta_2} ds dt \right)^2.
  \end{equation}  We now suppose $ 0 \le \phi $ as above and take $u=\phi^\gamma$ and put into abouve ($\gamma \ge 1$).  
  Then we arrive at 
  \begin{eqnarray*}
  \int_{\widehat{\Omega}} \phi^{2\gamma} s^{2 \beta_1} t^{2 \beta_2} ds dt &\le &\gamma^2 \left( \int_{\widehat{\Omega}} \left\{| \nabla_{s,t} \phi| s^\frac{m-1}{2} t^\frac{n-1}{2} \right\} \left\{\phi^{\gamma-1} s^{\beta_1 - \frac{m-1}{2}} t^{\beta_2- \frac{n-1}{2}} \right\} ds dt \right)^2 \\
  &\le& \gamma^2 \| \nabla \phi \|_{L^2(\Omega)}^2 \int_{\widehat{\Omega}} \phi^{2(\gamma-1)} s^{2 \beta_1 - (m-1)} t^{2 \beta_2 - (n-1)} ds dt,
  \end{eqnarray*} where we performed the Cauchy–Schwarz inequality
 and recall $ \| \nabla \phi \|_{L^2} \le 1$.   \\
 
 We now perform the iterations.   For (\ref{first_star}) we will follow the exact same choice of parameters as in Proposition \ref{henon_imbed} and this gives the desired result.   Note in the first step here we choose the parameters so that the right hand side is exactly 
 \[ \int_{\widehat{\Omega}} \phi^{2^*} s^{m-1} t^{n-1} ds ds,\] which we know is controlled by the critical Sobolev imbedding.     \\

 \noindent  To prove (\ref{second_star}) the only difference is we choose the parameters so that in the first step of the iteration the right hand side is 
 \[ \int_{\widehat{\Omega}} \phi^2 s^{m-1} t^{n-1} ds dt.\] If one performs the iteration they get the desired result. 
 
\end{proof}

 \begin{coro} \label{comp_hen_ful} For $N \ge 3$ even, $ \alpha>0$ and 
 \begin{equation} \label{comp_imb_full}
  \frac{2N+2 \alpha-4}{N-2}<p< \frac{2N+2\alpha}{N-2},
  \end{equation} 
  we have $K_+ \subset \subset L^p( \IR^N, |x|^\alpha dx)$.  
 \end{coro}

 \begin{proof}   Let $ \phi \in K_+$.  Using a suitable compactly supported radial cut off function and Corollary \ref{imbed_+} we see that there is some $C$ (independent of $ \phi$) such that 
 \[ \left( \int_{B_1}  \phi(x)^p |x|^\alpha dx \right)^\frac{1}{p} \le C \| \phi \|_{H^1(\IR^N)},\] and hence we really only need to bound the integral on the region $|x| \ge 1$.  
 
 By using  Proposition  \ref{full_st} and similar arguments that we used to prove  Corollary  
 \ref{imbed_+} we can show  for all integers $k,i \ge 0$ there is some constant depending just on $k,i$  such that  for all $ \phi \in K_+$ with $ \| \phi \|_{H^1} \le 1$ one has 
 \begin{equation} \label{on_t} 
 \int_{\IR^N} \phi^{2+2i} |x|^{2i(n-1)} dx \le C_i,
 \end{equation} 
  \begin{equation} \label{on_tt} 
   \int_{\IR^N} \phi^{2^*+2k} |x|^{2k(n-1)} dx \le C_k.
   \end{equation}   Recall we really only need the estimate on the region $ \frac{\pi}{8}<\theta<\frac{\pi}{4}$  (where $ s$ and $t$ are comparable) and then we can extend to the full region via monotonicity and symmetry.     We now interpolate between these to get the desired result.    Again we fix $ \phi \in K_+$ with $ \| \phi \|_{H^1} \le 1$.   Then we have, for $ \tau>1$,

  \begin{eqnarray*}
\int_{|x|>1} \phi^p |x|^\alpha dx &=& \int_{|x|>1} \left\{ \phi^\frac{(2+2i)}{\tau} |x|^\frac{2i(n-1)}{\tau}   \right\}
\left( \phi^{p - \frac{(2 +2i)}{\tau}} |x|^{ \alpha - \frac{2i(n-1)}{\tau}} \right)  dx \\
& \le &\left(\int_{|x|>1} \phi^{2+2i} |x|^{2i(n-1)} dx \right)^\frac{1}{\tau} \left( \int_{|x|>1} \phi^{\tau'   \left( p - \frac{2+2i}{\tau} \right) } |x|^{ \tau' \left( \alpha - \frac{2i(n-1)}{\tau} \right)}  dx \right)^\frac{1}{\tau'}.
\end{eqnarray*} 

We now choose an appropriate $\tau$ and we will be more general than we need to.   Assume $i,k \ge 0$ are integers and we suppose $ 2+2i <p<2^* +2k$.     Take $ \tau>1$ such that 
\[ \tau' \left( p - \frac{2+2i}{\tau} \right) =2^*+2k,\] and then note we have an estimate provided \[ \tau' \left( \alpha- \frac{2i(n-1)}{\tau} \right) \le 2k(n-1),\] after considering 
(\ref{on_tt}).    Note on can explicitly compute $\tau$  from the first equation to get  
\[ \tau = \frac{2^* + 2k-2-2i}{2^*+2k-p}. \]   Now one needs to check if the second inequality holds.  
For our purposes it will be sufficient to take $ i=0$ and $k$ large.   So define $\tau_k$ by 
\[ \tau_k =  \frac{2^*+2k-2}{2^*+2k-p},  \] and so note that $ \tau_k \searrow 1$ as $ k \rightarrow \infty$.  So we need $ 2<p<2^*+2k$ and $ \tau_k' \alpha \le 2k(n-1)$ which we can rewrite as $ \frac{\alpha \tau_k}{2(n-1)} \le k (\tau_k-1)$  but note that 
\[ k ( \tau_k-1) = \frac{k(p-2)}{2^*+2k-p} \rightarrow \frac{p-2}{2},\] as $ k \rightarrow \infty$. So we see the desired result holds for  large integers $k$ provided  
\[ \frac{\alpha}{2(n-1)} < \frac{p-2}{2},\] which is exactly the lower bound on $p$ from (\ref{comp_imb_full}).   The above shows that for the desired range of parameters we have a continuous imbedding.  We now need to improve this to a compact imbedding.    Note the only potential loss of compactness is if we lose mass at $ \infty$.  Take $ \{ \phi_k \}_k \subset K_+$ with $ \| \phi_k \|_{H^1} \le 1$ and fix $ p,\alpha$ as in (\ref{comp_imb_full}) and then note by taking $ \E>0$ small enough we have $ p,\alpha_\E:=\alpha+\E$ still satisfies (\ref{comp_imb_full}).    Then from the above results we have for some $ \E>0$ that 
\[ \int_{|x|>1} \phi_k^p |x|^{\alpha+\E} dx \le C_\E,\] and hence for large $R$ we have 
\[ \int_{|x|>R} \phi_k^p |x|^\alpha dx \le \frac{C_\E}{R^\E},\] for all $k$ and this is sufficient to rule out a loss of compactness at $ \infty$.   
 \end{proof}

 We now turn to the pointwise invariance property.  We need to show that given $u \in K_+$ there is some $ v \in K_+$ which satisfies 
    \begin{equation} \label{linear_full}
-\Delta v +v =|x|^\alpha u^{p-1}  \mbox{ in } \IR^N, 
\end{equation}   
    
 \begin{prop} \label{pointwise_inv_full} (Pointwise invariance property) Suppose $2m=2n=N$ and $u \in K_+$.  Then there is some $v \in K_+$ which satisfies (\ref{linear_full}). 
 \end{prop} 
 
 \begin{proof} Let $u \in K_+$ and for integers $i \ge 1$ consider the problem
 \begin{equation} \label{eq_hen_pointwise_full_approx}
\left\{\begin{array}{ll}
-\Delta v +v= |x|^\alpha u^{p-1} &  \mbox{ in } B_i, \\
v= 0 &   \mbox{ on } \partial B_i.
\end {array}\right.
\end{equation}  Using the same proof as in Proposition  \ref{henon_prop_point}
we can show there is some $ v_i \in K_+(B_i)$, where $K_+(B_i)$ is the obvious extension of $K_+$ from the unit ball to the ball of radius $i$, which satisfies (\ref{eq_hen_pointwise_full_approx}).    Extend $v_i$ to be zero outside $B_i$.  Then note by multiplying the equation for $v_i$ and integrating over $B_i$ we obtain  
\begin{eqnarray*}
\int_{B_i} | \nabla v_i|^2+ v_i^2 dx &=& \int_{B_i} |x|^\alpha u^{p-1} v_i dx \\ 
&=& \int_{B_i} \left\{ u^{p-1} |x|^\frac{\alpha}{p'} \right\} \left\{ v_i |x|^\frac{\alpha}{p} \right\} dx \\ 
& \le & \left( \int_{B_i} u^{p'(p-1)} |x|^\alpha dx \right)^\frac{1}{p'} \left( \int_{B_i} v_i^p |x|^\alpha dx \right)^\frac{1}{p}
\end{eqnarray*} and then note $ p'(p-1)=p$. Using the imbedding from Corollary \ref{comp_hen_ful}  there is some $C$ such that we have 
\[ \int_{\IR^N} | \nabla v_i|^2 + v_i^2 dx \le C \| u \|_{H^1(\IR^N)}^{p-1} \| v_i \|_{H^1(\IR^N)},\] and this shows that $\{v_i \}_i$ is bounded in $H^1(\IR^N)$.   By passing to a sequence we can assume there is some $ v \in H^1(\IR^N)$ with $ v_i \rightharpoonup v$ in $H^1(\IR^N)$ and $ v$ is an $H^1(\IR^N)$ energy solution of (\ref{linear_full}).   Furthermore we can use arguments similar to before to show that $ v \in K_+$, we omit the details.

 \end{proof}

 \noindent
 \textbf{Proof of Theorem \ref{henon_full_thm}.}
   Here again, we are going to use Theorem  \ref{var-pri} for the proof. 
   Note that conditions $(i)$ and $(ii)$ in Theorem \ref{var-pri} follows from Corollary  \ref{comp_hen_ful} and Proposition \ref{pointwise_inv_full} respectively. This proves the existence of a weak solution $u$ of (\ref{henon_full}).   It also follows from Theorem \ref{nonradial}
   that for 
   \[ \frac{4(N+2)}{\beta_1} < p-2 ,\]
   the ground state solution $ u $ obtained above is non-radial. Here $\beta_1$  is the best constant for the Hardy inequality on $\R^N$ defined in (\ref{hardyy00}).
   Thus,  our solution $u$ is non-radial provided 
  \[\max \left\{\frac{4(N+2)}{\beta_1}+2, \frac{2N+2 \alpha-4}{N-2}\right\}  < p< \frac{2N+2\alpha}{N-2}. \]

\noindent
{\it Regularity of the solution.}  Here we can use a proof similar to the proof of  Theorem \ref{non-radial-henon} but one needs to insert a suitable cut off function.   We omit the details. 
\hfill $\Box$

 \section{A singular potential problem} \label{singular_sect}
 
 Here we examine the problem  
 \begin{equation} \label{eq_sing_pot}
\left\{\begin{array}{ll}
-\Delta u + \frac{u}{|x|^\alpha}  = u^{p-1} &  \mbox{ in } B_1, \\
u= 0 &   \mbox{ on } \partial B_1,
\end {array}\right.
\end{equation} where $ N \ge 3$ and $ \alpha>2$.  In particular we are interested in nonradial positive classical solutions.  Note that we are taking $ \alpha>2$ which can be thought of as super critical values of $ \alpha$.  
Let $H$ denote the completion of the $ \{ u \in C^\infty_c(B_1 \backslash \{0\}): u = u(s,t)\}$ under the norm 
 \[ \|u\|_H^2:=\int_\Omega | \nabla u|^2 + \frac{u^2}{|x|^\alpha} dx.\]   Note if $ \alpha \ge N$  then $H$ does not contain $C_c^\infty(B_1)$ and hence we need to be a bit careful when we define what we mean by a  solution.   
 
\begin{definition}\label{weaksingu} We  call $u $ a weak $H$ solution of 
  \begin{equation} \label{sing_lin_weak}
\left\{\begin{array}{ll}
-\Delta u + \frac{u}{|x|^\alpha}  = f(x) &  \mbox{ in } B_1, \\
u= 0 &   \mbox{ on } \partial B_1,
\end {array}\right.
\end{equation} provided $ u \in H$ and 
\begin{equation} \label{weak_eq_def}
\int_{B_1} \left( \nabla u \cdot \nabla \phi + \frac{ u \phi}{|x|^\alpha} \right) dx = \int_{B_1} f(x) \phi dx \quad \forall \phi \in H.
\end{equation}  
 
\end{definition}

 We will assume that we are in the case of $m=n$ since we will want to work on a suitable version of $K_+$ which we now define.    We define $K_+$ to be exactly analogous to the way it was defined for the H\'enon problem on the ball except now we add the extra condition that $ u \in H$.

 \begin{thm} \label{sing_theorem}    Suppose $m=n$ and consider the problem (\ref{eq_sing_pot}). The following assertions hold;
 \begin{enumerate}
     \item Suppose $ 2<p< \frac{2N+2\alpha-4}{N-2}$, then there is a positive classical  $K_+$ ground state solution of (\ref{eq_sing_pot}).  In addition for all $t>0$ there is some $C_t$ such that $u(x) \le C_t |x|^t$ in $B_1$. 
     
     \item The ground state solution from part 1 is nonradial provided 
      \[p-2 > 4(N+2)/\beta_\alpha(\Omega),\] where 
     \[ \beta_\alpha:= \inf_{0\not= \phi \in H}   \frac{  \int_{B_1} | \nabla \phi|^2 + \frac{ \phi^2}{|x|^\alpha} dx }{  \int_{B_1} \frac{\phi^2}{|x|^2} dx}.\]   We will show $ \beta_\alpha \rightarrow \infty$ as $ \alpha \rightarrow \infty$ and hence this result is nonempty. 
 \end{enumerate}

 \end{thm}

 \begin{remark}  We are able to prove similar results for nonradial domains  provided they are  domains of double  revolution symmetry with the $ \pi/2$ or $ \pi/4$ symmetry and the needed monotonicity. In these cases one works on a suitable version of $K_-$ but we chose not to include these results since the imbeddings we are able to prove appear to be nonoptimal. 
 \end{remark}

\begin{lemma}\label{betalpha}  We have $ \lim_{\alpha \rightarrow \infty} \beta_\alpha=\infty$. 

\end{lemma}

\begin{proof} Recall the boundary Hardy inequality gives 
  \[ \int_{B_1} | \nabla \phi|^2 dx \ge \frac{1}{4} \int_{B_1} \frac{ \phi^2}{ (1-|x|)^2} dx \quad \forall \phi \in H_0^1(B_1).\]  Define 
  \[ H_\alpha(r):=    r^2 \left(  \frac{1}{4(1-r)^2} + \frac{1}{r^\alpha} \right) \quad 0<r<1,\] and we set $ C_\alpha=\min_{0<r<1} H_\alpha(r)$ and note $ C_\alpha \rightarrow \infty$ as $ \alpha \rightarrow \infty$.   Then note we have 
  \begin{eqnarray*}
  \int_{B_1} | \nabla \phi|^2 dx + \int_{B_1} \frac{\phi^2}{|x|^\alpha} dx - C_\alpha \int_{B_1} \frac{ \phi^2}{|x|^2} dx & \ge & \int_{B_1} \left(   \frac{1}{4(1-|x|)^2} + \frac{1}{|x|^\alpha} - \frac{C_\alpha}{|x|^2} \right) \phi^2 dx \\
  &=& \int_{B_1} \frac{\phi^2}{|x|^2} \left( \frac{|x|^2}{4(1-|x|)^2} + \frac{|x|^2}{|x|^\alpha} - C_\alpha \right) dx \\
  &=& \int_{B_1} \frac{\phi^2}{|x|^2} \left( H_\alpha (|x|)-C_\alpha \right) dx \\
  & \ge &0,
  \end{eqnarray*} and from this we see that $\beta_\alpha \ge C_\alpha$ which proves the desired result.

\end{proof}

 \begin{lemma}\label{compsingu} Suppose  $m=n$ and  $ 1 \le p < \frac{2N+2\alpha-4}{N-2}$. Then  $K_+ \subset \subset L^p(B_1)$. 
 
 \end{lemma}

 \begin{proof} Suppose in addition to the hypothesis on $p$ take $p>2$ and then for  $u \in K_+$ with $ \| u\|_H =1$ and $1<\tau<\infty$ we have  
 \begin{eqnarray*} 
 \int_{B_1} u^p dx &=& \int_{B_1} \frac{ u^\frac{2}{\tau} }{ |x|^\frac{\alpha}{\tau}  }  \left\{ u^{ p - \frac{2}{\tau}} |x|^\frac{\alpha}{\tau}  \right\} dx \\
  & \le & \|u\|_H^\frac{2}{\tau}  \left( \int_{B_1}  u^{\tau'( p - \frac{2}{\tau})} |x|^\frac{\tau' \alpha}{\tau}  dx \right)^\frac{1}{\tau'}
 \end{eqnarray*} and now note we need to have some sort of H\'enon type imbedding for $K_+$.  Note this same proof so far would work on a general domain with a function with any type of symmetry.   By Corollary \ref{imbed_+} we see this integral on the right is bounded by a constant provided we have 

  \begin{equation} \label{in_par}
  \tau' \left( p - \frac{2}{\tau} \right) <  \frac{ 2N+ 2 \frac{\tau' \alpha}{\tau}  }{N-2}.
  \end{equation}  Since $ p<\frac{2N+2\alpha -4}{N-2}$ we can show (\ref{in_par}) for $ \tau$ sufficiently close to $ 1$ and this completes the proof of the continuity of the imbedding.  For compactness we use compactness in $L^1(B_1)$ along with standard $L^p$ interpolation. 
 \end{proof}
 
  In the proof of the above  Lemma it is  apparent that once one has a type of H\'enon imbedding then they get an suitable imbedding for $H$.

 \begin{prop} \label{point_sing} (Pointwise Invariance) Take $m=n$ and suppose $ u \in K_+$.  Then there is some $ v \in K_+$ which solves 
  \begin{equation} \label{eq_sing_invari}
\left\{\begin{array}{ll}
-\Delta v + \frac{v}{|x|^\alpha}  = u^{p-1} &  \mbox{ in } B_1, \\
v= 0 &   \mbox{ on } \partial B_1.
\end {array}\right.
\end{equation}
 \end{prop}
 
 \begin{proof}

 Our approach will be to approximate the domain via an annulus and take a limit.   For $ \E>0$ small set  $A_\E:=\{ x \in \IR^N: \E<|x|<1 \}$ and let   $ u \in K_+$.
 Consider the problem
 \begin{equation} \label{eq_cut_sing}
\left\{\begin{array}{ll}
-\Delta v^\E + \frac{v^\E}{|x|^\alpha}  = u^{p-1} &  \mbox{ in } A_\E, \\
v^\E= 0 &   \mbox{ on } \partial A_\E.
\end {array}\right.
\end{equation}
 Note this problem essentially fits into the exact framework of Proposition \ref{pointwise_inv} part 2 except for this $ |x|^\alpha $ term;  but this term has no effect on the approach.   So if we let $ K_+(A_\E)$ denote the obvious extension of $K_+$ to $A_\E$ we see that $ v^\E \in K_+(A_\E)$.  We now extend $v^\E$ to $B_1$ by extending it to be zero outside $ A_\E$ and note $ v_\E \in K_+$.   Then note we 
 \begin{eqnarray*}
 \int_{B_1} | \nabla v^\E|^2  + \int_{B_1} \frac{ (v^\E)^2}{|x|^\alpha} dx &=& \int_{B_1} u^{p-1} v^\E dx \\
 & \le & \left( \int_{B_1} u^{p'(p-1)} dx \right)^\frac{1}{p'} \left( \int_{B_1} (v^\E)^p dx \right)^\frac{1}{p} \\
 &\le& \|  u\|_{L^p(B_1)}^\frac{p}{p'} C \| v^\E \|_{H}
 \end{eqnarray*} where in the last step we used the imbedding of $K_+$.  From this we see there is some constant $C_u$ such that  $ \| v^\E \|_H \le C_u$ for all $ \E>0$ small.    From this we see we can pass to a subsequence and find some $v \in H$ such that $ v^\E \rightharpoonup v$ in $H$ and also since $K_+$ is convex and closed in $H$ we have it weakly closed in $H$ and hence $ v \in K_+$.    Note if $ \phi \in C_c^\infty(B_1 \backslash \{0\})$ we can easily pass to the limit in 
 \begin{equation} \label{weak_H}
 \int_{B_1} \nabla v^\E \cdot \nabla \phi + \frac{ v^\E \phi}{|x|^\alpha } dx = \int_{B_1} u^{p-1} \phi dx,
 \end{equation} and hence we have a solution at least on the punctured ball.
  \end{proof}

 \noindent 
 \textbf{Proof of Theorem \ref{sing_theorem}.}  
  We shall begin by observing that Theorem   \ref{var-pri}  can be easily adapted  to deal with singular problems like (\ref{eq_sing_pot}). The only major change is to replace the notion of the weak solutions in 
   condition  $(ii)$ of  Theorem \ref{var-pri} by  the one in Definition \ref{weaksingu} where the test functions $\phi$ belong to the space $H\cap L^p(B_1).$  Both conditions (i) and $(ii)$ follow from Lemma \ref{compsingu} and Proposition  \ref{point_sing} respectively. This proves the existence of a weak solution for  (\ref{eq_sing_pot}). 
   Also,  a similar argument as in the proof of Theorem \ref{nonradial} shows that the solution is nonradial provided
    \[p-2 > 4(N+2)/\beta_\alpha(\Omega).\] 
    
   Moreover, by Lemma \ref{betalpha} we have that  $ \beta_\alpha \rightarrow \infty$ as $ \alpha \rightarrow \infty$ and hence 
   the ground state solution $ u $  is non-radial for large values of $\alpha$. \\
   
 \noindent 
 \emph{Regularity of ground state solution.}  Let $ u \in K_+$ denote  a ground state solution of (\ref{eq_sing_pot}) and note since 
  $ 2<p< \frac{2N+2\alpha-4}{N-2}$ there is some $ t_0>1$ such that  $ u \in L^{p+2t_0-2}(B_1)$ (after considering the imbedding result).  For $k \ge 0$ define 
  \[ t_{k+1}= \frac{p t_k}{2}- \frac{(p-2)}{2},\] and note that $ t_k \nearrow \infty$ as  $k \rightarrow \infty$.  We will now show one has the following iteration result:  for $ k \ge 0$   
  \begin{equation} \label{iter_sing}
  \mbox{ if }  u \in L^{p+2t_k-2}(B_1) \quad \mbox{ then } \quad u \in L^{p+2t_{k+1}-2}(B_1).
  \end{equation} We now prove this iteration step;  let $ k \ge 0$ and suppose $ u \in L^{p+2t_k-2}(B_1)$.  For $ m \ge 1$ set 
  \begin{equation} 
\phi_m(x)=\left\{\begin{array}{ll}
u(x)^{2t_k-1}  &  \mbox{ if } u(x)<m, \\
u(x) m^{2t_k-2}     &   \mbox{ if }u(x) \ge m, \\
\end {array}\right.
\end{equation} and since $ 2t_k -1 >1$ we see that $ \phi_m \in H_0^1(B_1)$ and its also clear that we have $ \int_{B_1} |x|^{-\alpha} \phi_m^2 dx <\infty$ and hence $ \phi_m \in H$, so we can use $ \phi_m$ as a test function in the definition of $u$ be a weak $H$ solution of (\ref{eq_sing_pot}) to arrive at (after dropping a couple of positive terms from the left)
\begin{eqnarray*} 
\frac{(2t_k-1)}{t_k^2} \int_{\Omega_m} | \nabla u^{t_k}|^2 dx + \int_{\Omega_m} \frac{u^{2t_k}}{|x|^\alpha} dx & \le & \int_{\Omega_m} u^{p+2t_k-2} dx  \\
&&+ m^{2t_k-2} \int_{\Omega_m^c} u^p dx
\end{eqnarray*} where $\Omega_m=\{x \in B_1: u(x)<m\}$ and $ \Omega_m^c$ is its compliment in $B_1$.  Set $ \E_m:= m^{2t_k-2} \int_{\Omega_m^c} u^p dx$ and we will later show that $ \E_m \rightarrow 0$ as $ m \rightarrow \infty$.  Then note passing to the limit in the above inequality we arrive at 
\begin{equation} \label{ineq_sing} \frac{(2t_k-1)}{t_k^2} \int_{B_1} | \nabla u^{t_k}|^2 dx + \int_{B_1} \frac{ u^{2t_k}}{|x|^\alpha} dx \le \int_{B_1} u^{p+2t_k-2} dx, 
\end{equation} and note the integral on the right is finite since we have $ u \in L^{p+2t_k-2}(B_1)$ by hypothesis.  From this we see that $ u^{t_k} \in H$ and note that $u^{t_k}$ and now its easy to see that $u^{t_k} \in K_+$ and hence by the imbedding result we have $ u^{t_k} \in L^p(B_1)$ but note $ t_k p = 2t_{k+1} + p-2$ and hence we have proven the iteration step.  We now show  $\E_m \rightarrow 0$.   By hypothesis we have $ u \in L^{p+2t_k-2}(B_1)$ and hence $ \delta_m:=\int_{\Omega_m^c} u^{p+2t_k-2} dx \rightarrow 0$ and note that $ \E_m \le \delta_m$ which gives the desired result.    

With this iteration we have $ u \in L^T(B_1)$ for all $ 1<T<\infty$.  At this point we could attempt to appeal to some linear theory to show $u$ is bounded but we prefer to follow the iteration through.   Once we have $u$ bounded then we will switch to linear theory. \\ 

Starting at (\ref{ineq_sing}) and dropping a portion of the zero order part of the norm we arrive 
\[ \| u^{t_k} \|_H^2 \le \frac{t_k^2}{2t_k-1} \int_{B_1} u^{p+2t_k-2} dx,\] and using the imbedding of $K_+ $ into $L^p(B_1)$ we arrive at 
\[ \| u\|_{L^{p+2t_{k+1}-2}} \le  \left(  \frac{C_0 t_k^2}{2t_k-1} \right)^\frac{1}{2t_k} \| u\|_{L^{p+2t_k-2}}^\frac{p+2t_k -2}{2t_k},\] where $C_0$ is coming from the imbedding.   We write 
\[ \beta_k:= \| u \|_{L^{p+2t_k-2}}, \quad \gamma_k:=  \left(  \frac{C_0 t_k^2}{2t_k-1} \right)^\frac{1}{2t_k}, \quad \delta_k:=\frac{p+2t_k -2}{2t_k}\] and hence we have 
\[ \beta_{k+1} \le \gamma_k \beta_k^{\delta_k},\] for all $k \ge 0$.  Writing out the iteration we arrive at 
\[ \beta_{n+1} \le \left(\beta_0^{ \prod_{j=0}^n \delta_j} \right) \prod_{k=0}^n \left( \gamma_k^{  \prod_{i=k}^n \delta_i} \right).  \]  We now wish to show the right hand side in bounded in $n$ and hence this would give us the desired $L^\infty$ bound on $u$.  We first show that 
$ T_n:= \prod_{j=0}^n \delta_j$ is bounded.   Consider the log of $T_n$ and note we have 
\begin{eqnarray*}
\ln(T_n) &=& \sum_{j=0}^n \ln (\delta_j) \\ 
&=& \sum_{j=0}^n \ln \left( 1 + \frac{p-2}{2 t_j} \right) \\
& \le & \sum_{j=0}^n   \frac{p-2}{2t_j}
\end{eqnarray*} where we used the fact that $p>2$ and log is concave.  Now note one can get the explicit formula $ t_k = C \left( \frac{p}{2} \right)^k+1$ where $ C>0$ since $ t_0>1$. From this we see that $ \ln(T_n)$ is bounded and hence we have the same for $T_n$.   We now define $ T_{k,n}:=\prod_{i=k}^n \delta_i$ and similarly we get 
\begin{eqnarray*}
\ln(T_{k,n}) &=& \sum_{i=k}^n \ln \left(  1 + \frac{p-2}{2+2C2^{-i} p^i} \right)  \\
& \le & \sum_{i=k}^n   \frac{p-2}{  2 + 2C 2^{-i} p^i} \\
& \le & \frac{\hat{C} 2^k}{p^k}
\end{eqnarray*}  for some $ \hat{C}$ independent of $ k$ and $n$.  This shows $ T_{k,n}$ is bounded above. From this we see that to show $\prod_{k=0}^n \left( \gamma_k^{  \prod_{i=k}^n \delta_i} \right)$ is bounded it is sufficient to show that $ P_n:=\sum_{k=0}^n \ln( \gamma_k)$ is bounded.   But note that 
\[ P_n = \sum_{k=0}^n \frac{1}{2t_k} \ln \left(   \frac{C_0 t_k^2}{2t_k-1} \right),\] and noting the growth of $t_k$ we easily see this is bounded in $n$.   This completes the proof that $ u$ is bounded.

We will now apply Proposition \ref{weight_linear} to get more regularity.   Take $ t<0$ but very close and then note that $ u^{p-1} \in Y_t$ and by uniqueness of the solution to the linear problem we have $ u \in X_t$ and hence we have  $ |u(x)| \le C |x|^{t+\alpha}$.   We can now iterate this process.  For instance we have $| u(x)^{p-1}| \le C |x|^{(p-1)(t+\alpha)}|$ and we choose $ t_1:= (p-1)(t+\alpha)$ and apply the linear theory again to see that 
$ |u(x)| \le C |x|^{(p-1) (t+\alpha) + \alpha}$.  Writing out the iteration we see that for all $ t>0$ there is some $C_t>0$ such that $u(x) \le C_t |x|^t$.

 \hfill $\Box$

 
 We now state a result from the preprint \cite{sing_pot} but we include a partial proof for the readers convenience.  This result will only be used when showing the decay of the solution near the origin. 
\begin{prop} \label{weight_linear} \cite{sing_pot}  (Linear theory for $ -\Delta \phi + \frac{\phi}{|x|^\alpha}$ in weighted $L^\infty$  spaces) For $N \ge 3,  \alpha>2$ and   $ t \in \IR$ define  the norms
 \[ \| f\|_{Y_t}:=\sup_{0<|x| \le 1} |x|^{-t} |f(x)|, \qquad \| \phi \|_{X_t}:= \sup_{0<|x| \le 1} |x|^{-t-\alpha} |\phi(x)|,\] we let $Y_t$ denote the completion of the bounded functions under the $Y_t$ norm and $X_t$ to denote the  continuous functions on $ \overline{B_1} \backslash \{0\}$ which have finite $X_t$ norm and with $ \phi=0$ on $ \partial B_1$. 
  Let $N \ge 3, \alpha >2$ and $ t \in \IR$. Then there is some $C>0$ such that for all $ f \in Y_t$ there is a  $ \phi \in X_t$ such that  
  \begin{equation} \label{init_lin}
 \left\{ \begin{array}{lcl}
\hfill   -\Delta \phi(x) + \frac{\phi(x)}{|x|^\alpha}  &=& f(x)\qquad \mbox{ in } B_1 \backslash \{0\},   \\
\hfill  \phi &=& 0 \hfill  \quad \mbox{ on }     \partial B_1,
\end{array}\right.
  \end{equation} and one has the estimate $ \| \phi\|_{X_t} \le C \| f\|_{Y_t}$.  For 
  \begin{equation} \label{cond_t}
  t> -\alpha - \frac{ \left\{ N-2 + \sqrt{ N^2-4N+8 } \right\}}{2},
  \end{equation} the solution $\phi$ is unique. 
 
\end{prop}
 
 \begin{proof} Fix $N,\alpha $ and $t$ as in the hypothesis.  Let $ f \in Y_t$ with $ \|f\|_{Y_t}=1$.  Since $ \alpha>2$ we can fix $ 0<\E'<\frac{1}{4}$ small such that 
  \[ 1 - (t+\alpha) (t+\alpha-1) |x|^{\alpha-2} - (N-1)(t+\alpha) |x|^{\alpha-2} \ge \frac{1}{2} \qquad \forall\; 0<|x| \le \E',\] and note $ \E'$ only depends on $ N,\alpha$ and $t$.    We can now choose $ C_i=C_i(N,\alpha,t)>0$ such that 
  \[ C_1 \left\{1 - (t+\alpha) (t+\alpha-1) |x|^{\alpha-2} - (N-1)(t+\alpha) |x|^{\alpha-2} \right\} + C_2 \frac{\left\{ 2N + \frac{1-|x|^2}{|x|^\alpha} \right\} }{|x|^t} \ge 1, \; \;  \forall \;  0<|x|<1.\] For $ R_1<R_2$ we set $ A_{R_1,R_2}:=\{x \in \IR^N: R_1 <|x|<R_2\}$.  For $ 0<\E< \frac{\E'}{2}$ consider 
    \begin{equation} \label{annul_begin}
 \left\{ \begin{array}{lcl}
\hfill   -\Delta \phi_\E(x) + \frac{\phi_\E(x)}{|x|^\alpha}  &=& f(x)\qquad \mbox{ in } A_{\E,1},   \\
\hfill  \phi_\E &=& 0 \hfill  \quad \mbox{ on }     \partial A_{\E,1},
\end{array}\right.
  \end{equation} and note there is a classical solution.  Set $ \overline{\phi}(x):=C_1 |x|^{t+\beta} + C_2 (1-|x|^2)$ and by the maximum principle we have $ | \phi_\E(x)| \le \overline{\phi}$ in $A_{\E,1}$ for all small $\E$ (note $ \E'$ is fixed and we be varying $\E$). In particular there is some $ C_3>0$ such that $ \sup_{A_{\frac{\E'}{2},1}} | \phi_\E| \le C_3$ for all  small $\E>0$.  We now set 
$\overline{\psi}(x):= C_4  |x|^{t+\alpha}$ where $C_4=C_3+2$.  Then we can apply the maximum principle on $ A_{ \E, \E'}$ to see that $ | \phi_\E(x)| \le \psi(x)=C_4 |x|^{t+\alpha}$ in $ A_{ \E, \E'}$.  This shows that there is some $C>0$ such that for all small $\E>0$ we have $ \| \phi_\E\|_{X_t} \le C \|f \|_{Y_t}$ (where the norms are now over the annulus).  The main point is the constant $C$ does not depend on $\E$.  Taking $ \E=\E_m \searrow 0$ and applying a diagonal argument (using the equation to obtain the needed compactness away from the origin) there is some $ \phi \in X_t$ which solves (\ref{init_lin}) and we have the desired estimate.  \\

We now prove the uniqueness part.   Let $ \phi \in X_t$ solve (\ref{init_lin}) with $ f=0$.  We write $ \phi(x)= \sum_{k=0}^\infty a_k(r) \psi_k(\theta)$ where $ (\psi_k, \lambda_k)$ are the eigenpairs of the Laplace-Beltrami operator $ -\Delta_\theta = - \Delta_{S^{N-1}}$ on the unit sphere $ S^{N-1}$.   Then for all $ k \ge 0$ we have $a_k$ satisfies 
\begin{equation} \label{unique}
-a_k''(r) - \frac{(N-1) a_k'(r)}{r} +\frac{\lambda_ka_k}{r^2}+ \frac{a_k(r)}{r^\alpha}= 0  \qquad 0<r<1,
\end{equation}
with $ a_k(1)=0$ and $|a_k(r)|\le C_k r^{t+\alpha}$.   We now need to show that $ a_k=0$ for all $ k \ge 0$.  Take $ w(r):= r^\gamma a_k(r)$ where $ t+\alpha+\gamma>0$.  Then note we have $  w(1)=0=\lim_{r \searrow 0} w(r)$ and hence if $ w$ is not identically zero we can (after multiplying by $ -1$) see that $w$ attains its max at some $ 0<r_0<1$ with $ w(r_0)>0$,  $ w''(r_0) \le 0$ and $ w'(r_0)=0$.    Note the equation for $ w$ is given by 
\[ w''(r) + \left( \frac{N-1}{r} - \frac{2 \gamma}{r} \right) w'(r) + C_k(r) w(r), \quad 0<r<1,\] where 
\[ C_k(r) = \frac{\gamma (\gamma+1)}{r^2} - \frac{\gamma(N-1)}{r^2} - \frac{\lambda_k^2}{r^2}- \frac{1}{r^\alpha}.\]   Note if $ C_k(r_0) <0$ then evaluating the equation for $ w$ at $ r_0$ gives a contradiction.   Now note that 
\[ r^2 C_k(r_0) <  \gamma (\gamma+1) - \gamma(N-1) -1,\] and hence we have the desired contradiction provided $\gamma (\gamma+1) - \gamma(N-1) -1 \le 0$.  Let $ \gamma_-<\gamma_+$ denote the roots of this quadratic equation and note we need some $ \gamma$ such that $ t+\alpha + \gamma>0$ and $ \gamma \in (\gamma_-, \gamma_+)$.  So to find such a $ \gamma$ it is sufficient that $ t+\alpha + \gamma_+>0$ and writing this out gives (\ref{cond_t}). 

 \end{proof}

 \section{Nonradial solutions when $\Omega$ is a radial domain.} \label{non-rad-sect} 


In this section we  discuss the case when $a(x)=a(|x|)$ is radial,  and $\Omega$ is a radial domain,  
 that is $\Omega=\{x:\,  R_1\leq  |x|<R_2\} $ where $R_1 \geq 0$ and $R_2 \in (R_1, +\infty],$ 
 \begin{equation} \label{eqzvv}
\left\{\begin{array}{ll}
-\Delta u+\lambda u = a(|x|) u^{p-1} &  \mbox{ in } \Omega, \\
u>0    &   \mbox{ in } \Omega, \\
u= 0 &   \mbox{ on } \pOm.
\end {array}\right.
\end{equation}
where $\lambda=0$ for bounded domains and $\lambda=1$ where $\Omega=\IR^N$.  Note we are writing a general form that can handle all radial domains we consider.   When $ R_1=0$ then we are either on a ball or the full space.  When $ R_1>0$ then we are taking $ R_2$ finite (we are not examining exterior domains here) and then we should take $\Omega:=\{x : R_1<|x|<R_2\}$.  
 We shall prove that the solution obtained in Theorem \ref{theorem-nonlinear-annular}, Theorem  \ref{non-radial-henon}  and Theorem \ref{henon_full_thm} are  nonradial  under certain assumptions on $\Omega$ and $p$.\\

We require some preliminaries before stating  our theorem for the radial domain.  Consider the variational formulation of an eigenvalue problem given by
\begin{equation} \label{var_mu}
    \mu_1=\inf_{\psi \in H^1_{loc}(0, \frac{\pi}{4})}\Big \{ \int_0^{\frac{\pi}{4}} |\psi'(\theta)|^2 \omega(\theta) \,d\theta; \quad    \int_0^{\frac{\pi}{4}} |\psi(\theta)|^2\omega(\theta) \,d\theta=1, \int_0^{\frac{\pi}{4}} \psi(\theta)\omega(\theta) \,d\theta=0     \Big\},
\end{equation} where $\omega(\theta):=\cos^{n-1}(\theta) \sin^{n-1}(\theta)$ and suppose $ \psi_1$ satisfies the minimization problem.  Then $ (\mu_1,\psi_1)$ satisfies 
\begin{equation} \label{eigen_p}
\left\{\begin{array}{ll}
 -\partial_\theta (\omega(\theta) \psi_1'(\theta) )= \mu_1 \omega(\theta)\psi_1(\theta)&  \mbox{ in } (0, \frac{\pi}{4}), \\
\psi'(\theta)>0   &   \mbox{ in } (0, \frac{\pi}{4}), \\
\psi_1'(0)= \psi_1'(\frac{\pi}{4})=0, & 
\end {array}\right.
\end{equation}  
and note $(\mu_1,\psi_1)$ is the second eigenpair, the first eigenpair is given by $ (\mu_0,\psi_0)=(0,1)$. \\
An easy computation shows that 
\[\mu_1=   4(N+2), \qquad  \psi_1(\theta)=-\cos(4\theta)+\frac{2-N}{2+N}.\]

   
   We also recall the definition of the best constant in Hardy inequality for the domain $\Omega$,  that is,
\begin{equation} 
\beta_\lambda(\Omega)=\inf_{u\in H_0^1(\Omega)}
\frac{\int_\Omega |\nabla u|^2 \, dx+\lambda \int_\Omega u^2 \, dx}{\int_\Omega \frac{u^2}{|x|^2} \, dx}. \end{equation}
   
We are now ready to state our general theorem regarding the existence of a non-radial solution for a fully radial problem.
\begin{thm}\label{nonradial}  Let $u$ be the $K_+$ ground state solution obtained in either of Theorems \ref{theorem-nonlinear-annular},  \ref{non-radial-henon}  or \ref{henon_full_thm}.  If \[p-2 > 4(N+2)/\beta_\lambda(\Omega),\]  then $u$ is a nonradial function.
\end{thm}


\noindent
\textbf{Proof.} Let us assume  that   $u$   is a radial function.
Note that $K=K_+$ consists of functions $w=w(r,\theta)$ where $\theta \mapsto w(r,\theta)$ is non-decreasing on the interval $(0, \pi/4)$. Recall  that  $E_K(u)= c>0$ where the critical value $c$ is characterized by
$$c= \inf_{\gamma\in \Gamma}\max_{\tau\in [0,1] }E_K[\gamma(\tau)],$$
where \[\Gamma= \{\gamma\in C([0,1], V) : \gamma(0)=0\neq \gamma(1), E_K(\gamma(1))\leq 0\}.\]
For the sake of simplifying the notations, we use $E$ instead of $E_K$ in the rest of the proof. Let  $ \psi_1$ satisfies (\ref{eigen_p}), and  let $\psi$ be the extension of $\psi_1$  evenly across $ \theta= \frac{\pi}{4}.$  Note that $\psi$  solves the same equation on $(0, \frac{\pi}{2}).$ 
Set $v(r,\theta)=u(r)\psi(\theta)$ and note that $u+tv$ belongs to the set $K$ for $0<t<1$.  We first show that
\begin{equation}\label{qw}
 \int_\Omega |\nabla v|^2 \, dx+\lambda \int_\Omega v^2 \, dx-(p-1)\int_\Omega|a(|x|)u|^{p-2}v^2\,dx <0.
\end{equation}
To this end we need to show that $M(u,v)<0$ where
\begin{equation}\label{qw00}
  M(u,v):=\int_{\widehat\Omega} s^{n-1}t^{n-1}(v_t^2+v_s^2+\lambda v^2) \, ds\, dt-(p-1)\int_{\widehat\Omega}s^{m-1}t^{n-1} a(s,t)u^{p-2}v^2\,ds\,dt <0.
\end{equation}
Note first that it follows from the equation $-\Delta u+\lambda u=a(r)u^{p-1}$ that 
\begin{eqnarray}\label{qw1}
\int_{R_1}^{R_2} (u_r^2+\lambda u^2) r^{N-1}\, dr=\int_{R_1}^{R_2}a(r)u^p r^{N-1}\, dr.
\end{eqnarray}
It also from the definition of $\beta=\beta_\lambda(\Omega)$,  the best constant in Hardy inequality, that   
\begin{equation}\label{hardyy000} 
\beta \int_{R_1}^{R_2} \frac{u^2}{r^2}r^{N-1}\, dr\leq \int_{R_1}^{R_2} (u_r^2+\lambda u^2) r^{N-1} \, dr. \end{equation}
It follows from (\ref{qw1}) by writing $M(u,v)$ in polar coordinates that  
\begin{eqnarray*}
M(u,v)&=&\int_{R_1}^{R_2} \int_0^{\frac{\pi}{2}}\Big(\psi^2u_r^2+\frac{u^2\psi'^2}{r^2}+\lambda u^2\psi^2-(p-1)a(r)u^{p}\psi^2\Big)r^{N-1}\omega(\theta) \,d\theta\, dr\\
 &=& \int_{R_1}^{R_2} \int_0^{\frac{\pi}{2}}\frac{u^2\psi'^2}{r^2}r^{N-1}\omega(\theta)  \,d\theta\, dr-(p-2)\int_{R_1}^{R_2} \int_0^{\frac{\pi}{2}}\psi^2(u_r^2+\lambda u^2) r^{N-1}\omega(\theta) \,d\theta\, dr,
 \end{eqnarray*}
 where $\omega(\theta)=\cos^{n-1}(\theta) \sin^{n-1}(\theta)$.
 This together with the definition of $\mu_1=4(N+2)$ in (\ref{var_mu}) and the inequality  (\ref{hardyy000}) imply that 
\begin{eqnarray*}
M(u,v)
 &=& \mu_1\int_{R_1}^{R_2} \int_0^{\frac{\pi}{2}}\frac{u^2\psi^2}{r^2}r^{N-1}\omega(\theta) \,d\theta\, dr-(p-2)\int_{R_1}^{R_2} \int_0^{\frac{\pi}{2}}\psi^2(u_r^2+\lambda u^2) r^{N-1}\omega(\theta) \,d\theta\, dr\\
 &=& \int_0^{\frac{\pi}{2}} |\psi(\theta)|^2\omega(\theta) \,d\theta \Big (   \mu_1\int_{R_1}^{R_2} \frac{u^2}{r^2}r^{N-1}\, dr-(p-2)\int_{R_1}^{R_2} (u_r^2+\lambda u^2) r^{N-1}\, dr\Big)\\
 &\leq & \int_0^{\frac{\pi}{2}} |\psi(\theta)|^2\omega(\theta) \,d\theta  \Big (  \frac{\mu_1}{\beta}\int_{R_1}^{R_2} u^2_r r^{N-1}\, dr-(p-2)\int_{R_1}^{R_2} (u_r^2+\lambda u^2) r^{N-1} \, dr \Big)\\
 &=&\int_0^{\frac{\pi}{2}} |\psi(\theta)|^2\omega(\theta) \,d\theta \int_{R_1}^{R_2} (u_r^2+\lambda u^2) r^{N-1} \, dr\Big(\frac{\mu_1}{\beta}-(p-2)\Big)<0,
\end{eqnarray*}
where the last inequality follows from the fact that
\[\frac{\mu_1}{\beta}-(p-2)=\frac{4(N+2)}{\beta}-(p-2)<0.\]

 Set $\gamma_\sigma (\tau)= \tau({u}+ \sigma v)l$, where $l>0$ is chosen in such a way that $E\big( ({u}+ \sigma v)l\big)\leq 0$ for all $|\sigma|\leq 1$. Note that $\gamma_{\sigma}\in \Gamma$. We shall show that there exists $\sigma>0$ such that for every $\tau\in [0,1]$ one has  $E(\gamma_\sigma(\tau))< E({u})$, and  therefore, 
\[c\leq \max_{\tau\in[0,1]} E(\gamma_\sigma(\tau))< E({u}),\]
which leads to a contradiction since $E(u)=c.$
Note first that there exists a unique  smooth real  function $g$ on a small neighbourhood of zero with $g'(0)=0$ and $g(0)=1/l$ such that 
$\max_{\tau\in[0,1]} E(\gamma_\sigma(\tau))=E\big (g(\sigma)({u}+ \sigma v)l\big).$
We now define $h: \R \to\R$ by $$h(\sigma)=E\big (g(\sigma)({u}+ \sigma v)l\big)-E(u).$$ Clearly we have $h(0)=0$.  Note also that $h'(0)=0$ due to the facts that $E'(u)=0$ and $\int \psi \omega(\theta) \, d\theta=0. $ Finally 
 $h''(0)<0$ due to (\ref{qw}).  This in fact show that 
 \[\max_{\tau\in[0,1]} E(\gamma_\sigma(\tau))=E\big (g(\sigma)({u}+ \sigma v)l\big)<E(u),\]
for small $\sigma>0$  as desired.
\hfill $\Box$

 \section{Domains of triple revolution} \label{dim+1}

 In this section we consider domains of triple revolution.  In particular we consider 
 \begin{equation} \label{eq_triple}
\left\{\begin{array}{ll}
-\Delta u  = a(x) u^{p-1} &  \mbox{ in } \Omega, \\
u= 0 &   \mbox{ on } \pOm,
\end {array}\right.
\end{equation}
 where $\Omega$ is a bounded domain in $\IR^N$ which has a smooth boundary and which is a domain 
 of triple revolution.  Consider 
  \[ s= \left\{ x_1^2+ \cdot \cdot \cdot + x_m^2\right\}^\frac{1}{2}, \quad  t = \left\{ x_{m+1}^2+ \cdot \cdot \cdot + x_{m+n}^2\right\}^\frac{1}{2}, \quad \tau:=\left\{ x_{m+n+1}^2+ \cdot \cdot \cdot + x_{N}^2 \right\}^\frac{1}{2},\] so $s,t,\tau$ has dimension $m,n, l=N-(m+n)$ respectively.   Here the function $a$ is a function of $(t,s,\tau),$ that is $a=a(t,s,\tau).$   
  \begin{remark}
   Note that a radial domain  and a domain of double revolution are particular cases of domains of  triple revolution. However, domains of triple revolutions are not necessarily radial or domains of double revolution. Besides providing a framework to deal with more general domains,  this will create a pathway  to prove several multiplicity results for positive solutions on radial domains.     For instance an annulus can be seen as a radial domain and a domain of double revolution as well as a domain of triple revolution. Thus,  one can obtain new positive solutions for  a radial problem by looking into solutions having  a nontrivial  triple symmetry.   This is indeed the main motivation for this section.   
  \end{remark}

  In the previous sections we used polar coordinates in the $(s,t)$ plane.    In this section we will use spherical coordinates to describe the coordinates $ (s,t,\tau)$: \\
 \begin{equation}  \label{spher_cord} 
 s= r \sin(\theta) \cos(\phi), \quad t = r \sin(\theta) \sin(\phi), \quad \tau=r \cos(\theta),
 \end{equation}
 where $0<\theta< \pi$, $ 0<\phi < 2 \pi$ and $ r>0$;  but of course we have restricted $(s,t,\tau)$ to the first octant in $ \IR^3$ and hence  $ 0<\theta< \frac{\pi}{2}$, $ 0<\phi<\frac{\pi}{2}$, and $ r>0$.  
 Note that the function $a$ can be also seen as a function of
 $( \phi, \theta, r),$ that is $a=a( \phi, \theta,r).$
 
 The monotonicity we will use will be in $ \phi$ and hence it is also very natural to consider cylindrical coordinates for $ (s,t,\tau)$ but we chose spherical for variety and also since we have the case of an annulus in mind which may be more natural to consider spherical coordinates.  

 We now define 
 \[ U=\left\{(s,t,\tau) \in \IR^3: x=(x_1=s,x_2=0,..., x_m=0, x_{m+1}=t, x_{m+2}=0,..., x_N= \tau) \in \Omega \right\},\] where $x_i =0$ for $ i \notin \{1, n+1,N \}$.     We define $ \widehat{\Omega}=\{(s,t,\tau) \in U: s,t,\tau >0 \}$.  We now define 
 \[ \widetilde{\Omega}= \left\{(\phi,\theta,r) \in \left(0,\frac{\pi}{2} \right) \times   \left(0,\frac{\pi}{2} \right) \times (0,\infty):  (s,t,\tau) \in \widehat{\Omega} \right\},\] and 
 we also  define a subset of $ \widetilde{\Omega}$ given by 
  \[ \widetilde{\Omega}_0= \left\{(\phi,\theta,r) \in \left(0,\frac{\pi}{4} \right) \times  \left(0,\frac{\pi}{2} \right)  \times (0,\infty):  (s,t,\tau) \in \widehat{\Omega} \right\},\] where note the only change is we are now restricting $ 0<\phi<\frac{\pi}{4}$.  \\ 
 Take $ G=O(m) \times O(n) \times O(l)$ and  consider 
  \[ H^1_{0,G}(\Omega)=\left\{ u \in H_0^1(\Omega): gu = u \; \; \forall g \in G \right\}.\]
  We are now ready to state our monotonicity assumptions for the domians of triple revolution.  
  
  \begin{definition}\label{triple-def}[\textbf{The monotonicity assumption on the functions and the domain}]\\  
  Let $\Omega$ be a bounded domain of triple revoluion in $\R^N=\R^m\times\R^n\times\R^l.$
  
\begin{enumerate}

 \item ($K_-$ definition and domain assumptions) Suppose $ g^i=g^i( \phi, \theta)$ is smooth and positive on $ [0, \frac{\pi}{2}] \times [0, \pi/2]$ and for each fixed $ \theta \in (0,\pi/2)$ we have:
    $ \phi \mapsto g^i(\phi,\theta)$ even about $ \phi=\frac{\pi}{4}$, for $i=2$ we have the map is decreasing in $ \phi$ on $(0,\frac{ \pi}{4})$ and $ i=1$ we have it increasing in $ \phi$ on $(0,\frac{ \pi}{4})$. We also $ g^1< g^2$ on  $ [0, \frac{\pi}{2}] \times [0, \pi/2]$.    We consider domains where 
    \[ \widetilde{\Omega}=  \left\{(\phi,\theta,r): g^1(\phi,\theta)<r<g^2(\phi,\theta) \mbox{ for }  (\phi,\theta) \in \left(0, \frac{\pi}{2}\right) \times \left(0, \frac{\pi}{2} \right) \right\}.\] 
    Define $K_-$ to be  the set of nonnegative functions $ u \in H_{0,G}^1(\Omega)$ with $ u_\phi \le 0$ in $ \widetilde{\Omega}_0$ and which are even across $ \theta=\frac{\pi}{4}$.

\item ($K_+$ definition and domain assumptions) Suppose $ g^i=g^i( \phi, \theta)$ is smooth and positive on $ [0, \frac{\pi}{2}] \times [0, \pi/2]$ and for each fixed $ \theta \in (0,\pi/2)$ and $i=1,2$  we have
    $ \phi \mapsto g^i(\phi,\theta)$ is constant on $(0, \pi/2)$.  We consider domains $\Omega$ where 
        \[ \widetilde{\Omega}=  \left\{(\phi,\theta,r): g^1(\phi,\theta)<r<g^2(\phi,\theta) \mbox{ for }  (\phi,\theta) \in \left(0, \frac{\pi}{2}\right) \times \left(0, \frac{\pi}{2} \right) \right\}.\]   Note this includes the case of an annulus. 
      Define $K_+$  to be  the set of nonnegative functions $ u \in H_{0,G}^1(\Omega)$ with $ u_\phi \ge 0$ in $ \widetilde{\Omega}_0$ and which are even across $ \phi=\frac{\pi}{4}$.  
    
   \item  ( $K_{-,\frac{\pi}{2}}$ definition and domain assumptions) Suppose $ g^i=g^i( \phi, \theta)$ is smooth and positive on $ [0, \pi/2] \times [0, \pi/2]$ and for each fixed $ \theta \in (0,\pi/2)$ we have:
     the map  $ \phi \mapsto g^2( \phi,\theta)$ is decreasing in $ \phi$ on $(0,\frac{ \pi}{2})$ and $ \phi \mapsto g^1(\phi,\theta)$  is increasing in $ \phi$ on $(0,\frac{ \pi}{2})$. We alsohave  $ g^1< g^2$ on  $ [0, \frac{\pi}{2}] \times [0, \pi/2]$.    We consider domains $ \Omega$  where 
    \[ \widetilde{\Omega}=  \left\{(\phi,\theta,r): g^1(\phi,\theta)<r<g^2(\phi,\theta) \mbox{ for }  (\phi,\theta) \in \left(0, \frac{\pi}{2}\right) \times \left(0, \frac{\pi}{2} \right) \right\}.\] 
    Define $K_{-,\frac{\pi}{2}}$ to be  the set of nonnegative functions $ u \in H_{0,G}^1(\Omega)$ with $ u_\phi \le 0$ in $ \widetilde{\Omega}$. 
    \end{enumerate}  
\end{definition}

  Here we state our main theorem for this section.

\begin{thm} \label{triple-thm}   Let  $\Omega$ be a bounded domain of triple revolution in $\R^N=\R^m\times\R^n\times\R^l$ and consider  (\ref{eq_triple}) with $ a=a(\phi,\theta,r) $ positive and sufficiently smooth.
    
    \begin{enumerate} \item Suppose $m=n$ and  $\Omega$ is a  domain satisfying the symmetry condition part 1 of Definition \ref{triple-def} and $ a_\phi \le 0$ in $ \widetilde{\Omega}_0$. Then for all 
   \[ 2< p<  \min \left\{ \frac{2(n+m+1)}{n+m-1},  \frac{2(n+l+1)}{n+l-1} \right\},\]there is a positive classical $K_-$ ground state solution $u $ of (\ref{eq_triple}).   Note this case includes the case of $\Omega$ an annulus. 
   
   \item Suppose $m=n$ and  $\Omega$  is a  domain satisfying the symmetry condition part 2 on Definition \ref{triple-def} and $ a_\phi \geq 0$ in $ \widetilde{\Omega}_0$. Then for all 
  \[ 2< p<  \min  \left\{ \frac{2(l+2)}{l},  \frac{2(n+m+1)}{n+m-1} \right\},\] there is a positive classical $K_+$ ground state solution $u $ of (\ref{eq_triple}).   Note this case includes the case of $\Omega$ an annulus. 
  
   \item Suppose $\Omega$  is a  domain satisfying the symmetry condition part 3  on Definition \ref{triple-def} with $n\leq m$   and $ a_\phi \le 0 $ in $ \widetilde{\Omega}$. 
    Then for all 
     \[ 2< p<  \min \left\{ \frac{2(n+m+1)}{n+m-1},  \frac{2(n+l+1)}{n+l-1} \right\},\]
    there is a positive   classical $K_{-,\frac{\pi}{2}}$ ground state solution $ u $ of (\ref{eq_triple}). 
      \end{enumerate} 

\end{thm}

Before discussing the proofs we write out some formula's we will need soon.  Given a function $v(x)$ defined on $\Omega$ (which has the $G$ symmetry) we have 
  \[ \int_\Omega v(x) dx = c \int_{\widehat{\Omega}} v(s,t,\tau) s^{n-1} t^{n-1} \tau^{l-1} ds dt d \tau,   \] where we are abusing notation as usual. If we further abuse notation we can write this in terms of spherical coordinates as 
  \[ \int_{\widetilde{\Omega}} v( \phi, \theta,r) d \mu(\phi,\theta,r)\] where 
  
    \[ d \mu( \phi,\theta,r)= r^{N-1} \sin^{m+n-1} (\theta) \cos^{m-1}(\phi) \sin^{n-1}(\phi) \cos^{l-1}(\theta) d \phi d \theta dr,\]  and   in the case of $m=n$ we have 
     \[ d \mu( \phi,\theta,r)= r^{2n+l-1} \sin^{2n-1} (\theta) \cos^{n-1}(\phi) \sin^{n-1}(\phi) \cos^{l-1}(\theta) d \phi d \theta dr.\]    
  Also note we can write the square of the gradient as 
  \[ | \nabla u(x)|^2 = u_r^2+  \frac{u_\theta^2}{r^2} + \frac{ u_\phi^2}{r^2 \sin^2(\theta)}. \]

As before we begin by examining the added compactness one gets. 

 \begin{thm}\label{compact-m}(Imbeddings for annular  domains) \label{compact_no_mono} Let $\Omega$ denote an annular of triple revolution  in $\R^N=\R^m\times\R^n\times\R^l$. 

\begin{enumerate}  \item (Imbedding without monotonicity) Suppose $\Omega$ has no monotonicity and  \[ 1 \le p < p_1(m,n,l):=\min \left\{ \frac{2(n+m+1)}{n+m-1},  \frac{2(m+l+1)}{m+l-1}, \frac{2(n+l+1)}{n+l-1} \right\}.\] Then 
$ H^1_{0,G}(\Omega) \subset \subset  L^p(\Omega)$.

\item (Imbedding with monotonicity) Suppose $ \Omega$ satisfy the symmetry condition part 1  in Definition \ref{triple-def},  $ n \le m$ and   
\[ 1 \le p<  p_2(m,n,l):=\min \left\{ \frac{2(n+m+1)}{n+m-1},  \frac{2(n+l+1)}{n+l-1} \right\}.\]
Then $K_- \subset \subset L^p(\Omega)$.
\item Suppose $ \Omega$ satisfy the symmetry condition part 3 in Definition \ref{triple-def},  $ n \le m$ and   
\[ 1 \le p<  p_2(m,n,l):=\min \left\{ \frac{2(n+m+1)}{n+m-1},  \frac{2(n+l+1)}{n+l-1} \right\}.\]
Then $K_{-, \frac{\pi}{2}} \subset \subset L^p(\Omega)$.
\item   Suppose $ \Omega$ satisfy the symmetry condition part 2 in Definition \ref{triple-def}  and   
\[ 1 \le p<  p_3(m,n,l):=\min  \left\{ \frac{2(l+2)}{l},  \frac{2(n+m+1)}{n+m-1} \right\}.\]
Then $K_+ \subset \subset L^p(\Omega)$.  
\end{enumerate} 
\end{thm}

\begin{proof}  1.     This part follows from Theorem \ref{embed-mrev}. \\

\noindent  2.  By using 
 spherical coordinates for $ (s,t,\tau)$
 \[ s= r \sin(\theta) \cos(\phi), \quad t = r \sin(\theta) \sin(\phi), \quad \tau=r \cos(\theta),\]
 we have that  

\begin{eqnarray*}
\int_{\widehat{\Omega}}  u(s,t,\tau)^p s^{m-1} t^{n-1} \tau^{l-1} ds dt d\tau \\=\int_{0}^{\pi/2}\int_{0}^{\pi/2}\int_{g_1}^{g_2} r^{N-1}\sin^{m-1}(\theta)\cos^{m-1}(\phi)\sin^{n-1}(\theta)\sin^{n-1}(\phi) \cos^{l-1}(\theta) u(\phi,\theta, r)^p\, dr \, d \theta d\phi.
\end{eqnarray*} 
For $\phi \in [\pi/3, \pi/2]$ we have that
$\sin(\phi) \leq c \sin(\phi-\pi/4)$ for some constant $c>0$.  Thus, considering the  evenness 
 properties of $g_1$, $g_2$  and $\phi \mapsto u(\phi,  \theta, r)$  across $ \phi=\frac{\pi}{4}$  we obtain that 

\begin{eqnarray*}
\int_{\pi/3}^{\pi/2}\int_{g_1(\phi,\theta)}^{g_2(\phi,\theta)} r^{N-1}\cos^{m-1}(\phi)\sin^{n-1}(\phi) u(\phi,  \theta, r)^p\, dr \,  d\phi\\ \leq c^{n-1} \int_{\pi/3}^{\pi/2}\int_{g_1(\phi-\pi/4,\theta) }^{g_2(\phi-\pi/4,\theta) } r^{N-1}\cos^{m-1}(\phi-\pi/4)\sin^{n-1}(\phi-\pi/4)  u(\phi-\pi/4,  \theta, r)^p\, dr \,  d\phi\\
=c^{n-1}\int_{\pi/12}^{\pi/4}\int_{g_1(\phi,\theta)}^{g_2(\phi,\theta)} r^{N-1}\cos^{m-1}(\phi)\sin^{n-1}(\phi)  u(\phi,  \theta, r)^p\, dr \,  d\phi.
\end{eqnarray*}

Thus,  there is a constant $C_1>0$ such that 
\begin{eqnarray*}
\int_{0}^{\pi/2}\int_{g_1}^{g_2} r^{N-1}\cos^{m-1}(\phi)\sin^{n-1}(\phi) u(\phi,  \theta, r)^p\, dr \,  d\phi\\ \leq C_1 \int_{0}^{\pi/3}\int_{g_1}^{g_2} r^{N-1}\cos^{m-1}(\phi)\sin^{n-1}(\phi) u(\phi,  \theta, r)^p\, dr \,  d\phi.
\end{eqnarray*}
On the other hand,
\begin{eqnarray}\label{one}
 \int_{\pi/4}^{\pi/2}\int_{0}^{\pi/3}\int_{g_1}^{g_2} r^{N-1}\cos^{m-1}(\phi) r^{n-1} \sin^{n-1}(\phi) u(\phi,  \theta, r)^p\sin^{n+m-2}(\theta)\cos^{l-1}(\theta)\, dr \, d \phi\, d\theta \nonumber \\ =\int_{\{\widehat \Omega, \,\, s\geq \beta\}}  u(s,t, \tau )^p s^{m-1} t^{n-1}\tau^{l-1} ds dt d\tau
\end{eqnarray}
and
\begin{eqnarray}\label{two}
 \int_{0}^{\pi/4}\int_{0}^{\pi/3}\int_{g_1}^{g_2} r^{N-1}\cos^{m-1}(\phi) r^{n-1} \sin^{n-1}(\phi) u(\phi,  \theta, r)^p \sin^{n+m-2}(\theta)\cos^{l-1}(\theta)\, dr \, d \phi\, d\theta \nonumber \\=\int_{\{\widehat \Omega, \,\, \tau\geq \beta\}}  u(s,t, \tau )^p s^{m-1} t^{n-1}\tau^{l-1} ds dt d\tau
\end{eqnarray}
for some positive constant $\beta$. Therefore, for $(\ref{one})$, we have 
\begin{eqnarray*}
\left (\int_{\{\widehat{\Omega}, \,\, s\geq \beta\}}  u(s,t, \tau )^p s^{m-1} t^{n-1} \tau^{l-1}ds dt\right)^{2/p} \leq  C_2 \left (\int_{\{{\widehat{\Omega}}, \,\, s\geq \beta\}}  u(s,t, \tau)^p  t^{n-1}\tau^{l-1} ds dt\right)^{2/p}.
\end{eqnarray*}
Thus, by part 1),
\begin{eqnarray*}
 \left (\int_{\{{\widehat{\Omega}}, \,\, s\geq \beta\}}  u(s,t, \tau)^p  t^{n-1}\tau^{l-1} ds dt\right)^{2/p} &\leq & C_3\int_{\{{\widehat{\Omega}}, \,\, s\geq \beta\}}  (u^2+u_s^2+u_t^2+u_\tau^2)  t^{n-1}\tau^{l-1} ds dt d \tau\\ & \leq & C_4 \int_{\{{\widehat{\Omega}}, \,\, s\geq \beta\}}  (u^2+u_s^2+u_t^2+u_\tau^2)  t^{n-1} s^{m-1} \tau^{l-1}ds dtd\tau \\ &\leq & C_4 \int_{{\widehat{\Omega}}}  (u^2+u_s^2+u_t^2+u_\tau^2)  t^{n-1} s^{m-1}\tau^{l-1} ds dtd\tau=
C_5\|u\|^2_{H^1(\Omega)}.
\end{eqnarray*}
By a similar argument 
 for $(\ref{two})$ we have 
\[ \int_{0}^{\pi/4}\int_{0}^{\pi/3}\int_{g_1}^{g_2} r^{N-1}\cos^{m-1}(\phi) r^{n-1} \sin^{n-1}(\phi) u(\phi,  \theta, r)^p\, dr \, d \phi\, d\theta\leq C_6\|u\|^p_{H^1(\Omega)},
\]
from which the desired result follows.\\

\noindent  3.  Proof follows by the same argument as in  part 2.\\
\noindent  4.  Proof follows by the same argument as in the proof of Theorem \ref{embed-mrev}.

\end{proof}

\begin{remark}  It is worth noting that $p_i(m,n,l)$ for $i=1,2,3$ in Theorem \ref{compact-m} give an improved embedding beyond the standard  Sobolev embeddings.  In fact,  we have the following,
\begin{itemize}
    \item $p_1(m,n,l)> \frac{2N}{N-2}$ provided $m,n,l>1.$   
    \item  $p_2(m,n,l)\geq p_1(m,n,l),$ provided $ m \ge n$.   
    \item $p_3(m,n,l)>\frac{2N}{N-2}$ if and only if $1<l<N-2$.    Also, $p_3(m,n,l)\geq p_i(m,n,l)$ for $i=1,2$ provided  $n,m>1$.  \\
    Moreover, if $N$ is odd then  $p_3(m,n,l)$ is maximized (here $N$ is fixed and we are varying $m,n,l$) when  $l=(N-1)/2$  with the value
    \[p_3(m,n,\frac{N-1}{2})=\frac{2(N+3)}{N-1}\]
   and note that 
    \[\frac{2(N+3)}{N-1}> \frac{2N}{N-2} \text{  if and only if   } N>3.\]
\end{itemize}
\end{remark}

As before we consider the following linear problem given by 
  \begin{equation} \label{linear_tau}
\left\{\begin{array}{ll}
-\Delta v = a(x) u^{p-1} &  \mbox{ in } \Omega, \\
v= 0 &   \mbox{ on } \pOm.
\end {array}\right.
\end{equation}

 \begin{thm} \label{pointwise_tau} (Pointwise invariance property) 
 
 \begin{enumerate}
 
     \item Suppose $m=n \ge 1$,  $\Omega$ satisfies the $K_-$ domain assumptions from Definition \ref{triple-def} and $a_\phi \le 0$ in $\widetilde{\Omega}_0$. If $ u \in K_-$ and $v$ satisfies (\ref{linear_tau})  then $ v \in K_-$.  
     
     \item Suppose $m=n \ge 1$,  $\Omega$ satisfies the $K_+$ domain assumptions from Definition \ref{triple-def} 
     and $a_\phi \ge  0$ in $\widetilde{\Omega}_0$. If $ u \in K_+$ and $v$ satisfies (\ref{linear_tau})  then $ v \in K_+$. 
     
     \item  Suppose $\Omega$ satisfies the $K_{-,\frac{\pi}{2}}$ domain assumptions from Definition \ref{triple-def} and $a_\phi \le  0$ in $\widetilde{\Omega}$.   If $ u \in K_{-,\frac{\pi}{2}}$ and $v$ satisfies (\ref{linear_tau})  then $ v \in K_{-,\frac{\pi}{2}}$.

      \end{enumerate}
 \end{thm}
 
 \begin{remark}One can surely remove the $n \ge 2$ restriction but when proving $w=0$ in $ \Omega$ one needs to try a bit harder when choosing a suitable cut off (here we would have $ dim(\Gamma)=N-2$ and not strictly less that $N-2$). 
 \end{remark}

\noindent 
\textbf{Proof of Theorem \ref{pointwise_tau}.} Parts 1,2: We begin by taking $ u \in K_\pm$ since much of the proof is the same for either case and as before we consider $u_k=\min \{ u(x),k \}$ where $k$ is a large integer and note $ u_k \in K_\pm$.   Let $v^k$ denote a solution of (\ref{linear_tau}) with $u$ replaced with $u_k$ and then note by elliptic regularity we have $v^k \in H^1_{0,G}(\Omega) \cap  C^{1,\delta}(\overline{\Omega})$ for all $ 0<\delta<1$.   Now note we can write 
 \[ \Delta v^k(x) = v^k_{ss}+ v^k_{tt} +v^k_{\tau \tau}+ \frac{(m-1) v^k_s}{s}+ \frac{(n-1) v^k_t}{t}  + \frac{(l-1) v^k_\tau}{\tau},\] and a computation shows that 
  

 \begin{eqnarray*}
\frac{v^k_s}{s}+ \frac{v^k_t}{t}  &=& \frac{2 v^k_r}{r} + \frac{2  v^k_\theta}{r^2 \tan(\theta)}  + \frac{ v^k_\phi}{r^2 \sin^2(\theta)} \left( \frac{1}{\tan(\phi)}- \tan(\phi) \right).
 \end{eqnarray*} and 
 \[ \frac{v_\tau^k}{\tau}= \frac{v_r^k}{r} - \frac{ \tan(\theta) v_\theta^k}{r^2}.\] 
  From this we see that the equation for $v^k$ in spherical coordinates is given by 
 $ L(v^k) = a u_k^{p-1}$  where 
 \begin{eqnarray*}
 L(v) &=& \left\{-v_{rr}- \frac{(2n+l-1) v_r}{r} - \frac{ v_{\phi \phi}}{r^2 \sin^2(\theta)}- \frac{v_{\theta \theta}}{r^2} - \frac{v_\theta}{r^2 } \left( \frac{2n-1}{\tan(\theta)}- (l-1) \tan(\theta) \right) \right\}  \\
 &&+ \frac{(n-1) v_\phi}{r^2 \sin^2(\theta)} h(\phi), \\
 &=& L_0(v) + \frac{(n-1) v_\phi}{r^2 \sin^2(\theta)} h(\phi),
 \end{eqnarray*}
 where $ h(\phi)= \tan(\phi)- \frac{1}{\tan(\phi)}$. 
  Note from this we see that 
 \begin{equation} \label{der_pp}
 \partial_\phi L(v)= L(v_\phi) + \frac{(n-1) v_\phi}{r^2 \sin^2(\theta)} h'(\phi).
 \end{equation}

 We now show that $v^k$ has the desired symmetry across $ \phi=\frac{\pi}{4}$.  Note we have 
 \[L(v^k)(\phi,\theta,r) = a(\phi, \theta,r) u_k(\phi,\theta,r)^{p-1}=:g(\phi,\theta,r) \quad \mbox{ in } \widetilde{\Omega},\] with suitable boundary conditions.  Define 
 \[ \widehat{v}(\phi,\theta,r)= v^k\left( \frac{\pi}{2}-\phi,\theta,r\right),\]  and hence our goal is to show that $ \widehat{v}=v$ which would prove $ v^k$ is even in $ \phi$ across $ \phi=\frac{\pi}{4}$.  A 
  computation shows 
 \[ L(\widehat{v})(\phi,\theta,r)=L_0(v^k)( \frac{\pi}{2} - \phi,\theta,r) + \frac{(n-1) (-1)v_\phi(\frac{\pi}{2}-\phi,\theta,r)}{r^2 \sin^2(\theta)} h(\phi)\]but noting that $h$ is odd across $ \phi=\frac{\pi}{4}$, ie. $ -h(\phi)= h( \frac{\pi}{2}- \phi)$, we have 
 \[ L( \widehat{v})( \phi,\theta,r)= L(v^k)( \frac{\pi}{2}-\phi,\theta,r)= g( \frac{\pi}{2}-\phi,\theta,r)=g(\phi,\theta,r)\] after noting that $g$ is even across $ \phi=\frac{\pi}{4}$ since both $a$ and $u_k$ are.   Hence we see that $L(\widehat{v})(\phi,\theta,r)=L(v^k)(\phi,\theta,r)$ in $ \widetilde{\Omega}$.    We now discuss the boundary conditions for $ v^k$ (and $ \widehat{v}$) in some detail.  This will be more needed later when we examine the monotonicity of $v^k$.    
  Define
 \begin{equation} \label{bc_1}
 \Gamma_{\phi=0}=\left\{ (\phi=0, \theta,r): g^1(0,\theta)<r<g^2(0,\theta), 0<\theta< \frac{\pi}{2} \right\}  \mbox{ and similarly  } \Gamma_{\phi = \frac{\pi}{2}}, 
 \end{equation}
 
 \begin{equation} \label{bc_2}
 \Gamma_{r=g^1}=\left\{ (\phi, \theta,g^1(\phi,\theta)): 0<\phi<\frac{\pi}{2}, 0<\theta<\frac{\pi}{2} \right\}  \mbox{ and similarly  } \Gamma_{r=g^2}, 
 \end{equation}
   \begin{equation} \label{bc_3}
 \Gamma_{\theta=0}=\left\{ (\phi, \theta=0,r): g^1(\phi,0)<r<g^2(\phi,0), 0<\phi<\frac{\pi}{2}  \right\}  \mbox{ and similarly  } \Gamma_{\theta= \frac{\pi}{2}}. 
 \end{equation}  
 
 First note that $ v^k, \widehat{v}$ are both zero on $ \Gamma_{r=g^i}$ for $ i=1,2$ (to see the result for $\widehat{v}$ use the fact that $g^i$ is even across $ \phi=\frac{\pi}{4}$.   By the smoothness of $v^k$ (and hence $\widehat{v}$) (and since the functions are even across $ \phi=0$ and $ \phi=\frac{\pi}{2}$) we have $v^k_\phi= \widehat{v}_\phi=0$ on $ \Gamma_{\phi=0}$ and $ \Gamma_{\phi=\frac{\pi}{2}}$.   By smoothness and symmetry we also get $ v^k_\theta= \widehat{v}_\theta=0$ on $\Gamma_{\theta= \frac{\pi}{2}}$.    Note $\Gamma_{\theta=0}$  corresponds to a portion of the positive $ \tau$ axis.   Set $ w(\phi, \theta,r)= v^k(\phi, \theta,r) - \widehat{v}(\phi, \theta,r)$ defined on $ \widetilde{\Omega}$.   Also note we have $ L(w)(\phi,\theta,r) =0$ for $ (\phi,\theta,r) \in \widetilde{\Omega}$ with $ w=0$ on $ \Gamma_{r=g^i}$ for $ i=1,2$;   $ w_\phi=0$ on $ \Gamma_{\phi=0} \cup \Gamma_{\phi=\frac{\pi}{2}}$ and $ w_\theta=0$ on $ \Gamma_{\theta= \frac{\pi}{2}}$.     Set $ \Gamma=\{x \in \Omega:  s=t=0 \}$ and note that $ dim(\Gamma) = l= N-2n \le N-2$.  Also note in terms of $x$ we have 
 \[ \Delta w(x)=0  \quad \mbox{ in } \Omega \backslash \Gamma,\]  with $ w=0$ on $ \pOm$.    We now claim that since $ w \in C^{1,\alpha}(\overline{\Omega}) \cap C^\infty( \overline{\Omega} \backslash \Gamma)$ and since $ dim(\Gamma) \le N-2$ we have $ \Delta w=0$ in $ \Omega$ in sense of distributions and then we can apply the maximum principle to see $ w=0$ in $\Omega$.  
 
 We now prove the claim.  Take a smooth function $ g$ on $ \IR$ with $ g(t)=0$ for $ t \le 1$ and $ g(t)=1$ for $ t \ge 2$  and consider $\delta_\Gamma(x)=dist(x,\Gamma)$ (the Euclidean distance) and fix $ x_0 \in \Gamma$ but not an endpoint since the endpoints lie on $\pOm$.  Note that $\delta_\Gamma$ is smooth near $ x_0$ and we now set 
 \[ \gamma_\E(x) = g \left(  \frac{\delta_\Gamma(x)}{\E} \right),\] and note $g_\E$ is smooth near $ x_0$.  Let $ \psi $ be smooth and compactly supported near $x_0$ and note a    
 computation shows that 
 \begin{eqnarray*}
 \Big| \int_\Omega \gamma_\E w \Delta \psi dx \Big| & = &\Big| \nabla \gamma_\E \cdot \left\{ \nabla w \psi - w \nabla \psi \right\} dx \Big| \\
 & \le & C \int_\Omega | \nabla \gamma_\E(x)|dx
 \end{eqnarray*} where $C$ independent of $ \E$ for small $\E$.  We now claim the right hand side converges to zero and hence we'd have 
 $ \int_\Omega w \Delta \psi dx=0$  which shows that $ \Delta w=0$ in $ \Omega$ in the sense of distributions.   We can now use  Hausdorff measure to prove the result but we prefer to use the box counting dimension, see \cite{hardy_cowan}  for instance.  Note that we have 
 \[ N-2 \ge dim_{box}(\Gamma):=N- \lim_{t \searrow 0} \frac{ \log(| \Gamma_t|)}{ \log(t)},\] where $|\Gamma_t|$ is the $N$ dimensional measure of $ \Gamma_t=\{x \in \Omega: \delta_\Gamma(x)<t \}$.   So there is some $\alpha(t) \rightarrow 0$ as $ t \searrow 0$ such that $|\Gamma_t| \le t^{\alpha(t)+2}$.   Then note we have 
 
 \begin{eqnarray*}
 \int_\Omega | \nabla \gamma_\E(x)| dx & \le & C \int_{\E<\delta_\Gamma<2\E}  \frac{1}{\E} dx \\ 
 & \le & C  \frac{| \Gamma_{2 \E}|}{\E}  \\ 
 & \le & C \frac{(2\E)^{2+ \alpha(2\E)}}{\E} \rightarrow 0,
 \end{eqnarray*} which proves the claim. \\

 \noindent 
 \textbf{Monotonicity.} We now show that $v^k$ has the desired monotonicity in $ \phi$ on $ \widetilde{\Omega}_0$.   Note that by (\ref{der_pp}) we see 
 \begin{equation} \label{shiii}
      L(v^k_\phi) + \frac{(n-1) h'(\phi) v_\phi^k   }{r^2 \sin^2(\theta)} = \partial_\phi (a u_k^{p-1}) \quad \mbox{ in }  \widetilde{\Omega}_0,
       \end{equation}
      and note $h'(\phi) \ge 0$ and hence there is hope for a maximum principle for the operator on the left acting on $ v^k_\phi$.
 We now define the boundaries and note we are really taking the boundaries from above and suitably adjusting them to $ 0<\phi<\frac{\pi}{4}$ instead of $ 0<\phi<\frac{\pi}{2}$. So we have 
 \begin{equation} \label{bc_10}
 \Gamma_{\phi=0}^0=\left\{ (\phi=0, \theta,r): g^1(0,\theta)<r<g^2(0,\theta), 0<\theta< \frac{\pi}{2} \right\}  \mbox{ and similarly  } \Gamma_{\phi = \frac{\pi}{4}}^0, 
 \end{equation}
 
 \begin{equation} \label{bc_20}
 \Gamma_{r=g^1}^0=\left\{ (\phi, \theta,g^1(\phi,\theta)): 0<\phi<\frac{\pi}{4}, 0<\theta<\frac{\pi}{2} \right\}  \mbox{ and similarly  } \Gamma_{r=g^2}^0, 
 \end{equation}
   \begin{equation} \label{bc_30}
 \Gamma_{\theta=0}^0=\left\{ (\phi, \theta=0,r): g^1(\phi,0)<r<g^2(\phi,0), 0<\phi<\frac{\pi}{4}  \right\}  \mbox{ and similarly  } \Gamma_{\theta= \frac{\pi}{2}}^0, 
 \end{equation}  \\ 
 
 
 \noindent
 \emph{Boundary terms $\Gamma_{\phi=0}^0 \cup \Gamma_{\phi=\frac{\pi}{4}}^0$.}
 Note by the smoothness and symmetry of $v^k$ we have $ v^k_\phi=0$ on $ \Gamma_{\phi=0}^0 \cup \Gamma_{\phi=\frac{\pi}{4}}^0$.    \\ 
 
  \noindent
 \emph{Boundary terms $ \Gamma_{r=g^1}^0 \cup \Gamma_{r=g^2}^0$.}  The boundary conditions here depend on with case of $K_\pm$ we are in.  First consider the case of $K_+$.   In this case because $v^k=0$ on $\Gamma_{r=g^1}^0 \cup \Gamma_{r=g^2}^0$ and $g^i$ is constant in $ \phi$ we see that $ v^k_\phi=0$ on  $\Gamma_{r=g^1}^0 \cup \Gamma_{r=g^2}^0$.
   We now suppose we in the case of $K_-$.  In this case we are either in the case of a annulus or a more general domain with suitable monotonicity.  In the case of a annulus we have $v^k_\phi=0$ on $\Gamma_{r=g^1}^0 \cup \Gamma_{r=g^2}^0$ as in the case of $K_+$.  Using the fact that $v^k \ge 0$ in $ \Omega$ with $ v^k=0$ on $ \pOm$ and the monotonicity of the maps $ \phi \mapsto g^i(\phi,\theta)$ we see that $ v^k_\phi \le 0$ on $ \Gamma_{r=g^1}^0 \cup \Gamma_{r=g^2}^0$. \\ 
 
 \noindent 
 \emph{Boundary terms  $\Gamma_{\theta=0}^0 \cup \Gamma_{\theta= \frac{\pi}{2}}^0$.} First we consider $ \Gamma_{ \theta= \frac{\pi}{2}}^0$.   Note by the smoothness of $ v^k$ we have $ v^k_\theta=0$ on $ \Gamma_{ \theta= \frac{\pi}{2}}^0$ and hence we have $ 0= (v^k_\theta)_\psi = (v^k_\psi)_\theta$ on $\Gamma_{ \theta= \frac{\pi}{2}}^0$.   We now examine the term  $ \Gamma_{\theta=0}^0$.  Note that we can write $ \nabla v^k(x)$ as  
 \begin{equation} \label{grad_form}
 \nabla v^k(x) = v_r^k \hat{r} + \frac{ v_\theta^k}{r} \hat{\theta} + \frac{v^k_\phi}{r \sin(\theta)} \hat{\phi},
 \end{equation} where $ (\hat{\phi},\hat{\theta},\hat{r})$ are the unit vectors in spherical coordinates. From this we see that 
 \begin{equation} \label{bound_near_axis}
 | v_\phi^k| \le r \sin(\theta) | \nabla v^k(x)|.
 \end{equation}
 This shows that, at least in some limiting sense,  we have $v_\phi^k=0$ on $ \Gamma_{\theta=0}^0$.  We can now either work in spherical coordinates or translate back to coordinates in $x$; we will choose the latter since its more familiar to apply the maximum principle.
Writing the left hand side of (\ref{shiii}) we arrive at 
\[-\Delta v_\phi^k(x) + (n-1)H(x) v_\phi^k (x)= \partial_\phi (a u_k^{p-1}), \quad \mbox{ in } \Omega_0,\] where $\Omega_0:=\{x \in \Omega:  0<\phi<\frac{\pi}{4},   x \notin \Gamma \}$ and 
\[ H(x)=  \frac{ \sum_{i=1}^{2n} x_i^2  }{  \left( \sum_{i=1}^n x_i^2  \right)\left(  \sum_{i=n+1}^{2n} x_i^2 \right)}.\] 


We now consider the case of $K_-$. Let $ \E>0$ be small and consider $ \psi=(v_\phi^k-\E)_+$ and note $ \psi=0$ near $\Gamma$  after considering (\ref{bound_near_axis}) and also $ \psi=0$ near the portions of the boundary of $ \Omega_0$ corresponding to $ \phi=0$ and $ \phi=\frac{\pi}{4}$ (but we really will only need the result for $ \phi=0$ since $H$ is not singular at $ \phi=\frac{\pi}{4}$).     Note that  $ \psi=0$ near $ \partial \Omega_0$.  From this we have 
\[ \int_{\Omega_0} \nabla v^k_\phi \cdot \nabla \psi dx + (n-1) \int_{\Omega_0} H v_\phi^k \psi dx = \int_{\Omega_0} \partial_\phi (a u_k^{p-1}) \psi dx \le 0,\] after noting the assumptions on $a$ and $u$.   From this one sees that 
\[ \int_{\Omega_0} | \nabla \psi|^2 dx + (n-1) \int_{\Omega_0} H \psi^2 dx \le 0,\] and hence we have $ \psi=0$ in $ \Omega_0$ and hence we have $ v_\phi^k \le \E$ a.e. in $ \Omega_0$ and hence we have the desired result after noting $ \E>0$ is arbitrary.  \\ 

We now consider the case of $K_+$.  Consider $ \psi=(v_\phi^k+\E)_-$ where $ \E>0$ is small.   Then note we have $ \psi=0$ near $ \pOm_0$. As above we get 
\[ \int_{\Omega_0} \nabla v^k_\phi \cdot \nabla \psi dx + (n-1) \int_{\Omega_0} H v_\phi^k \psi dx = \int_{\Omega_0} \partial_\phi (a u_k^{p-1}) \psi dx \ge  0,\] after noting the assumptions on $a$ and $u$.  From this we can argue that 
\[ \int_{\Omega_0} | \nabla \psi|^2 dx + (n-1) \int_{\Omega_0} H \psi^2 dx \le 0,\] and hence $ \psi=0$ which gives $ v^k_\phi \ge -\E$ and hence we get the desired result.     We now need to pass to the limit in $k$,  but this follows from similiar arguments that we used in previous sections.  \\

\noindent 
3.  The proof for this part follows from similar type computations as in \cite{orgin_an} and some of the ideas used in part 1 and 2 of the previous proof to deal with the extra variable $\tau$, we omit the details. 
\hfill $\Box$  \\

 \noindent
 \textbf{Proof of Theorem \ref{triple-thm}.}     Once again, we are going to use Theorem  \ref{var-pri} for the proof. Note that
conditions (i) and (ii) in Theorem  \ref{var-pri} follows from Theorems \ref{compact-m} and \ref {pointwise_tau} respectively.
This proves the existence of a non-negative weak solution u of (\ref{eq_triple}).     To prove the solution is positive and  regular we use the same arguments we have used in the previous sections, we omit the details. 
 
  \hfill $\Box$

\subsection{Nonsymmetric solutions on domains of triple revolution} 

In this section we examine the case where the domain, the equation and $a$ have added symmetry and we examine the existence of solutions which do not inherit the same symmetry.    We also recall the definition of the best constant in Hardy inequality for the domain $\Omega$,  that is,
\begin{equation} \label{hardyy50}
\beta_0(\Omega)=\inf_{u\in H_0^1(\Omega)}
\frac{\int_\Omega |\nabla u|^2 \, dx}{\int_\Omega \frac{u^2}{|x|^2} \, dx}. \end{equation}

We first consider the case of radial symmetry and then we consider the case of cylindrical symmetry around the $\tau$ axis in the case of the variables $ (s,t,\tau)$.  

\subsubsection{The case of the annulus} 

Here we examine the case of $a(x)=a(|x|)$ is radial,  and $\Omega$ is the annulus  $\Omega=\{x:\,  R_1\leq  |x|<R_2\}$.

\begin{thm}   \label{annulus_triple_dep} 
 Let $u$ be the solution obtained in either parts of Theorem \ref{triple-thm}.  If $\beta_0$ is large enough then  $u$ depends on all three variables in a non-trivial way. 
\end{thm} 

\begin{proof}
We just do the proof for part 3 of Theorem \ref{triple-thm}. Other cases follows by the same argument.

Define 
\[w_l(\theta)= \sin^{N-l-1} (\theta)\cos^{l-1}(\theta), \qquad  w_{m,n}(\phi)=\cos^{m-1}(\phi) \sin^{n-1}(\phi).\]
  Consider the variational formulation of  eigenvalue problems given by
\begin{equation} \label{var_mu3}
    \mu_l=\inf_{\psi \in H^1_{loc}(0, \frac{\pi}{2})}\Big \{ \int_0^{\frac{\pi}{2}} |\psi'(\theta)|^2 w_l(\theta) \,d\theta; \quad    \int_0^{\frac{\pi}{2}} |\psi(\theta)|^2w_l(\theta) \,d\theta=1, \int_0^{\frac{\pi}{2}} \psi(\theta)w_l(\theta) \,d\theta=0     \Big\},
\end{equation}
and
\begin{equation} \label{var_mu31}
    \mu_{m,n}=\inf_{\psi \in H^1_{loc}(0, \frac{\pi}{2})}\Big \{ \int_0^{\frac{\pi}{2}} |\psi'(\phi)|^2 w_{m,n}(\phi) \,d\phi; \quad    \int_0^{\frac{\pi}{2}} |\psi(\phi)|^2 w_{m,n}(\phi) \,d\phi=1, \int_0^{\frac{\pi}{2}} \psi(\phi) w_{m,n}(\phi)\,d\phi=0     \Big\}.
\end{equation}
Let $\psi_l$ be the unique minimizer  in (\ref{var_mu3}), and $\psi_{m,n}$ be the unique minimizer in (\ref{var_mu31}).
Let $E$ be the formal Euler-Lagrange functional of (\ref{eq_triple}).\\
Let $u$ be the solution obtained in part 3 of  Theorem \ref{triple-thm}.    We divide the proof into two cases.  We first show that $u$ depends on $\theta$ in a nontrivial way provided
\[\frac{(p-1)\mu_l}{p-2}<\beta_0.\]
Then hen we show that $u$ depends on $\theta$ in a non-trivial way provided 
\[ \frac{(p-1)\mu_{m,n}}{p-2}<\beta_0.\]

{\it Case I.}  We proceed by way of contradiction.  Let us assume that $u$ is not a  function of $\theta.$
 
Set $v(r,\phi, \theta)=u(r, \phi) \psi_l(\theta)$.  We  just need to show that 
\begin{equation}\label{qw3}\langle E^{''}(u); v,v\rangle :=\int_\Omega |\nabla v|^2 \, dx-(p-1)\int_\Omega|a(|x|)u|^{p-2}v^2\,dx <0.\end{equation}

Note first that $u=u(r,\phi)$  satisfies the equation $-\Delta u=a(r)u^{p-1}$.  Multiplying both sides of the equation by $u(r,\phi)\psi^2_l(\theta)$ and integrating in spherical coordinates imply that

\begin{eqnarray}\label{qw13}
\int_{R_1}^{R_2}\int_{0}^{\frac{\pi}{2}}\int_{0}^{\frac{\pi}{2}}  ( u_r^2+\frac{u_\phi^2}{r^2 \sin^2(\theta)}) \psi_l^2(\theta) d \mu( \phi,\theta,r)=\int_{R_1}^{R_2}\int_{0}^{\frac{\pi}{2}} \int_{0}^{\frac{\pi}{2}}a(r)u^p \psi_l^2(\theta) d \mu( \phi,\theta,r).
\end{eqnarray}

It also follows  from the definition of $\beta_0=\beta_0(\Omega)$,  the best constant in Hardy inequality (\ref{hardyy50}) for the function $v=u(r,\phi) \psi_l(\theta)$ that 
\begin{equation}\label{hardyy11}
\int_{R_1}^{R_2}\int_{0}^{\frac{\pi}{2}}\int_{0}^{\frac{\pi}{2}} \left (u_r^2\psi_l^2(\theta)+ \frac{u_\phi^2\psi_l^2(\theta)}{r^2 \sin^2(\theta)}+\frac{u^2 \psi'^2_l}{r^2}\right )  d \mu( \phi,\theta,r) \geq \beta_0   \int_{R_1}^{R_2}\int_{0}^{\frac{\pi}{2}}\int_{0}^{\frac{\pi}{2}}  \frac{u^2\psi_l^2(\theta)}{r^2}  d \mu( \phi,\theta,r). \end{equation}
It now follows that   
\begin{eqnarray*}
\langle E^{''}(u); v,v\rangle &=& \int_{R_1}^{R_2} \int_0^{\frac{\pi}{2}}\int_0^{\frac{\pi}{2}}\left (u_r^2\psi_l^2+ \frac{u_\phi^2\psi_l^2}{r^2 \sin^2(\theta)}+\frac{u^2 \psi'^2_l}{r^2}-(p-1)a(r)u^{p}\psi_l^2  \right )  d \mu( \phi,\theta,r)\\
 &=& \int_{R_1}^{R_2} \int_0^{\frac{\pi}{2}}\int_0^{\frac{\pi}{2}}\frac{u^2\psi_l'^2}{r^2}d \mu( \phi,\theta,r)-(p-2)\int_{R_1}^{R_2} \int_0^{\frac{\pi}{2}}\int_0^{\frac{\pi}{2}}\left (u_r^2\psi_l^2+ \frac{u_\phi^2\psi_l^2}{r^2 \sin^2(\theta)}\right) d \mu( \phi,\theta,r)\\
 &=& (p-1)\int_{R_1}^{R_2} \int_0^{\frac{\pi}{2}}\int_0^{\frac{\pi}{2}}\frac{u^2\psi_l'^2}{r^2}d \mu-(p-2)\int_{R_1}^{R_2} \int_0^{\frac{\pi}{2}}\int_0^{\frac{\pi}{2}}\left (u_r^2\psi_l^2+ \frac{u_\phi^2\psi_l^2}{r^2 \sin^2(\theta)}+\frac{u^2 \psi'^2_l}{r^2}\right) d \mu\\
 &=& (p-1)\mu_l\int_{R_1}^{R_2} \int_0^{\frac{\pi}{2}}\int_0^{\frac{\pi}{2}}\frac{u^2\psi_l^2}{r^2}d \mu-(p-2)\int_{R_1}^{R_2} \int_0^{\frac{\pi}{2}}\int_0^{\frac{\pi}{2}}\left (u_r^2\psi_l^2+ \frac{u_\phi^2\psi_l^2}{r^2 \sin^2(\theta)}+\frac{u^2 \psi'^2_l}{r^2}\right) d \mu\\
 & \leq & \left ( \frac{(p-1)\mu_l}{\beta_0}-(p-2)\right )\int_{R_1}^{R_2} \int_0^{\frac{\pi}{2}}\int_0^{\frac{\pi}{2}}\left (u_r^2\psi_l^2+ \frac{u_\phi^2\psi_l^2}{r^2 \sin^2(\theta)}+\frac{u^2 \psi'^2_l}{r^2}\right) d \mu<0
 \end{eqnarray*}

{\it Case II.} Similar to the previous case,  we proceed by way of contradiction.  Let us assume that $u$ is not a  function of $\phi.$
 Set $v(r,\phi, \theta)=u(r, \theta) \psi_{m,n}(\phi)$. To conclude the proof we  show that 
\begin{equation}\langle E^{''}(u);v,v\rangle :=\int_\Omega |\nabla v|^2 \, dx-(p-1)\int_\Omega|a(|x|)u|^{p-2}v^2\,dx <0.\end{equation}

Note first that $u=u(r,\theta)$  satisfies the equation $-\Delta u=a(r)u^{p-1}$.  Multiplying both sides of the equation by $u(r,\theta)\psi^2_{m,n}(\phi)$ and integrating in spherical coordinates imply that

\begin{eqnarray}\label{qw132}
\int_{R_1}^{R_2}\int_{0}^{\frac{\pi}{2}}\int_{0}^{\frac{\pi}{2}}  ( u_r^2+\frac{u_\theta^2}{r^2 }) \psi_{m,n}^2(\phi) d \mu( \phi,\theta,r)=\int_{R_1}^{R_2}\int_{0}^{\frac{\pi}{2}} \int_{0}^{\frac{\pi}{2}}a(r)u^p \psi_{m,n}^2(\phi) d \mu( \phi,\theta,r).
\end{eqnarray}

It also follows  from the definition of $\beta_0=\beta_0(\Omega)$,  the best constant in Hardy inequality for the function $v=u(r,\theta) \psi_{m,n}(\phi)$ that 
\begin{equation}\label{hardyy22}
\int_{R_1}^{R_2}\int_{0}^{\frac{\pi}{2}}\int_{0}^{\frac{\pi}{2}} \left (u_r^2\psi_{m,n}^2(\phi)+ \frac{u^2\psi'^2_{m,n}(\phi)}{r^2 \sin^2(\theta)}+\frac{u_\theta^2 \psi^2_{m,n}}{r^2}\right )  d \mu( \phi,\theta,r) \geq \beta_0   \int_{R_1}^{R_2}\int_{0}^{\frac{\pi}{2}}\int_{0}^{\frac{\pi}{2}}  \frac{u^2\psi_{m,n}^2(\phi)}{r^2}  d \mu( \phi,\theta,r). \end{equation}
As in the proof of case one can deduce that   
\begin{eqnarray*}
\langle E^{''}(u); v,v\rangle \leq  \left ( \frac{(p-1)\mu_{m,n}}{\beta_0}-(p-2)\right )\int_{R_1}^{R_2}\int_{0}^{\frac{\pi}{2}}\int_{0}^{\frac{\pi}{2}} \left (u_r^2\psi_{m,n}^2(\phi)+ \frac{u^2\psi'^2_{m,n}(\phi)}{r^2 \sin^2(\theta)}+\frac{u_\theta^2 \psi^2_{m,n}}{r^2}\right )  d \mu( \phi,\theta,r) <0
 \end{eqnarray*}
\end{proof}
 
 We recall the following result from \cite{orgin_an}  about the largeness of the best constant $\beta_0$ in the hardy inequality where the domain is an annulus. 
 
 \begin{prop}\label{annul_resu}  \cite{orgin_an}  

\begin{itemize}
\item Let  $ R_1=R$ and $ R_2=R+1.$  Then $\beta_0$  is sufficiently large for large values of  $R$.   

\item Let $ R<\gamma(R)$ with $ \frac{\gamma(R)}{R} \rightarrow 1$ as $ R \rightarrow \infty$.   With $ \Omega_R=\{x \in \IR^N: R<|x|<\gamma(R) \}$ then for large enough $R$  the  $\beta_0$  corresponding to $ \Omega_R$ is sufficiently large.
\end{itemize}
\end{prop}
 
 \begin{coro} \label{mult_ann} 
Let $p>2$ and $N>3.$ Consider the problem (\ref{eq_triple}) where $\Omega=\{x \in \R^N:\,\, R< |x|<R+1\}$ and $a\equiv1.$ 
 For large values of $R,$ there are at least
 \[ \Bigl\lfloor\frac{N}{2} \Bigr\rfloor+\Bigl\lfloor\frac{N-1}{2} \Bigr\rfloor+\Bigl\lfloor\frac{N-2}{2} \Bigr\rfloor+...+\Bigl\lfloor\frac{N-k}{2} \Bigr\rfloor, \qquad k=\Bigl\lfloor\frac{N}{3} \Bigr\rfloor, \]
 
 positive non radial solutions. Here $\Bigl\lfloor z \Bigr\rfloor$  stands for the  floor of $z \in \R.$ 
 \end{coro}  
 \begin{proof} Here we are going to use the $K_{-,\frac{\pi}{2}}$ symmetry in Theorem \ref{annulus_triple_dep}  and therefore $m$ and $n$ can be different.
 The cardinality of the set 
 \[D_2=\left\{(m,n) \in \N \times \N;\,\,   m+n=\N,\, 1\leq m\leq n\right\}
\] 
is $\Bigl\lfloor\frac{N}{2} \Bigr\rfloor
$, and for each $(m,n)\in D_2$  there exists  a no-radial solution $u$ which is invariant in $O(m)\times O(n)$ when $R$ is large enough as we have shown in \cite{orgin_an}.
Also,  the cardinality of the set 
 \[D_3=\{(m,n,l) \in \N \times \N \times \N;\,\,   m+n+l=\N,\, 1\leq m\leq n\leq l\}
\] 
is\[ \Bigl\lfloor\frac{N-1}{2} \Bigr\rfloor+\Bigl\lfloor\frac{N-2}{2} \Bigr\rfloor+...+\Bigl\lfloor\frac{N-k}{2} \Bigr\rfloor, \qquad k=\Bigl\lfloor\frac{N}{3} \Bigr\rfloor.\] By Theorem \ref{annulus_triple_dep}, for each $(m,n,l)\in D_3$  there exists  a  solution $u$ which is invariant in $O(m)\times O(n) \times O(l)$ and it is non invariant in $O(j)\times O(N-j)$ for any $j \in \{1,...,N-1\}.$ This completes the proof.
\end{proof}

\subsubsection{The case of symmetry in $\phi$} 

In this section we examine the case where the domain and $a$ have symmetry in $ \phi$.   In terms of the coordinates $ (s,t,\tau)$ we are examining the case where we have cylindrical symmetry around the $ \tau$ axis.     We suppose $m=n$ and $\Omega$ satisfies assumption 2 from Definition \ref{triple-def}, ie. suppose $ g^i=g^i( \phi, \theta)$ is smooth and positive on $ [0, \frac{\pi}{2}] \times [0, \pi/2]$ and for each fixed $ \theta \in (0,\pi/2)$ and $i=1,2$  we have   $ \phi \mapsto g^i(\phi,\theta)$ is constant on $(0, \pi/2)$.  

We further assume that $ a=a(r,\theta)$.  Then looking at (\ref{eq_triple}) (written in terms of $ (r,\phi,\theta)$) one sees that it is reasonable to look for solutions of (\ref{eq_triple}) which don't depend on $ \phi$ and in fact one can use the same imbedding to obain a solution for the given range of paramters that doesn't depend on $ \phi$.    Our next theorem gives sufficient conditions under which the ground state solution depends on $ \phi$ in a nontrivial way.

\begin{thm} \label{phi-dep}  Suppose $\Omega$ satisfies the above hypothesis and $ a_\phi=0$.  

\begin{enumerate}
    \item Suppose $p$ satisfies hypothesis from Theorem \ref{triple-thm} part 2  and $u$ is $K_+$ ground state solution promised by Theorem \ref{triple-thm} part 2. If $ \beta_0$ is large enough then $u$ is a function that depends on $\phi$ in a nontrivial way.
    
    \item Suppose $p$ satisfies hypothesis from Theorem \ref{triple-thm} part 1  and $u$ is $K_-$ ground state solution promised by Theorem \ref{triple-thm} part 1. If $ \beta_0$ is large enough then $u$ is a function that depends on $\phi$ in a nontrivial way.
    \end{enumerate}
    
\end{thm}

\begin{proof} The proof follows the same strategy as the proof of Theorem \ref{annulus_triple_dep}. 
\end{proof}

\begin{remark} One can examine multiplicity type results for these domains also, we leave this to the interested reader. 
\end{remark}


\end{document}